\documentclass[oneside,article]{memoir}
\usepackage{amsmath}         % align-like environments
\usepackage{amsthm,thmtools,thm-restate} % theorem-like environments
\usepackage{enumitem}        % enumerate-like environments
\usepackage{mathtools}

\usepackage{tikz}
  \usetikzlibrary{matrix,arrows,positioning,decorations.markings,backgrounds}
  
\usepackage{tikz-cd}

\usepackage{float} %to force figure into right place
\usepackage{caption, subcaption} % for subfigures

\usepackage[nobysame,alphabetic,initials]{amsrefs} % bibliography tools
\usepackage[hidelinks]{hyperref}                   % hyperlinks
\usepackage[capitalize,nosort]{cleveref}           % named references

\hypersetup{linktoc=all} % Adds hyperlinks to Table of Contents

\usepackage{amssymb}
\usepackage[mathscr]{euscript}
\usepackage{relsize}
\usepackage{stmaryrd}
\usepackage{stackengine}
\usepackage{mathdots} % for inverse diagonal dots

\usepackage{indentfirst} % indent the first paragraph
\usepackage{multicol}

\usepackage{lmodern}
\usepackage{libertine}
\usepackage[libertine]{newtxmath}

\usepackage{booktabs}

\usepackage{arydshln} % Dashed lines in tabular environments

\usepackage{algorithm}
\usepackage{algpseudocode}
\usepackage{algorithmicx}

\usepackage[inline,nomargin]{fixme} % fixme notes
  \fxsetup{targetlayout=color}

\usepackage{float} %to force figure into right place
\usepackage{wrapfig} % to wrap figures around text
\usepackage{stackrel} % for stacking arrows

\isopage[12]

\setlrmargins{*}{*}{1}
\checkandfixthelayout

\counterwithout{section}{chapter}
\usepackage{appendix}

\newcommand{\cSet}{\mathsf{cSet}}
\newcommand{\Graph}{\mathsf{Graph}}
\newcommand{\Set}{\mathsf{Set}}
\newcommand{\Kan}{\mathsf{Kan}}

\newcommand{\reali}[2][]{\lvert #2 \rvert_{#1}} % realization from cSet to Graph - optional argument is for length-m realization
\newcommand{\dom}{\operatorname{dom}}

\newcommand{\face}[2]{\partial^{#1}_{#2}} % face operator
\newcommand{\degen}[2]{\sigma^{#1}_{#2}} % degeneracy operator

  \newcommand{\boxcat}{\mathord{\square}} % box category
  
  \newcommand{\conn}[2]{\gamma^{#1}_{#2}} % negative connection operator on a cubical set
  \newcommand{\cube}[1]{\mathord{\square^{#1}}} % standard n-cube
  \newcommand{\obox}[2]{\mathord{\sqcap^{#1}_{#2}}} % open box (place i and epsilon in same argument)
  \newcommand{\dfobox}[1][n]{\mathord{\sqcap^{#1}_{i,\varepsilon}}} % open box with default argument (optional argument is n)
\newcommand{\restr}[2]{{#1}|_{#2}} % restriction of a map
\newcommand{\pbtick}{\lrcorner}
\newcommand{\potick}{\ulcorner}
\newcommand{\const}[1]{\mathsf{const}_{#1}} % constant

  \DeclareFontFamily{U}{dmjhira}{}
  \DeclareFontShape{U}{dmjhira}{m}{n}{ <-> dmjhira }{}

  \DeclareRobustCommand{\yo}{\text{\usefont{U}{dmjhira}{m}{n}\symbol{"48}}}

  \newcommand{\adj}{\dashv}

  \newcommand{\we}{\mathrel\sim}
  
  \newcommand{\Ho}{\operatorname{Ho}}

  \newcommand{\op}{{\mathord\mathrm{op}}}

  \newcommand{\ob}{\operatorname{ob}}

  \newcommand{\push}{\cup}

  \newcommand{\colim}{\operatorname*{colim}}
  \newcommand{\slice}{\mathbin\downarrow}

  \newcommand{\Ex}{\operatorname{Ex}}

  \newcommand{\sd}{\operatorname{sd}} % subdivision of a simplex is written in lowercase
  \newcommand{\Sd}{\operatorname{Sd}}

  \newcommand{\Sk}{\operatorname{sk}}

  \newcommand{\bd}{\partial}

  \newcommand{\gprod}{\otimes}

  \newcommand{\id}[1][]{\operatorname{id}_{#1}}
  
  \newcommand{\cod}{\operatorname{cod}}

  \newcommand{\simp}[1]{\mathord\Delta^{#1}}

  \newcommand{\uvar}{\mathord{\relbar}}

  \renewcommand{\tilde}{\widetilde}

  \renewcommand{\hat}{\widehat}

  \newcommand{\nat}{{\mathord\mathbb{N}}}

  \newcommand{\cat}[1]{\mathscr{#1}}

  \newcommand{\ncat}[1]{\mathsf{#1}}

  \newcommand{\sSet}{\ncat{sSet}}

  \newcommand{\from}{\colon}

  \newcommand{\ito}{\hookrightarrow}

  \newcommand{\weto}{\mathrel{\ensurestackMath{\stackon[-2pt]{\xrightarrow{\makebox[.8em]{}}}{\mathsmaller{\mathsmaller\we}}}}}
  \newcommand{\weot}{\mathrel{\ensurestackMath{\stackon[-2pt]{\xleftarrow{\makebox[.8em]{}}}{\mathsmaller{\mathsmaller\we}}}}}

  \declaretheorem[style=definition,within=section]{definition}
  \declaretheorem[style=definition,numberlike=definition]{example}
  \declaretheorem[style=definition,numberlike=definition]{remark}

  \declaretheorem[style=definition,numberlike=definition]{construction}

  \declaretheorem[style=plain,numberlike=definition]{corollary}
  \declaretheorem[style=plain,numberlike=definition]{lemma}
  \declaretheorem[style=plain,numberlike=definition]{proposition}
  \declaretheorem[style=plain,numberlike=definition]{theorem}
  \declaretheorem[style=plain,numberlike=definition]{conjecture}
  
  \declaretheorem[style=plain,numbered=no,name=Theorem]{theorem*}
  \declaretheorem[style=plain,numbered=no,name=Conjecture]{conjecture*}

  \Crefname{corollary}{Corollary}{Corollaries}
  \Crefname{definition}{Definition}{Definitions}
  \Crefname{lemma}{Lemma}{Lemmas}
  \Crefname{proposition}{Proposition}{Propositions}
  \Crefname{remark}{Remark}{Remarks}
  \Crefname{theorem}{Theorem}{Theorems}
  \Crefname{conjecture}{Conjecture}{Conjectures}
  \Crefname{notation}{Notation}{Notations}
  \Crefname{convention}{Convention}{Conventions}
  \Crefname{construction}{Construction}{Constructions}

  \newlist{axioms}{enumerate}{1}
  \Crefname{axiomsi}{}{}

  \newenvironment{tikzeq*}
  {
    \begingroup
    \begin{equation*}
    \begin{tikzpicture}[baseline=(current bounding box.center)]
  }
  {
    \end{tikzpicture}
    \end{equation*}
    \endgroup
    \ignorespacesafterend
  }

  \tikzset
  {
    diagram/.style=
    {
      matrix of math nodes,
      column sep={4.3em,between origins},
      row sep={4em,between origins},
      text height=1.5ex,
      text depth=.25ex
    },
    over/.style={preaction={draw=white,-,line width=6pt}},
    every to/.style={font=\footnotesize},
    inj/.style={right hook->},
    surj/.style={-{Latex[open]}},
    cof/.style={>->},
    fib/.style={->>},
  }

  \DeclareFontFamily{U}{mathx}{\hyphenchar\font45}

  \DeclareFontShape{U}{mathx}{m}{n}{
    <5> <6> <7> <8> <9> <10>
    <10.95> <12> <14.4> <17.28> <20.74> <24.88>
    mathx10}{}

  \DeclareSymbolFont{mathx}{U}{mathx}{m}{n}

  \DeclareFontFamily{U}{mathb}{\hyphenchar\font45}

  \DeclareFontShape{U}{mathb}{m}{n}{
    <5> <6> <7> <8> <9> <10>
    <10.95> <12> <14.4> <17.28> <20.74> <24.88>
    mathb10}{}

  \DeclareSymbolFont{mathb}{U}{mathb}{m}{n}

  \DeclareMathSymbol{\Rsh}{\mathrel}{mathb}{"E9}

  \DeclareFontFamily{U}{MnSymbolA}{}

  \DeclareFontShape{U}{MnSymbolA}{m}{n}{
    <-6> MnSymbolA5
    <6-7> MnSymbolA6
    <7-8> MnSymbolA7
    <8-9> MnSymbolA8
    <9-10> MnSymbolA9
    <10-12> MnSymbolA10
    <12-> MnSymbolA12}{}

  \DeclareSymbolFont{MnSyA}{U}{MnSymbolA}{m}{n}

  \DeclareMathSymbol{\twoheaddownarrow}{\mathrel}{MnSyA}{27}

  \newcommand{\MSC}[1]{%
    \let\thempfn\relax
    \footnotetext[0]{2020 Mathematics Subject Classification: #1.}
  }

\tikzstyle{vertex}=[circle, fill, minimum size=5pt, inner sep=0pt]
\tikzstyle{openvertex}=[circle, draw, fill=white, line width=0.85pt, opacity=0.6, minimum size=5pt, inner sep=0pt]
\tikzstyle{colorvertex}=[circle, draw, minimum size=6pt, inner sep=0pt]
\tikzstyle{Zcube}=[circle, fill, minimum size=3pt, inner sep=0pt, outer sep = 3pt]
\definecolor{vertex1}{RGB}{255,0,0}
\definecolor{vertex2}{RGB}{0,0,255}
\definecolor{vertex3}{RGB}{0,255,0}
\definecolor{vertex4}{RGB}{255,255,0}
\definecolor{vertex5}{RGB}{255,0,255}
\definecolor{grayfill}{gray}{0.92}

\newcommand{\NN}{\mathbb{N}}
\newcommand{\ZZ}{\mathbb{Z}}

\newcommand{\eps}{\varepsilon}

\renewcommand{\cat}[1]{\mathcal{#1}}

\renewcommand{\Top}{\ncat{Top}}
\newcommand{\Fun}{\operatorname{Fun}}
\newcommand{\bang}{\mathord{!}}
\newcommand{\im}{\operatorname{im}} % image 

\newcommand{\catname}[1]{\mathsf{#1}} % font for named categories

\renewcommand{\Top}{\catname{Top}}

\newcommand{\Pos}{\catname{Pos}}

\newcommand{\aSet}{\catname{aSet}} % presheaves on a small category A
\newcommand{\aSeti}{\catname{aSet}_+} % presheaves on A_+
\newcommand{\boxcati}{\boxcat_{\mathrm{inj}}} % the cube category with only injections
\newcommand{\boxcatin}[1][n]{\boxcat^{\leq #1}_{\mathrm{inj}}} % the cube category with only injections
\newcommand{\cSeti}{\cSet_{\mathrm{semi}}} % the category of semicubical sets

\newcommand{\cosk}{\operatorname{cosk}}

\newcommand{\Tfib}{{T}^{\mathrm{fib}}}

\newcommand{\gnerve}[1][]{\operatorname{N}_{#1}} % Nerve of a graph
\newcommand{\gtimes}{\mathop{\square}}
\newcommand{\gcube}[2]{I_{#1}^{\gtimes #2}} % the n-cube graph
\newcommand{\Cyl}[1]{\operatorname{Cyl}_{#1}}
\newcommand{\susp}[1][]{\Sigma_{#1}}

\newcommand{\Gn}[1][m]{\Gamma_{#1}} % same macro without mandatory argument

\thanksmarkseries{arabic}

\begin{document}

\author{Daniel Carranza \and Krzysztof Kapulkin}
\title{Discrete homotopy hypothesis for $n$-types}

\maketitle

 \begin{abstract} 
  We show that discrete and classical homotopy theories are equivalent after localizing at $n$-equivalences for any non-negative integer $n$.
  By constructing an explicit homotopy inverse to the graph nerve functor associating an $n$-fibrant cubical set to a graph, we are also able to give explicit computations of several previously unknown discrete homotopy groups of boundaries of cubes and suspensions of cycles.
 \end{abstract}
 
% \begin{flushright}
%   \emph{Dedicated to H{\'e}l{\`e}ne Barcelo \\ on the occasion of her retirement.}
% \end{flushright}
\begin{flushright}
  \emph{Dedicated to H{\'e}l{\`e}ne Barcelo.}
\end{flushright}

\tableofcontents*

 \section*{Introduction}

Discrete homotopy theory is an emerging area of mathematics applying intuitions and constructions from continuous settings to rigid and discrete objects such as graphs.
It associates discrete homotopy and homology groups to graphs which carry important combinatorial information about the structure thereof \cite{barcelo-kramer-laubenbacher-weaver,babson-barcelo-longueville-laubenbacher,barcelo-capraro-white}.
As such, it has found applications both within and outside mathematics; the former include areas such as matroid theory \cite{maurer:basis-graphs-i,chalopin-chepoi-osajda} and hyperplane arrangements \cite{barcelo-laubenbacher:perspectives}, while the latter feature application to network analysis \cite{atkin:i, atkin:ii}, combinatorial time series \cite{barcelo-laubenbacher:perspectives}, and, most recently, topological data analysis \cite{kapulkin-kershaw:data}, or TDA for short.
In TDA in particular, the combination of discrete homology and persistence has proven to be more noise resistant than the ``standard'' techniques based on the Vietoris--Rips complex.

Despite these successes, little is known about discrete homotopy and homology groups of a graph \cite{barcelo-greene-jarrah-welker:comparison}.
For example, the discrete homology groups of the \emph{Greene sphere}, a graph with 10 vertices and 16 edges, are known only up to degree 3, and even then the result was obtained using large parallel computations \cite{kapulkin-kershaw:computations}.
Its discrete homotopy groups are only known in degrees 0, 1, and 2.
Speaking of Quillen's algebraic K-theory, Friedlander once remarked: ``Quillen has played a nasty trick on us, giving us very interesting invariants, with which we struggle to make the most basic calculation.''
It is striking how well this quote, modulo the use of Quillen's name, describes discrete homotopy and homology groups.

Some of the aforementioned computations were made possible thanks to a recent proof \cite{carranza-kapulkin:cubical-graphs} of the conjecture of Babson, Barcelo, de Longueville, and Laubenbacher \cite{babson-barcelo-longueville-laubenbacher} that discrete homotopy and homology groups of a graph can be topologically realized.
More precisely, the conjecture asserted that the discrete homotopy and homology groups of a graph $G$ agree with the classical, or continuous, homotopy and homology groups of the space $\reali{\gnerve[1] G}$ that one can associate to $G$.
Here, $\gnerve[1] G$ denotes the $1$-nerve of $G$, a cubical set whose $n$-cubes are given by probing $G$ with hypercube graphs $Q^n$, i.e., $(\gnerve[1] G)_n = \{ \text{graph maps } Q^n \to G \}$.
A cubical set $X$ is a combinatorial model of a space \cite{kan:abstract-homotopy-i,cisinski:presheaves,jardine:a-beginning,carranza-kapulkin:homotopy-cubical} consisting of a family of sets $X_0$, $X_1$, $X_2$, \ldots of formal cubes, connected by face and degeneracy maps indicating which cubes are faces or degeneracies of others.
Given a cubical set $X$, one obtains a space by geometrically realizing $X$, i.e., first taking a disjoint union of cubes in all dimensions and then gluing them together according to the face and degeneracy maps.
The resulting space is usually denoted $\reali{X}$.
Thus the now-proven conjecture gives a recipe on how to build the space from a graph in a way that their classical and discrete homotopy/homology groups agree: informally speaking, glue in an $n$-cube for every map $Q^n \to G$.

The key observation in the proof of the conjecture was that the functor $N_1$ is only the first in a sequence of functors:
\[ \gnerve[1] \longrightarrow \gnerve[2] \longrightarrow \gnerve[3] \longrightarrow \cdots \longrightarrow \gnerve[\infty]\text{,} \]
where $\gnerve[m]$ probes a graph with hypercube graphs of larger edge length.
While $Q^n$ is the $n$-fold box product of a single edge with itself, the $n$-cubes of $\gnerve[m] G$ are obtained by taking the $n$-fold box product of $I_m$, a graph with $m+1$ vertices and $m$ edges arranged in line, with itself.
The functor $\gnerve[\infty]$ is then the colimit of this sequence in the category of functors from graphs to cubical sets.
The upshot is that $\gnerve[\infty]$ is a Kan complex, a particularly nice kind of a cubical set, making it easier to analyze. 
Moreover, all functors in the sequence are weakly equivalent, meaning that their homotopy and homology groups are isomorphic.
Thus instead of working with $\gnerve[1]$, one may equivalently work with $\gnerve[\infty]$, in which case the conjecture is much easier to show.

Unfortunately, even for very small graphs, the space $\reali{\gnerve[1] G}$ is an infinite-dimensional cell complex, or, put differently, the cubical set $\gnerve[1] G$ has non-degenerate cubes in arbitrarily high dimensions, and thus there is virtually no hope of computing discrete homotopy and homology groups via traditional topological methods.
However, the aforementioned result made it possible to import some of the common computational tools from (classical) homotopy theory to discrete homotopy theory; these include the Hurewicz theorem and the long exact sequence of a fibration.

It has since been conjectured by several experts in the field that discrete homotopy theory might be equivalent to the classical one.
This statement can be made precise in several ways, including using the theory of $\infty$-categories and presentations thereof \cite{joyal:qcat-kan,joyal:theory-of-qcats,lurie:htt,cisinski:higher-categories}.
Specifically, it is expected that the graph nerve functor $\gnerve[\infty] \colon \Graph \to \cSet$ from graphs to cubical sets induces an equivalence of homotopy theories/$\infty$-categories.
Here, $\Graph$ denotes the category of graphs with weak equivalences given by the maps inducing isomorphsims on all discrete homotopy (and consequently homology) groups; and $\cSet$ denotes the category of cubical sets.
A proof of this conjecture, typically termed the \emph{discrete homotopy hypothesis}, has remained elusive, but the consequences thereof would have a profound impact on the field.
Indeed, with the equivalence in place, one could formally transfer numerous results from classical to discrete homotopy theory, including: the Blakers--Massey theorem, the Mayer--Vietoris sequence, and Brown's representability (of cohomology) theorem, to name just a few.

A possible approach towards the proof of this conjecture was sketched in a recent preprint \cite{kapulkin-mavinkurve:n-types}---it proposes restricting attention to proving the equivalence of homotopy theories of graphs and spaces localized at $n$-equivalences first.
An $n$-equivalence is a map inducing an isomorphism on homotopy groups from $0$ to $n$.
These maps are easier to control in the discrete context, and therefore the proof seems more plausible.
Loosely speaking, localizing a category at a class of maps means that we treat these maps as isomorphism, although we do so in a suitable homotopy theoretic/higher categorical sense.
With an equivalence established between all finite localizations, one could hope to prove the general case via some type of Postnikov tower convergence argument.
As a proof of concept, the authors of the preprint prove that discrete and classical homotopy theories agree when localized at $1$-equivalences, in which case both classical and discrete homotopy theories recover the homotopy theory of groupoids.

In the present paper, we take the natural next step in this research program, namely, we show that the functor $\gnerve[\infty] \colon \Graph \to \cSet$ is an equivalence when both sides are localized at $n$-equivalences.
Since the category of cubical sets is known to be equivalent to spaces, we may work with it, instead of the category of topological spaces.
Formally, our main theorem can be stated as follows:
\begin{theorem*}[cf.~\cref{nerve-equiv-main-thm}]
  For any non-negative integer $n$, the graph nerve functor $\gnerve[\infty] \colon \Graph \to \cSet$ is an equivalence after localizing both sides at $n$-equivaelnces.
\end{theorem*}
This shows in particular that discrete homotopy theory is rich enough to present all homotopy $n$-types, and provides strong evidence in support of the discrete homotopy hypothesis.

A natural line of attack for this problem is to construct a homotopy inverse of $\gnerve[\infty]$, i.e., an operation turning a cubical set $X$ into a graph, say $G_X$, such that $X$ is $n$-equivalent to $\gnerve[\infty] G_X$.
In constructing such an inverse, one is naturally tempted to try reproducing the cellular nature shared by cubical sets and CW-complexes in the world of graphs.
Every cubical set (or CW-complex) can be built by gluing in one cube (or cell) at a time, that is, taking a pushout of $X$ and the combinatorial $n$-cube $\cube{n}$ along its boundary $\bd\cube{n}$ to obtain $X'$.
For graphs, the analogue would be the pushout of a graph $G$ and the graph-version of the $n$-cube, i.e., $\gcube{m}{n}$, along its natural boundary.
As often in life, there is good news and bad news.
The good news is that one can easily reproduce this construction functorially, using the appropriate $m$-realization functor $\reali[m]{-} \colon \cSet \to \Graph$.
The $m$-realization functor is a left adjoint to the $m$-nerve $\gnerve[m]$, and thus it preserves pushouts, allowing us to build a large class of graphs in this ``one-cell-at-a-time'' manner.

The bad news is that, while one can reasonably expect this construction to work, proving it appears highly non-trivial.
The key difficulty is that pushouts of graphs admit a large number of non-trivial maps from $Q^n$ into the newly-glued graph.
% The key difficulty is that when analyzing pushouts, one has to ensure that the pieces are kept sufficiently ``far apart.''
If the components of a gluing are sufficiently ``far apart'', however, one can exert more control over the set of maps from $Q^n$ into the new graph.
One example that makes this intuition precise is as follows: if $X$ is a union of two induced subgraphs $A$ and $B$ where the shortest path from a vertex in the complement $X \setminus A$ to a vertex in the complement $X \setminus B$ has length at least $n+2$, then every map $Q^n \to X$ factors through either $A$ or $B$.
% One way of making this precise is to say that all maps from the hypercube graph $Q^k$ into the pushout $X \cup_{\bd \gcube{m}{n}} \gcube{m}{n}$ must factor through one of the components of the pushout.
% Thus taking the pushout does not allow for enough ``wiggle room'' to accommodate this requirement.

Thus, instead of taking pushouts, we use a different construction familiar from algebraic topology; namely, the double mapping cylinder.
The double mapping cylinder of a span 
\[ B \xleftarrow{f} A \xrightarrow{g} C \]
is given by the quotient of $B \sqcup A \times [0, 1] \sqcup C$ identifying $(a, 0)$ with $f(a)$ and $(a, 1)$ with $g(a)$.
This construction is a model for the homotopy pushout of the given span, and thus it is not surprising that it would behave better than the rigid pushout.
In the world of graphs, we, of course, need to replace the topological cylinder $X \times [0,1]$ by the box product of graphs $X \gtimes I_m$, but with this change, the construction presented above essentially works.
It adds just enough ``wiggle room'' to ensure that newly glued cells (i.e.\ maps $Q^n \to X$) can be analyzed in this way.

To show that the double mapping cylinder behaves in the expected way upon application of $\gnerve[\infty]$, we introduce the notion of $n$-skeletal pushout.
A pushout of graphs is $n$-skeletal if it is a pushout after applying the $1$-nerve  $\gnerve[1]$ followed by the $n$-skeleton $\Sk^n \colon \cSet \to \cSet$.
Here, the $n$-skeleton removes all nondegenerate $k$-cubes from a cubical set for $k > n$.
In essence, $n$-skeletal pushouts (along monomorphisms) are those squares in the category of graphs that become homotopy pushouts in the model structure for homotopy $(n-1)$-types on cubical sets, introduced in \cite{kapulkin-mavinkurve:n-types}.
A different way of saying it would be that we are ultimately interested in $n$-skeletal pushouts, and the construction of double mapping cylinder is just a convenient way to implement it.

The idea of $n$-skeletal pushouts can be traced back to the work of Barcelo, Greene, Jarrah, and Welker \cite{barcelo-greene-jarrah-welker:vanishing}.
Specifically, a version thereof can be seen in the proof of \cite[Thm.~5.2]{barcelo-greene-jarrah-welker:vanishing}.
They also identified the key example of an $n$-skeletal pushout discussed above, namely writing a graph $G$ as a union of two subgraphs such that the image of every map $Q^n \to G$ lies entirely in one of the two subgraphs.
This example was the key step in their proof of the partial Mayer--Vietoris sequence.
By analyzing their proof from the point of view of cubical sets, we are able to explicitly identify the requisite condition and apply methods from abstract homotopy theory, notably the model structure for homotopy $n$-types, in our analysis, thus deriving stronger conclusions from it.

An astute reader might now recognize a key difficulty: by switching from pushouts to double mapping cylinders, we have lost the necessary functoriality, which we previously had by applying the $m$-realization functor.
The remedy here is to not apply the $m$-realization to $X$ directly, but instead to devise an operation $F_!$ that enlarges $X$ first in such a way that gluing in an $n$-cube in $X$ translates to gluing in a cone to $F_! X$.
As the glued in cells are contractible, the double mapping cylinders obtained via gluing in cells are in fact cones up to homotopy.
Put together, the desired graph $G_X$ is constructed as the $m$-realization $\reali[m]{F_! X}$ of $F_! X$ for $m \geq n+2$.

There is one final caveat we need to address in the construction of $G_X$, namely that $F_!$ cannot be applied to arbitrary cubical sets, only to those arising from semi-cubical sets (i.e., cubical sets with only face maps, without degeneracies) by freely adding degeneracies.
The latter can be characterized exactly as those cubical sets having the \emph{fondind} property, that is, having the property that every \textbf{f}ace \textbf{o}f a \textbf{n}on-\textbf{d}egenerate cube \textbf{i}s \textbf{n}on-\textbf{d}egenerate.
This notion was introduced recently by Maehara, although published sources are difficult to find.
It is true that every cubical set is weakly equivalent to a fondind one, which we show in \cref{sec:fondind-cubification}, following an approach suggested to us by Arakawa and Tsutaya.
We are confident that the results of the appendix are well-known to experts and claim no originality for any of them.

So far, we have only addressed one part of the equivalence, namely homotopical essential surjectivity of $\gnerve[\infty]$.
Indeed, what we have described so far only ensures that every homotopy $n$-type can be modelled by a graph.
To complete the proof of our main theorem, we therefore need to address the up-to-homotopy fullness and faithfulness of $\gnerve[\infty]$.
For this, we once again go back to \cite{kapulkin-mavinkurve:n-types}, recalling that $\gnerve[\infty]$ is an exact functor of fibration category structures for homotopy $n$-types.
Because of that, it suffices to check that it satisfies the approximation properties of \cite{cisinski:categories-derivables}, which provide a relatively lightweight set of conditions to prove that an exact functor is an equivalence of homotopy theories.
This is the essential content of the proof of \cref{nerve-equiv-main-thm}.

There are two additional applications of these results, outside of the aforementioned \cref{nerve-equiv-main-thm}.
First, the proof gives an explicit construction of a graph whose discrete homotopy and homology groups up to a prescribed degree $n$ agree with those of a space, presented as, say, a simplicial complex.
We discuss this construction in the particular case of a simplicial complex presenting $\mathbb{R}P^2$.
Second, several naturally occurring graphs are in fact examples of double mapping cylinders, which allows us to compute a number of new homotopy groups of these graphs.
These include iterated suspensions of cycles and boundaries of $n$-cubes.
Overall, the results presented here greatly enrich the current knowledge of discrete homotopy theory from both theoretical and computational points of view.

\paragraph{Organization of the paper.}
We begin in \cref{sec:preliminaries} by reviewing the necessary background on discrete homotopy theory, cubical homotopy theory, and their interactions.
Our approach and terminology largely follow those of \cite{carranza-kapulkin:cubical-graphs,kapulkin-mavinkurve:n-types}.
In \cref{sec:mapping-cyl}, we discuss $n$-skeletal pushouts of graphs and their implementation via the double mapping cylinder.
In \cref{sec:inverse-construction}, we introduce the functor $F_!$.
In particular, it includes an explicit description of vertices and edges of our homotopy inverse in terms of $F$-\emph{sequences}.
In \cref{sec:main-thm} we prove our main theorem about the equivalence of discrete and classical homotopy theories after localizing at $n$-equivalences.
We conclude in \cref{sec:applications} by discussing computational applications of the previous results, including to explicitly constructing graphs of a prescribed discrete homotopy $n$-type presented by a simplicial complex, and computing previously unknown discrete homotopy groups of boundaries of $n$-cubes and suspensions of cycles.
As indicated above, \cref{sec:fondind-cubification} contains material on fondind and non-singular cubical sets, ensuring that our construction can be applied to an arbitrary homotopy $n$-type.

\paragraph{Acknowledgements.}
This project started with a discussion the second author had with Mark Behrens and Samira Jamil on whether the unit map $X \to \gnerve[n]\reali[n]X$ is an isomorphism on homology groups up to degree $n$.
We are deeply indebted for their insightful ideas that sparked the results in their present form.

We are also grateful to Kensuke Arakawa, Yuki Maehara, and Mitsunobu Tsutaya for their contributions related to \cref{sec:fondind-cubification}; in particular, for introducing the notion of a fondind presheaf and suggesting a proof technique allowing us to show that every cubical set is weakly equivalent to a fondind one.

We dedicate this paper to H{\'e}l{\`e}ne Barcelo, the founder of the field of discrete homotopy theory, whose trailblazing ideas continue to resonate through our work and our community.

 \section{Preliminaries} \label{sec:preliminaries}

In this section, we collected the necessary background material on discrete homotopy theory, cubical homotopy theory, and their interactions.

\subsection{Graphs and discrete homotopy theory}

We begin with a brief overview of the specific category of graphs under consideration.

A \emph{graph} is a set with a reflexive symmetric relation.
We write graphs using letters $G, H, \dots$, omitting the relation from our notation.
We refer to elements of the set as \emph{vertices} of the graph.
For a graph $G$, the notation $v \in G$ means $v$ is a vertex of $G$.
When a pair of elements $(v, w) \in G \times G$ is in the relation, we write $v \sim w$ and say $v$ is \emph{connected} to $w$, or that there is \emph{an edge} from $v$ to $w$.
We refer to the underlying set of a graph $G$ as the set of \emph{vertices} of $G$, denoted $V(G)$.

For a graph $G$, a \emph{subgraph} $H$ of $G$ consists of a subset $V(H) \subseteq V(G)$ together with a reflexive symmetric relation on $V(H)$ such that if $v \sim w$ in $H$ then $v \sim w$ in $G$.
The notation $H \subseteq G$ means $H$ is a subgraph of $G$.
An \emph{induced subgraph} is a subgraph $H \subseteq G$ with the property that $v \sim w$ in $H$ if and only if $v \sim w$ in $G$.

Given two graphs $G$ and $H$, a \emph{graph map} from $G$ to $H$ is a set function $f \from G \to H$ that preserves the relation.
We denote the category of graphs by $\Graph$, and remark that this category is a reflective subcategory of a presheaf category \cite[Def.~1.1]{carranza-kapulkin:cubical-graphs}.

Some basic graphs we consider are the \emph{interval graphs} and the \emph{cycle graphs}.
\begin{definition}
    For $m \geq 0$,
    \begin{enumerate}
        \item the \emph{$m$-interval} $I_m$ is the graph which has:
        \begin{itemize}
            \item as vertices, integers $0 \leq i \leq m$;
            \item an edge between $i$ and $i+1$ for all $i \in \{ 0, \dots, m-1 \}$.
        \end{itemize}
        \item the \emph{$m$-cycle} $C_m$ is the graph which has:
        \begin{itemize}
            \item as vertices, integers $0 \leq i \leq m-1$;
            \item an edge between $i$ and $i + 1$ for $i \in \{ 0, \dots, m-2 \}$, and an edge between $m-1$ and $0$.
        \end{itemize}
        \item the \emph{infinite interval} $I_\infty$ is the graph which has:
        \begin{itemize}
            \item as vertices, integers $i \in \ZZ$;
            \item an edge between $i$ and $i+1$ for all $i \in \ZZ$.
        \end{itemize}
    \end{enumerate}
\end{definition}

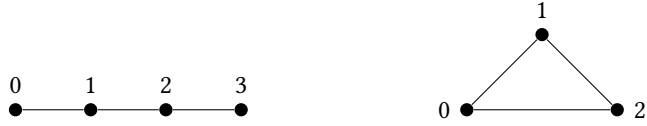
\begin{figure}[H]
    \centering
    \begin{tikzpicture}
        % vertices
        \node[vertex, label={0}] (0) {};
        \path (0) -- +(1, 0) node[vertex, label={1}] (1) {};
        \path (1) -- +(1, 0) node[vertex, label={2}] (2) {};
        \path (2) -- +(1, 0) node[vertex, label={3}] (3) {};

        \path (3) -- +(3, 0) node[vertex, label={west:0}] (1c) {};
        \path (1c) -- +(1, 1) node[vertex, label={1}] (2c) {};
        \path (2c) -- +(1, -1) node[vertex, label={east:2}] (3c) {};
        % edges
        \draw (0) to (1)
              (1) to (2)
              (2) to (3);
        \draw (1c) to (2c) to (3c) to (1c);
    \end{tikzpicture}
    \caption{The graphs $I_3$ and $C_3$, respectively.}
\end{figure}

To obtain a notion of homotopy in the setting of graphs, we fix a choice of product of graphs: the \emph{box product}.
\begin{definition}
    Given two graphs $G$ and $H$, their \emph{box product} is the graph $G \gtimes H$ whose set of vertices is given by $V(G \gtimes H) := V(G) \times V(H)$.
    Two vertices $(v, w) \sim (v', w')$ are connected if either:
    \begin{itemize}
        \item $v = v'$ and $w \sim w'$; or
        \item $v \sim v'$ and $w = w'$.
    \end{itemize}
\end{definition}
The product projection functions yield graph maps $\pi_G \from G \gtimes H \to G$ and $\pi_H \from G \gtimes H \to H$.
The notation $G^{\gtimes n}$ denotes the $n$-fold box product of a graph $G$ with itself.
Our convention is that if $n = 0$ then $G^{\gtimes 0} = I_0$.
The box product defines a closed symmetric monoidal product on the category of graphs \cite[Prop.~2.5]{kapulkin-kershaw:monoidal}.

Equipped with a notion of interval and a notion of product, we now define \emph{homotopies} between graph maps.
\begin{definition} \leavevmode 
    \begin{enumerate}
        \item Given graph maps $f, g \from G \to H$, a \emph{homotopy} from $f$ to $g$ is a graph map $\alpha \from G \gtimes I_m \to H$ for some $m \geq 0$ such that $\alpha(v, 0) = f(v)$ and $\alpha(v, m) = g(v)$ for all $v \in G$.
        \item A graph map $f \from G \to H$ is a \emph{homotopy equivalence} if there exists a graph map $g \from H \to G$ together with homotopies $gf \sim \id[G]$ and $fg \sim \id[H]$.
    \end{enumerate}
\end{definition}
We remark that homotopies form an equivalence relation on the set of graph maps $\{ G \to H \}$.

The notion of homotopy equivalences of graphs is often poorly behaved when compared to the analogous notion in the topological setting.
For instance, the cycle graphs $C_3$ and $C_4$ are homotopy equivalent to the graph $I_0$ with a single vertex, while \emph{none} of the cycles $C_5, C_6, C_7, \dots$ are homotopy equivalent.
It turns out that it is much more convenient to study graphs up to \emph{weak homotopy equivalence}.
These are graph maps which induce isomorphisms on the discrete homotopy groups, which we now define.

% A \emph{based graph} is a graph $G$ together with a distinguished vertex $x \in X$.
\begin{definition}
    Let $G$ be a graph and $v \in G$ be a vertex.
    For $n \geq 0$, the \emph{$n$-th discrete homotopy group} of $(G, v)$ is the quotient
    \[ A_n(G, v) := \{ f \from \gcube{\infty}{n} \to G \mid f(\vec{i}) = v \text{ for all but finitely many } \vec{i} \in \gcube{\infty}{n} \} \Big/ \sim , \]
    where we identify $f \sim g$ if and only if there exists a homotopy $\alpha \from \gcube{\infty}{n} \gtimes I_m \to G$ from $f$ to $g$ such that $\alpha(\vec{i}, t) = v$ for all but finitely many vertices $(\vec{i}, t) \in \gcube{\infty}{n} \gtimes I_m$.
\end{definition}
For $n \geq 1$, the set $A_n(G, v)$ has a group operation given by \emph{concatenation}, see \cite[Def.~1.28 \& 1.29]{kapulkin-mavinkurve:fundamental-group}.
When $n = 0$, the set $A_0(G, v)$ recovers the usual notion of connected components of a graph.
If $G$ is connected then $A_n(G, v)$ is independent of the choice of vertex $v \in G$.
In this case, we use the notation $A_n G$ to mean any of the isomorphic groups $A_n(G, v)$ for $v \in G$.

We now define the key notions of ``equivalence'' between graphs that we will study throughout the paper.
\begin{definition}
    Let $f \from G \to H$ be a graph map.
    \begin{enumerate}
        \item Given $n \geq 0$, we say $f$ is an $n$-equivalence if, for all $v \in G$ and $k \in \{ 0, \dots, n \}$, the map $A_k f \from A_k(G, v) \to A_k(H, f(v))$ is an isomorphism.
        \item We say $f$ is a \emph{weak homotopy equivalence} if it is an $n$-equivalence for all $n \in \NN$.
    \end{enumerate}
\end{definition}
Just as in algebraic topology, a major theme of discrete homotopy theory is to classify graphs up to weak homotopy equivalence.

Another invariant of discrete homotopy is the \emph{discrete homology groups} of a graph.
To any graph $G$, we may construct a chain complex $C_\bullet G$ over $\ZZ$ whose group of $n$-chains is free on the set of graph maps $\gcube{1}{n} \to G$.
\[ C_n G := \ZZ \langle \{ f \from \gcube{1}{n} \to G \} \rangle. \]
The chain map $\bd_n \from C_n G \to C_{n-1} G$ is defined on a basis element $f$ to be the alternating sum of its faces.
\[ \bd_n(f) := \sum\limits_{i=1}^{n} \big( (-1)^{i}(f\face{}{i,1} - f\face{}{1,0}) \big). \]  
One verifies that $\im \bd_{n+1} \subseteq \ker \bd_n$.
The group $C_n G$ admits a subgroup $D_n G$ generated by maps $f \from \gcube{1}{n} \to G$ which factor through some non-identity product projection $\sigma \from \gcube{1}{n} \to \gcube{1}{m}$.
That is,
\[ D_n G := \ZZ \langle \{ f \from \gcube{1}{n} \to G \mid f \text{ factors through a product projection } \sigma \from \gcube{1}{n} \to \gcube{1}{m} \text{ with } m < n \} \rangle. \]
One verifies that the chain map $\bd_n$ sends $D_n G$ to $D_{n-1} G$, hence the quotient $C_\bullet G / D_\bullet G$ is again a chain complex. 
The discrete homology of $G$ is defined to be the homology groups of this quotient complex.
\begin{definition}
    Let $G$ be a graph.
    For $n \geq 0$, the $n$-th \emph{discrete homology group} $H_n G$ of $G$ (with coefficients in $\ZZ$) is the $n$-th homology of the chain complex $(C_\bullet G / D_\bullet G, \bd)$.
    That is,
    \[ H_n(G) := \frac{\ker \big( \bd_n \from C_n G / D_n G \to C_{n-1} G / D_{n-1} G \big) }{\im \big( \bd_{n+1} \from C_{n+1} G / D_{n+1} G \to C_n G / D_n G \big)}. \]
\end{definition}
One may define discrete homology of a graph with coefficients in an arbitrary ring or abelian group by appropriately tensoring the chain complex $C_\bullet G / D_\bullet G$.
Similarly, one may define the discrete cohomology groups of a graph by dualizing the chain complex $C_\bullet G / D_\bullet G$.
That said, we will content ourselves to work entirely with homology groups using $\ZZ$ coefficients.
% The discrete cohomology groups of a graph are defined by dualizing the chain complex $C_\bullet G$.

\subsection{Cubical sets}

The primary way through which we relate the homotopy theory of graphs to the homotopy theory of spaces is via \emph{cubical sets}, which we now define.

Recall the box category $\Box$ is the category whose objects are posets of the form $[1]^n = \{0 \leq 1\}^n$, and whose morphisms are generated (inside the category of posets) under composition by the following four special classes:
\begin{itemize}
  \item \emph{faces} $\partial^n_{i,\varepsilon} \colon [1]^{n-1} \to [1]^n$ for $i = 1, \ldots , n$ and $\varepsilon = 0, 1$ given by:
  \[ \partial^n_{i,\varepsilon} (x_1, x_2, \ldots, x_{n-1}) = (x_1, x_2, \ldots, x_{i-1}, \varepsilon, x_i, \ldots, x_{n-1})\text{;}  \]
  \item \emph{degeneracies} $\sigma^n_i \colon [1]^n \to [1]^{n-1}$ for $i = 1, 2, \ldots, n$ given by:
  \[ \sigma^n_i ( x_1, x_2, \ldots, x_n) = (x_1, x_2, \ldots, x_{i-1}, x_{i+1}, \ldots, x_n)\text{;}  \]
  \item \emph{negative connections} $\gamma^n_{i,0} \colon [1]^n \to [1]^{n-1}$ for $i = 1, 2, \ldots, n-1$ given by:
  \[ \gamma^n_{i,0} (x_1, x_2, \ldots, x_n) = (x_1, x_2, \ldots, x_{i-1}, \max\{ x_i , x_{i+1}\}, x_{i+2}, \ldots, x_n) \text{.} \]
  \item \emph{positive connections} $\gamma^n_{i,1} \colon [1]^n \to [1]^{n-1}$ for $i = 1, 2, \ldots, n-1$ given by:
  \[ \gamma^n_{i,1} (x_1, x_2, \ldots, x_n) = (x_1, x_2, \ldots, x_{i-1}, \min\{ x_i , x_{i+1}\}, x_{i+2}, \ldots, x_n) \text{.} \]
\end{itemize}

These maps obey the following \emph{cubical identities}:

\[ \begin{array}{l l}
    \partial_{j, \varepsilon'} \partial_{i, \varepsilon} = \partial_{i+1, \varepsilon} \partial_{j, \varepsilon'} \quad \text{for } j \leq i; & 
    \sigma_j \partial_{i, \varepsilon} = \begin{cases}
        \partial_{i-1, \varepsilon} \sigma_j & \text{for } j < i; \\
        \id                                                       & \text{for } j = i; \\
        \partial_{i, \varepsilon} \sigma_{j-1} & \text{for } j > i;
    \end{cases} \\
    \sigma_i \sigma_j = \sigma_j \sigma_{i+1} \quad \text{for } j \leq i; &
    \gamma_{j,\varepsilon'} \gamma_{i,\varepsilon} = \begin{cases}
    \gamma_{i,\varepsilon} \gamma_{j+1,\varepsilon'} & \text{for } j > i; \\
    \gamma_{i,\varepsilon}\gamma_{i+1,\varepsilon} & \text{for } j = i, \varepsilon' = \varepsilon;
    \end{cases} \\
    \gamma_{j,\varepsilon'} \partial_{i, \varepsilon} = \begin{cases} 
        \partial_{i-1, \varepsilon} \gamma_{j,\varepsilon'}   & \text{for } j < i-1 \text{;} \\
        \id                                                         & \text{for } j = i-1, \, i, \, \varepsilon = \varepsilon' \text{;} \\
        \partial_{j, \varepsilon} \sigma_j         & \text{for } j = i-1, \, i, \, \varepsilon = 1-\varepsilon' \text{;} \\
        \partial_{i, \varepsilon} \gamma_{j-1,\varepsilon'} & \text{for } j > i;
    \end{cases} &
    \sigma_j \gamma_{i,\varepsilon} = \begin{cases}
        \gamma_{i-1,\varepsilon} \sigma_j  & \text{for } j < i \text{;} \\
        \sigma_i \sigma_i           & \text{for } j = i \text{;} \\
        \gamma_{i,\varepsilon} \sigma_{j+1} & \text{for } j > i \text{.} 
    \end{cases}
\end{array} \]

A \emph{cubical set} is a functor $\boxcat^\op \to \Set$.
The category of cubical sets is denoted by $\cSet$, and we use the term \emph{cubical map} to mean a morphism of cubical sets.
% We refer to morphisms of $\cSet$ as \emph{cubical maps}.
We adopt the standard convention of writing $X_n$ for the value of a cubical set $X$ at an object $[1]^n$.
Moreover, we write cubical operators on the right e.g.~given an $n$-cube $x \in X_n$ of a cubical set $X$, we write $x\face{}{1,0}$ for the $\face{}{1,0}$-face of $x$. 
We typically denote cubical sets using capital letters $X, Y, \dots$.
% We will write $\cSet$ for the category of cubical sets.
\begin{definition} Let $n \geq 0$.
    \begin{enumerate}
        \item The \emph{combinatorial $n$-cube} $\cube{n}$ is the representable functor $\boxcat(-, [1]^n) \from \boxcat^\op \to \Set$.
        \item The \emph{boundary of the $n$-cube} $\bd \cube{n}$ is the subobject of $\cube{n}$ defined by
        \[ \bd \cube{n} := \bigcup\limits_{\substack{j=1,\dots,n \\ \eta = 0, 1}} \operatorname{Im} \face{}{j,\eta}. \]
        \item Given $i = 1, \dots, n$ and $\eps = 0, 1$, the $(i, \eps)$-open box  $\dfobox$ is the subobject of $\bd \cube{n}$ defined by
        \[ \dfobox := \bigcup\limits_{(j,\eta) \neq (i,\varepsilon)} \operatorname{Im} \face{}{j,\eta}. \]
    \end{enumerate}
\end{definition}

The category of cubical sets admits a closed monoidal product, known as the \emph{geometric product} of cubical sets.
This product is constructed as the left Kan extension of the functor $\boxcat \times \boxcat \to \cSet$ given by $([1]^m, [1]^n) \mapsto \cube{m+n}$.
\[ \begin{tikzcd}
    \boxcat \times \boxcat \ar[r] \ar[d] & \cSet \\
    \cSet \times \cSet \ar[ur, "\gprod"']
\end{tikzcd} \]
Note that the geometric product of cubical sets is \emph{not} symmetric; while there is an isomorphism $\cube{m} \gprod \cube{n} \cong \cube{n} \gprod \cube{m}$ given by the identity map, this isomorphism fails to be natural with respect to the cubical face inclusions.

% The geometric product of cubical sets admits the following explicit description.
% \begin{proposition}[{\cite[Prop.~1.24]{doherty-kapulkin-lindsey-sattler}}] \label{thm:gprod_cube}
%     Let $X, Y$ be cubical sets.
%     \begin{enumerate}
%         \item For $k \geq 0$, the $k$-cubes of $X \gprod Y$ consist of all pairs $(x \in X_m, y \in Y_n)$ such that $m + n = k$, subject to the identification $(x\degen{}{m+1}, y) = (x, y\degen{}{1})$.
%         \item For $x \in X_m$ and $y \in Y_n$, the faces, degeneracies, and connections of the $(m+n)$-cube $(x, y)$ are computed by
%         \begin{align*}
%             (x, y)\face{}{i,\varepsilon} &= \begin{cases}
%                 (x\face{}{i,\varepsilon}, y) & 1 \leq i \leq m \\
%                 (x, y\face{}{i-m, \varepsilon}) & m+1 \leq i \leq m+n;
%             \end{cases} \\
%             (x, y)\degen{}{i} &= \begin{cases}
%                 (x\degen{}{i}, y) & 1 \leq i \leq m+1 \\
%                 (x, y\degen{}{i-m}) & m+1 \leq i \leq m+n;
%             \end{cases} \\
%             (x, y)\conn{}{i,\varepsilon} &= \begin{cases}
%                 (x\conn{}{i,\varepsilon}, y) & 1 \leq i \leq m \\
%                 (x, y\conn{}{i-m,\varepsilon}) & m+1 \leq i \leq n. \tag*{\qedsymbol}
%             \end{cases}
%         \end{align*}
%     \end{enumerate}
% \end{proposition}

Every cubical set admits a filtration by its \emph{skeleta}.
Given a cubical set $X$ and $n \geq -1$, let $\Sk^n X \subseteq X$ be the subobject of $X$ containing all non-degenerate cubes of $X$ of dimension $\leq n$ (in particular, $Sk^{-1} X = \varnothing$).
For $m \leq n$, we have an inclusion $\Sk^m X \ito \Sk^n X$ which is natural in the variable $X \in \cSet$.
\begin{proposition} \label{cubical-set-admits-skeletal-filtration}
    Let $X$ be a cubical set.
    For $n \geq 0$, the inclusion $\Sk^{n-1} X \ito \Sk^n X$ forms the right arrow in a pushout square
    \[ \begin{tikzcd}
        \displaystyle \coprod\limits_{x \in (X_n)_{\mathrm{nd}}} \bd \cube{n} \ar[r, "{\bd x}"] \ar[d, hook] & \Sk^{n-1} X \ar[d, hook] \\
        \displaystyle \coprod\limits_{x \in (X_n)_{\mathrm{nd}}} \cube{n} \ar[r, "{x}"] & \Sk^n X
    \end{tikzcd} \]
    where $(X_n)_{\mathrm{nd}}$ denotes the set of non-degenerate $n$-cubes of $X$.
    Moreover, the diagram
    \[ \Sk^0 X \ito \Sk^1 X \ito \dots \ito \Sk^n X \ito \dots \ito X \]
    is a colimit diagram. \qed
\end{proposition}
For each $n$, the functor $\Sk^n \from \cSet \to \cSet$ admits a right adjoint, known as the \emph{$n$-coskeleton} functor and denoted $\cosk^n \from \cSet \to \cSet$.
Explicitly, the set of $k$-cubes of the $n$-coskeleton of a cubical set $X$ is given by
\[ \cosk^n(X)_k := \cSet(\Sk^n \cube{k}, X). \]

Cubical sets are related to topological spaces via \emph{geometric realization}.
% Every cubical set admits a \emph{realization} as a topological space.
More precisely, the geometric realization functor $\reali{\uvar} \from \cSet \to \Top$ is defined as the extension by colimits of the functor $\boxcat \to \Top$ which sends $[1]^n$ to the topological $n$-cube $[0, 1]^n$.
\[ \begin{tikzcd}
    \boxcat \ar[r, "{[1]^n \mapsto [0, 1]^n}"] \ar[d] &[+3ex] \Top \\
    \cSet \ar[ur, dotted, bend right, "{\reali{\uvar}}"']
\end{tikzcd} \]

We use geometric realization to construct the \emph{Grothendieck} model structure on cubical sets, which models the homotopy theory of topological spaces.

\begin{definition}\leavevmode
    \begin{enumerate}
        \item A cubical map $X \to Y$ is a \emph{Kan fibration} if it has the right lifting property with respect to open box inclusions.
        That is, if for any commutative square,
        \[ \begin{tikzcd}
            \dfobox \ar[d, hook] \ar[r] & X \ar[d, "f"] \\
            \cube{n} \ar[r] & Y
        \end{tikzcd} \]
        there exists a map $\cube{n} \to X$ so that the triangles
        \[ \begin{tikzcd}
            \dfobox \ar[d, hook] \ar[r] & X \ar[d, "f"] \\
            \cube{n} \ar[ur, dotted] \ar[r] & Y
        \end{tikzcd} \]
        commutes.
        \item A cubical set $X$ is a \emph{Kan complex} if the unique map $X \to \cube{0}$ is a Kan fibration.
    \end{enumerate}
\end{definition}
We write $\Kan$ for the full subcategory of $\cSet$ consisting of Kan complexes.

\begin{theorem}[{Cisinski, cf.~\cite[Thm.~1.34]{doherty-kapulkin-lindsey-sattler}}]
	The category of cubical sets $\cSet$ admits a model structure, called the \emph{Grothendieck model structure}, whose:
	\begin{itemize}
		\item cofibrations are the monomorphisms;
		\item fibrations are the Kan fibrations; and
		\item weak equivalences are the maps $f \from X \to Y$ which become weak homotopy equivalences $\reali{f} \from \reali{X} \to \reali{Y}$ under geometric realization. \qed
	\end{itemize}
\end{theorem}

Given $n \geq 0$, we say that a continuous function $f \from S \to T$ between topological spaces is an \emph{$n$-equivalence} if it induces a bijection on connected components and, for all $x \in X$, the map $\pi_k f \from \pi_k (X, x) \to \pi_k(Y, y)$ is an isomorphism for all $k \in \{ 1, \dots, n \}$.
In particular, $f$ is a weak homotopy equivalence if it is an $n$-equivalence for all $n$.

We have an analogous definition for maps between cubical sets.
\begin{definition}
    Given $n \geq 0$, a cubical map $f \from X \to Y$ is an \emph{$n$-equivalence} if $\reali{f} \from \reali{X} \to \reali{Y}$ is an $n$-equivalence of topological spaces.
\end{definition}
The collection of $n$-equivalences participates in two distinct model structures on $\cSet$: the \emph{Cisniski model structure for $n$-types} and the \emph{transferred model structure for $n$-types}.
We first describe the Cisinski model structure, for which we will need the following definition.
\begin{definition}
    Let $n \geq 0$.
    \begin{enumerate}
        \item A cubical map $f \from X \to Y$ is a \emph{naive $n$-fibration} if it has the right lifting property against the set of maps:
        \[ \{ \obox{k}{i, \eps} \ito \cube{k} \mid k \geq 1, \ i = 1, \dots, k, \ \eps = 0, 1 \} \cup \{ \bd \cube{k} \ito \cube{k} \mid k \geq n+2 \}. \]
        \item A cubical set $X$ is \emph{$n$-fibrant} if the unique map $X \to \cube{0}$ is a naive $n$-fibration.    
    \end{enumerate}
\end{definition}
\begin{theorem}[{\cite[Thm.~2.7]{kapulkin-mavinkurve:n-types}}]
    The category of cubical sets $\cSet$ admits a model structure, called the \emph{Cisinski model structure for $n$-types}, whose:
	\begin{itemize}
		\item cofibrations are the monomorphisms;
		\item fibrant objects are the $n$-fibrant cubical sets; and
		\item weak equivalences are the $n$-equivalences. \qed
	\end{itemize}
\end{theorem}
By construction, the Cisinski model structure for $n$-types is a left Bousfield localization of the Grothendieck model structure.
We also remark that every naive $n$-fibration is a fibration in this model structure, though the converse only holds for maps between $n$-fibrant objects.

\begin{theorem}[{c.f.~\cite[Thm.~2.15]{kapulkin-mavinkurve:n-types}}] \label{transferred-model-structure}
    The category of cubical sets $\cSet$ admits a model structure, called the \emph{transferred model structure for $n$-types}, whose:
	\begin{itemize}
        \item cofibrations are the monomorphisms $f \from X \to Y$ such that the square
        % \[ \begin{tikzcd}
        %     \Sk^{n+1} X \ar[rr, hook] \ar[dd, "\Sk^{n+1} f"', hook] & {} &[-1.5em] X \ar[dd, "f", hook] \\
        %     {} & {} & {} \\[-1em]
        %     \Sk^{n+1} Y \ar[rr, hook] & {} &  Y
        % \end{tikzcd} \]
        \[ \begin{tikzcd}
            \Sk^{n+1} X \ar[r, hook] \ar[d, "\Sk^{n+1} f"'] & X \ar[d, "f"] \\
            \Sk^{n+1} \ar[r, hook] Y & Y
        \end{tikzcd} \]
        is a pushout square;
		\item fibrations are the cubical maps $f \from X \to Y$ such that $\cosk^{n+1} f \from \cosk^{n+1} X \to \cosk^{n+1} Y$ is a Kan fibration; and
		\item weak equivalences are the $n$-equivalences. \qed
	\end{itemize}
\end{theorem}
\begin{proof}
    The existence of the model structure is \cite[Thm.~2.15]{kapulkin-mavinkurve:n-types}.
    It remains only the verify the description of the cofibrations.
    To this end, fix a cubical map $f \from X \to Y$.

    If $f$ is a monomorphism then $\Sk^{n+1} f$ is a cofibration since $\Sk^{n+1}$ is left Quillen.
    This means that if the square
    \[ \begin{tikzcd}
        \Sk^{n+1} X \ar[r, hook] \ar[d, "\Sk^{n+1} f"'] & X \ar[d, "f"] \\
        \Sk^{n+1} \ar[r, hook] Y & Y
    \end{tikzcd} \]
    is a pushout then $f$ is a cofibration as a pushout of a cofibration.

    If $f$ is a cofibration then consider the induced map from the pushout
    \[ \Phi \from X \push_{\Sk^{n+1} X} \Sk^{n+1} Y \to Y . \]
    By construction, this map is a bijection on $k$-cubes for all $k \leq n+1$ (and injective for all $k$).
    It follows that the map $\cosk^{n+1} \Phi$ is an isomorphism, which means $\Phi$ is an acyclic fibration in the transferred model structure.
    With this, we have a lift of the square:
    \[ \begin{tikzcd}
        X \ar[r] \ar[d, "f"'] & X \push_{\Sk^{n+1} X} \Sk^{n+1} Y \ar[d, "\Phi"] \\
        Y \ar[r, equal] \ar[ur, dotted] & Y
    \end{tikzcd} \]
    The map $\Phi$ is a monomorphism and a split epimorphism, hence an isomorphism.
\end{proof}
By construction, this model structure is right transferred from the Grothendieck model structure along the $(n+1)$-coskeleton functor $\cosk^{n+1} \from \cSet \to \cSet$.
As such, it is cofibrantly generated \cite[Thm.~11.3.2]{hirschhorn:model-categories}.
\begin{remark}
    The last argument given for \cref{transferred-model-structure} can be used to give an alternative proof of \cite[Prop.~2.21]{kapulkin-mavinkurve:n-types}.
\end{remark}

These three model structures on $\cSet$ are related by the following Quillen adjunctions.
\begin{theorem}[{c.f.~\cite[Thm.~2.15]{kapulkin-mavinkurve:n-types}}] \label{n-equiv-quillen-pairs}
    Each adjunction in the diagram
    \[ \begin{tikzcd}[column sep = 0em, row sep = 3em]
        {} & \cSet_{\text{Grothendieck}} \ar[dl, "\id"', xshift=-4ex, bend right=10] \ar[dr, "\Sk^{n+1}", xshift=4ex, bend left=10] & {} \\
        \cSet_{n\text{-Cisinski}} \ar[ur, "\id" description] \ar[rr, "\Sk^{n+1}", yshift=1ex] & {} & \cSet_{n\text{-transferred}} \ar[ll, "\cosk^{n+1}", yshift=-1ex] \ar[ul, "\cosk^{n+1}" description]
    \end{tikzcd} \]
    is a Quillen adjunction.
    The bottom adjunction is moroever a Quillen equivalence. \qed
\end{theorem}

Mirroring the theory of simplicial sets, one may define homology groups of any cubical set as the homology of an appropriate chain complex.
More precisely, given a cubical set $X$, we define a chain complex $C_\bullet X$ whose group of $n$-chains is generated by the set of $n$-cubes
\[ C_n X := \ZZ \langle X_n \rangle. \]
Let $DX_n \subseteq X_n$ denote the subset of $n$-cubes $x \in X_n$ which can be written as $x = y\sigma$ for some $y \in X_{m}$ and some composite of non-identity degeneracy maps $\sigma \from [1]^n \to [1]^m$.
Taking free abelian groups, this subset $DX_n$ becomes a subgroup
\[ D_n X := \ZZ \langle DX_n \rangle. \]
Define a map $\bd_n \from C_n X \to C_{n-1} X$ on basis elements by
\[ \bd_n(x) := \sum\limits_{i=1}^n (-1)^n (x\face{}{i, 1} - x\face{}{i, 0}). \] 
One verifies that $\im \bd_{n+1} \subseteq \ker \bd_n$, and that $\bd_n$ sends $D_n X$ to $D_{n-1} X$.
Thus, we obtain a chain complex defined by the quotient $C_\bullet X / D_\bullet X$.
\begin{definition}
    Let $X$ be a cubical set.
    Given $n \geq 0$, the \emph{$n$-th cubical homology group} of $X$ is the $n$-th homology group of the chain complex $C_\bullet X / D_\bullet X$.
    That is,
    \[ H_n(X) := \frac{\ker \big( \bd_n \from C_n X / D_n X \to C_{n-1} X / D_{n-1} X \big) }{\im \big( \bd_{n+1} \from C_{n+1} X / D_{n+1} X \to C_n X / D_n X \big)}. \]
\end{definition}
Homology is a homotopy invariant of cubical sets, meaning it inverts weak homotopy equivalences.
\begin{proposition} \label{cubical-homology-inverts-equivs}
    Let $f \from X \to Y$ be a cubical map.
    \begin{enumerate}
        \item If $f$ is a weak equivalence then $H_n f$ is an isomorphism for all $n \geq 0$.
        \item Given $n \geq 0$, if $f$ is a $n$-equivalence then $H_k f \from H_k X \to H_k Y$ is an isomorphism for all $k \in \{ 0, \dots, n \}$.
    \end{enumerate}
\end{proposition}
\begin{proof}
    We first observe that the skeletal filtration on any cubical set $X$ endows the geometric realization $\reali{X}$ with the structure of a CW complex.
    The attaching maps in each dimension are specified by the pushout square in \cref{cubical-set-admits-skeletal-filtration}.
    
    For a cubical set $X$, let $\overline{X}$ denote the cubical set whose set of $n$-cubes is given by pairs $(\varphi, x)$ where $\varphi \from [1]^n \to [1]^k$ is any surjection in $\boxcat$ with $\dom \varphi = [1]^n$ and $x \in X_k$ is any $k$-cube in $X$.
    We quotient the set of such pairs by the relation generated by $(\varphi, x\degen{}{i}) = (\degen{}{i} \circ \varphi, x)$ for any product projection $\degen{}{i} \from [1]^k \to [1]^{k-1}$.
    This construction is functorial in $X$, and we have a natural map $\overline{X} \to X$ given by $(\varphi, x) \mapsto x\varphi$.
    % We have a natural map $\overline{X} \to X$ given by 
    By \cite[Prop.~2.28]{doherty:connections}, this map is a weak homotopy equivalence.

    One can construct a natural isomorphism of chain complexes
    \[ C_\bullet \overline{X} / D_\bullet \overline{X} \xrightarrow{\cong} C_\bullet^{\mathrm{CW}}(\reali{X}) \]
    from the cubical chain complex of $\overline{X}$ to the CW chain complex of the geometric realization $\reali{X}$.
    Part (1) now follows from the fact that CW homology inverts weak homotopy equivalences between CW complexes.

    For part (2), suppose $f \from X \to Y$ is an $n$-equivalence.
    By the existence of functorial factorization in the Grothendieck model structure, we may choose a commutative square
    \[ \begin{tikzcd}
        X \ar[r, "\sim"] \ar[d, "f"'] & \tilde{X} \ar[d, "\tilde{f}"] \\
        Y \ar[r, "\sim"] & \tilde{Y}
    \end{tikzcd} \]
    such that the vertical arrows are $n$-equivalences, the horizontal arrows are weak homotopy equivalences, and the cubical sets $\tilde{X}$ and $\tilde{Y}$ are Kan.
    Since every Kan complex is fibrant in the $n$-transferred model structure, the right arrow in the diagram
    \[ \begin{tikzcd}
        X \ar[r, "\sim"] \ar[d, "f"'] & \tilde{X} \ar[r] & \cosk^{n+1} \tilde{X} \ar[d, "{\cosk^{n+1}\tilde{f}}"] \\
        Y \ar[r, "\sim"] & \tilde{Y} \ar[r] & \cosk^{n+1} \tilde{Y}
    \end{tikzcd} \]
    is a weak homotopy equivalence by \cref{n-equiv-quillen-pairs}.
    The maps $\tilde{X} \to \cosk^{n+1} \tilde{X}$ and $\tilde{Y} \to \cosk^{n+1}\tilde{Y}$ induce bijections on sets of cubes up to dimension $\leq n+1$, hence they induce isomorphisms on homology up to dimension $n$.
    The maps $X \to \tilde{X}$ and $Y \to \tilde{Y}$ induce isomorphisms on homology by part (1).
    Likewise, the map $\cosk^{n+1} \tilde{f}$ induces an isomorphism on homology, hence $f$ induces an isomorphism on homology up to dimension $n$ by 2-out-of-3.
    % In the composite square,
    % \[ \begin{tikzcd}
    %     X \ar[r] \ar[d, "f"'] & \tilde{X} \ar[d, "\tilde{f}"] \ar[r] & \cosk^{n+1} \tilde{X} \ar[d, "{\cosk^{n+1} \tilde{f}}"] \\
    %     Y \ar[r] & \tilde{Y} \ar[r] & \cosk^{n+1} \tilde{Y}
    % \end{tikzcd} \]
    % the maps $\tilde{X} \to \cosk^{n+1}\tilde{X}$ and $\tilde{Y} \to \cosk^{n+1}\tilde{Y}$ are $n$-equivalences by \cite[Lem.~2.16]{kapulkin-mavinkurve}.
    % It follows from \cref{n-equiv-quillen-pairs} that the rightmost map is a weak homotopy equivalence between Kan complexes.
    % Now, there is a natural map of chain complexes
\end{proof}

We will see later how cubical homology generalizes the discrete homology of a graph.

\subsection{Semicubical sets}

Let $\boxcati$ denote the subcategory of the box category consisting of the monomorphisms, i.e.\ composites of face maps.
A \emph{semicubical set} is a presheaf over $\boxcati$, i.e.\ a functor $\boxcati^\op \to \Set$.
We denote the category of semicubical sets by $\cSeti$.

The inclusion functor $i \from \boxcati \to \boxcat$ induces an adjoint triple:
\[ \begin{tikzcd}
    \cSeti \ar[r, bend left=50, "{i_!}"{name=T}] \ar[r, bend right=50, "{i_*}"'{name=B}] &[+1.8em] \cSet \ar[l, "{i^*}"{name=M, description}]
    \arrow[from=M, to=T, phantom, "\perp"]
    \arrow[from=B, to=M, phantom, "\perp"]
\end{tikzcd} \]
For our purposes, the only functor of relevance in this diagram is $i_! \from \cSeti \to \cSet$, though we note that the existence of $i^*$ implies $i_!$ preserves colimits (as a left adjoint).

We invoke \cref{fondind-presheaf-cat-equiv} to identify the category $\cSeti$ of semicubical sets with the subcategory of $\cSet$ given by \emph{fondind cubical sets} (see \cref{def:fondind}) and maps that preserve non-degenerate cubes.
\begin{proposition}[c.f.~\cref{fondind-presheaf-cat-equiv}] \label{cseti-cat-equiv}
    The functor $i_! \from \cSeti \to \cSet$ restricts to an equivalence between the category $\cSeti$ of semicubical sets and the subcategory of $\cSet$ whose:
    \begin{itemize}
        \item objects are the cubical sets such that every face of a non-degenerate cube is non-degenerate; and whose
        \item morphisms are the cubical maps that send non-degenerate cubes to non-degenerate cubes. \qed
    \end{itemize}
\end{proposition}
We note that every monomorphism sends non-degenerate cubes to non-degenerate cubes.

In light of \cref{cseti-cat-equiv}, we borrow all our notation for semicubical sets from the analogous notation for cubical sets.
In particular,
\begin{itemize}
    \item The notation $\cube{n}$ can refer to either the representable functor in $\cSet$ or the representable functor in $\cSeti$.
    In either case, we mean the functor represented by the object $[1]^n$.
    % This is justified since $i_!$ sends the representable on $[1]^n$ to the representable on $[1]^n$.
    \item The notation $\bd \cube{n}$ can refer to an object in $\cSet$ or $\cSeti$.
    In either case, we mean the subobject of the (respective) representable $\cube{n}$ which excludes the non-degenerate $n$-cube $\id[\cube{n}] \from \cube{n} \to \cube{n}$.
    \item The notation $\obox{n}{i, \eps}$ can refer to an object in $\cSet$ or $\cSeti$.
    In either case, we mean the subobject of $\bd \cube{n}$ (respectively) which further omits the non-degenerate $(n-1)$-cube $\face{}{i, \eps} \from \cube{n-1} \to \cube{n}$.
\end{itemize}
In all three cases, the functor $i_!$ sends the respective object in $\cSeti$ to the corresponding object in $\cSet$.

\begin{remark}
    To clarify a potential point of confusion, the functor $i^*$ is \textbf{not} an inverse to $i_!$, because $i_!$ does not admit an inverse that takes values on the entire category of cubical sets.
    Consequently, we \textbf{never} use this functor to translate a cubical set into a semicubical set; instead, we use the inverse functor described in \cref{fondind-presheaf-cat-equiv}.
\end{remark}

The notion of skeleton of a cubical set also makes sense for semicubical sets.
For $n \geq -1$, the $n$-skeleton of a semicubical set $X$ is the subobject whose set of $k$-cubes is given by
\[ (\Sk^n X)_k := \begin{cases}
    X_k & \text{if } k \leq n \\
    \varnothing & \text{otherwise}.
\end{cases} \]
As before, if $m \leq n$ then there is an evident inclusion $\Sk^m X \subseteq \Sk^n X$.
We also have an analogue of \cref{cubical-set-admits-skeletal-filtration}.
\begin{proposition} \label{semicubical-set-admits-skeletal-filtration}
    Let $X$ be a semicubical set.
    For $n \geq 0$, the inclusion $\Sk^{n-1} X \ito \Sk^n X$ forms the right arrow in a pushout square:
    \[ \begin{tikzcd}
        \displaystyle \coprod\limits_{x \in X_n} \bd \cube{n} \ar[r, "{\bd x}"] \ar[d, hook] & \Sk^{n-1} X \ar[d, hook] \\
        \displaystyle \coprod\limits_{x \in X_n} \cube{n} \ar[r, "{x}"] & \Sk^n X
    \end{tikzcd} \]
    % where $(K_n)_{\mathrm{nd}}$ denotes the set of non-degenerate $n$-cubes of $K$.
    Moreover, the diagram
    \[ \Sk^0 X \ito \Sk^1 X \ito \dots \ito \Sk^n X \ito \dots \ito X \]
    is a colimit diagram. \qed
\end{proposition}

\subsection{Nerves of graphs}

Graphs and cubical sets are related by a collection of ``nerve and realization'' adjunctions.
More precisely, for each $m \in \NN$, we define a cocubical object $\boxcat \to \Graph$ by sending $[1]^n$ to $\gcube{m}{n}$, the $n$-dimensional cube of length $m$.
Via extension by colimits, we obtain an adjunction
\[ \begin{tikzcd}
    \cSet \ar[r, bend left=20, "{\reali[m]{\uvar}}"{name=T}] \ar[r, bend right=20, "{\gnerve[m]}"'{name=B}, leftarrow] &[+1.8em] \Graph
    \arrow[from=B, to=T, phantom, "\perp"]
\end{tikzcd} \]
whose left adjoint is the \emph{$m$-realization} functor $\reali[m]{\uvar} \from \cSet \to \Graph$ and whose right adjoint is the \emph{$m$-nerve} functor $\gnerve[m] \from \Graph \to \cSet$.
An explicit formula for the $m$-nerve is given by
\[ (\gnerve[m] G)_n := \Graph(\gcube{m}{n}, G). \]

Each of the $m$-realization functors is monoidal.
\begin{proposition} \label{realization-is-monoidal}
    For $m \in \mathbb{N}$, the functor $\reali[m]{\uvar} \from \cSet \to \Graph$ is strong monoidal with respect to the geometric product of cubical sets and the box product of graphs.
\end{proposition}
This implies that their right adjoints $\gnerve[m]$ are lax monoidal, though we do not use this fact.

For $m \geq 1$, we have two maps between interval graphs $\ell, r \from I_{m+1} \to I_m$ which contract the first and last edge, respectively.
That is,
\[ \ell(i) := \begin{cases}
    0 & \text{if } i = 0 \\
    i-1 & \text{otherwise.}
\end{cases} \quad r(i) := \begin{cases}
    m & \text{if } i = m+1 \\
    i & \text{otherwise}.
\end{cases} \]
These maps induce maps $\ell^{\gtimes n}, r^{\gtimes n} \from \gcube{m+1}{n} \to \gcube{m}{n}$, which in turn induce natural transformations $\ell_!, r_! \from \reali[m+1]{\uvar} \Rightarrow \reali[m]{\uvar}$ and $\ell^*, r^* \from \gnerve[m] \Rightarrow \gnerve[m+1]$.
The \emph{nerve functor}, or \emph{$\infty$-nerve functor}, is defined to be the colimit
\[ \gnerve[\infty] := \colim \left( \gnerve[1] \xrightarrow{\ell^*} \gnerve[2] \xrightarrow{r^*} \gnerve[3] \xrightarrow{\ell^*} \dots \right) \]
in the category of functors $\Fun(\Graph, \cSet)$.

The fundamental property of graph nerves is the following:
\begin{theorem}[{\cite[Thm.~4.1]{carranza-kapulkin:cubical-graphs}}] \label{nerve-main-thm}
    For any graph $G$,
    \begin{enumerate}
        \item the nerve $\gnerve[\infty] G$ is a Kan complex; and
        \item for all $m \in \NN$, the maps $\ell^*, r^* \from \gnerve[m] G \to \gnerve[m+1] G$ are acyclic cofibrations in the Grothendieck model structure on $\cSet$.
    \end{enumerate}
\end{theorem}
\begin{corollary}[{c.f.\ \cite[Prop.~5.12]{carranza-kapulkin:cubical-graphs}}] \label{nerve-reflects-equivs}
    Let $f \from G \to H$ be a graph map.
    \begin{enumerate}
        \item The following are equivalent:
        \begin{itemize}
            \item $f$ is a weak homotopy equivalence;
            \item for all $m \in \NN \cup \{ \infty \}$, the map $\gnerve[m] f$ is a weak homotopy equivalence;
            \item there exists $m \in \NN \cup \{ \infty \}$ such that the map $\gnerve[m] f$ is a weak homotopy equivalence.
        \end{itemize}
        \item For $n \geq 0$, the following are equivalent:
        \begin{itemize}
            \item $f$ is a $n$-equivalence;
            \item for all $m \in \NN \cup \{ \infty \}$, the map $\gnerve[m] f$ is an $n$-equivalence;
            \item there exists $m \in \NN \cup \{ \infty \}$ such that the map $\gnerve[m] f$ is an $n$-equivalence.
        \end{itemize}
    \end{enumerate}
\end{corollary}
\begin{proof}
    By \cite[Thm.~4.6]{carranza-kapulkin:cubical-graphs}, the discrete homotopy groups of a graph coincide with the cubical homotopy groups of its nerve.
    This then follows from \cref{nerve-main-thm}.
\end{proof}
% Note that the cubical homology of a graph is known to be invariant under weak equivalence of cubical sets \cite{}.

An immediate observation one can make is that the homology of the cubical set $\gnerve[1] G$ recovers the discrete homology of the graph $G$.
That is, $H_n(G) \cong H_n(\gnerve[1] G)$.
With this, \cref{nerve-reflects-equivs} shows that discrete homology inverts weak homotopy equivalences of graphs.
\begin{corollary} \label{discrete-homology-inverts-equivs}
    Let $f \from G \to H$ be a graph map.
    \begin{enumerate}
        \item If $f$ is a weak equivalence then $H_n f \from H_k G \to H_k H$ is an isomorphism for all $n \geq 0$.
        \item Given $n \geq 0$, if $f$ is a $n$-equivalence then $H_k f \from H_k G \to H_k H$ is an isomorphism for all $k \in \{ 0, \dots, n \}$.
    \end{enumerate}
\end{corollary}
\begin{proof}
    Combine \cref{cubical-homology-inverts-equivs} with \cref{nerve-reflects-equivs}.
\end{proof}

Currently, it is an open problem to determine whether the category of graphs admits a model structure whose weak equivalences are the weak equivalences of graphs.
However, the $\infty$-nerve functor can be used to create a \emph{fibration category} structure on the category of graphs.

% \begin{definition}
%     Let $f \from X \to Y$ be a graph map.
%     \begin{enumerate}
%         \item We say $f$ is a \emph{weak equivalence} if, for all $m \in \NN \cup \{ \infty \}$, the map $\gnerve[m] f$ is a weak equivalence of cubical sets.
%         \item For $n \geq 0$, we say $f$ is an \emph{$n$-equivalence} if, for all $m \in \NN \cup \{ \infty \}$, the map $\gnerve[m] f$ is an $n$-equivalence of cubical sets.
%     \end{enumerate}
% \end{definition}
% By \cref{nerve-main-thm}, the map $\gnerve[m] f$ is a weak equivalence (or $n$-equivalence) for all $m \in \NN \cup \{ \infty \}$ if and only if it is a weak equivlaence (respectively, $n$-equivalence) for any single $m \in \NN \cup \{ \infty \}$.

\begin{definition}
    A \emph{fibration category} is a category $\cat{C}$ with two subcategories of \emph{fibrations} and \emph{weak equivalences} such that (in what follows, an \emph{acyclic fibration} is a map that is both a fibration and a weak equivalence):
    \begin{enumerate}
     \item weak equivalences satisfy the two-out-of-three property; that is, given two composable morphisms:
     \[ X \overset{f}\longrightarrow Y \overset{g}\longrightarrow Z \]
     if two of $f, g, gf$ are weak equivalences then all three are;
     \item all isomorphisms are acyclic fibrations;
     \item pullbacks along fibrations exist; fibrations and acyclic fibrations are stable under pullback;
     \item $\cat{C}$ has a terminal object 1; the canonical map $X \to 1$ is a fibration for any object $X \in \cat{C}$ (that is, all objects are \emph{fibrant});
     \item every map can be factored as a weak equivalence followed by a fibration.
    \end{enumerate}
\end{definition}
\begin{example}
    The category $\Kan$ of cubical (or simplicial) Kan complexes admits a fibration category stucture whose fibrations are the Kan fibrations and whose weak equivalences are the weak equivalences of cubical (or simplicial) sets.
\end{example}
\begin{definition}
    A functor $F \from \cat{C} \to \cat{D}$ between fibration categories is \emph{exact} if it preserves fibrations, acyclic fibrations, pullbacks along fibrations, and the terminal object.
\end{definition}
\begin{theorem} \label{graph-fibration-cat}
    The category $\Graph$ of graphs admits a fibration category structure where:
    \begin{itemize}
        \item fibrations are the maps $f \from G \to H$ which are sent to fibrations $\gnerve[\infty] f \from \gnerve[\infty] G \to \gnerve[\infty] H$ under the nerve functor; and
        \item weak equivalences are the weak homotopy equivalences of graphs.
    \end{itemize}
    Moreover, the nerve functor $\gnerve[\infty] \from \Graph \to \cSet$ is exact.
\end{theorem}
% Given a graph $G$ and $m \in \NN$, the counit map $G \to \reali[m]{\gnerve[m]{G}}$ is an example of an acyclic fibration.
An example of an acyclic fibration of graphs is the counit map $\reali[m]{\gnerve[m]{G}} \to G$.
\begin{proposition} \label{counit-acyclic-fib}
    For any graph $G$ and $m \in \NN$, the counit map $\eps_{G} \from \reali[m]{\gnerve[m]{G}} \to G$ is an acyclic fibration.
\end{proposition}
\begin{proof}
    It suffices to show $\gnerve[m](\eps_G) \from \gnerve[m]{ \big( \reali[m]{\gnerve[m]{G}} \big)} \to \gnerve[m]{G}$ is an acyclic fibration (cf.\ \cite[Thm.~5.25]{carranza-kapulkin:cubical-graphs}).
    To this end, fix a commutative square
    \[ \begin{tikzcd}
        \bd \cube{n} \ar[r] \ar[d, hook] & \gnerve[m]{\big( \reali[m]{\gnerve[m]{G}} \big)} \ar[d, "{\gnerve[m]{\eps_G}}"] \\
        \cube{n} \ar[r] & \gnerve[m]{G}
    \end{tikzcd} \]
    It suffices to find a lift of the transposed square
    \[ \begin{tikzcd}
        \reali[m]{\bd \cube{n}} \ar[r] \ar[d, hook] & {\reali[m]{\gnerve[m]{G}}} \ar[d, "{\eps_G}"] \\
        \reali[m]{\cube{n}} \ar[r] & G
    \end{tikzcd} \]
    % Note that $\gcube{m}{n} \cong \reali[m]{\cube{n}}$.
    The universal property of the counit $\eps_G$ yields a unique map $\cube{n} \to \gnerve[m]{G}$ whose $m$-realization makes the bottom triangle in
    \[ \begin{tikzcd}
        \reali[m]{\bd \cube{n}} \ar[r] \ar[d, hook] & {\reali[m]{\gnerve[m]{G}}} \ar[d, "{\eps_G}"] \\
        \reali[m]{\cube{n}} \ar[r] \ar[ur, dotted] & G
    \end{tikzcd} \]
    commute.
    The uniqueness can be used to deduce that the upper triangle commutes as well.
\end{proof}
\begin{corollary} \label{unit-graph-weq}
    For any graph $G$ and $m \in \NN$, the maps 
    \[ \eta_{\gnerve[m] G} \from \gnerve[m]{G} \to \gnerve[m]{\big( \reali[m]{\gnerve[m]{G}} \big)} \quad\text{and}\quad \reali[m]{\eta_{G}} \from \reali[m]{G} \to \reali[m]{\gnerve[m]{\reali[m]{G}}} \] 
    are weak equivalences.
\end{corollary}
\begin{proof}
    The counit of the realization-nerve adjunction is a weak equivalence by \cref{counit-acyclic-fib}.
    This now follows from 2-out-of-3 and the triangle identities. 
\end{proof}

A fundamental open problem in the field of discrete homotopy theory is the \emph{discrete homotopy hypothesis}, which says that the homotopy theories of graphs and topological spaces are equivalent, and that this equivalence is given by the nerve functor.
More precisely, recall that the \emph{homotopy category} of any category $\cat{C}$ with weak equivalences $W \subseteq \operatorname{Mor} \cat{C}$ is the universal category $\cat{C}[W^{-1}]$ equipped with a functor $\gamma \from \cat{C} \to \cat{C}[W^{-1}]$ that inverts weak equivalences.
In particular, the \emph{homotopy category} of a fibration category $\cat{C}$ (\cite[{\S}1.2]{brown}) refers to its localization at weak equivalences, which we denote by $\Ho \cat{C}$.
\begin{definition} \label{def:weq-fib-cat}
    An exact functor $F \from \cat{C} \to \cat{D}$ between fibration categories is a \emph{weak equivalence} of fibration categories if it induces an equivalence of homotopy categories $\Ho \cat{C} \to \Ho \cat{D}$.
\end{definition}
\begin{conjecture}[Discrete Homotopy Hypothesis] \label{conj:nerve}
    The nerve functor $\gnerve[\infty] \from \Graph \to \cSet$ is a weak equivalence of fibration categories.
\end{conjecture}
\begin{remark}
    For readers familiar with the theory of $(\infty, 1)$-categories, we highlight that if an exact functor $F \from \cat{C} \to \cat{D}$ between fibration categories is a weak equivalence in the sense of \cref{def:weq-fib-cat} then it is a Dwyer--Kan equivalence between categories with weak equivalences \cite[Cor.~7.6.11]{cisinski:higher-categories}.
    
    That is to say, if the nerve conjecture holds, then the functor $\gnerve[\infty] \from \Graph \to \cSet$ ascends to an equivalence between the $(\infty, 1)$-category of graphs localized at weak equivalences and the $(\infty, 1)$-category of spaces/$\infty$-groupoids.
    As a result, any theorem that holds in the $\infty$-category of spaces (e.g.\ the Blakers--Massey theorem, the Mayer--Vietoris long exact sequence, representability of cohomology) would have an analogue for graphs.
\end{remark}
This paper represents a step towards proving this conjecture.

Note that weak equivalences between fibrations categories can be characterized as those functors which satisfy the \emph{right approximation properties}.
\begin{definition}[{\cite[{\S}3.6]{cisinski:categories-derivables}}]
    An exact functor $F \from \cat{C} \to \cat{D}$ between fibration categories satisfies the \emph{right approximation properties} if:
    \begin{enumerate}
        \item[AP1.] $F$ reflects weak equivalences; and
        \item[AP2.] For any object $X \in \cat{C}$ and morphism $Y \to FX$ in $\cat{D}$, there exists a morphism $X' \to X$ in $\cat{C}$ which fits into a commutative diagram
        \[ \begin{tikzcd}
            Y \ar[r] & FX \\
            Y' \ar[u, "\sim"] \ar[r, "\sim"] & FX' \ar[u]
        \end{tikzcd} \]
        where both maps out of $Y'$ are weak equivalences.
    \end{enumerate}
\end{definition}
\begin{proposition}{\cite[Thm.~3.19]{cisinski:categories-derivables}} \label{AP-properties-suffices}
    An exact functor between fibration categories is a weak equivalence if and only if it satisfies the right approximation properties. \qed
\end{proposition}

For any fixed $n$, the $n$-equivalences in $\Graph$ and $\cSet$ both participate in fibration category structures, and the nerve functor $\gnerve[\infty]$ is once again an exact functor.

Given $n \geq 0$, let $\cSet_n$ denote the full subcategory of $\cSet$ consisting of those cubical sets which are fibrant in the transferred model structure for $n$-types.
Note that every Kan complex is contained in $\cSet_n$.
\begin{theorem}
    For $n\geq 0$, the category $\cSet_n$ of cubical Kan complexes admits a fibration category structure where:
    \begin{itemize}
        \item fibrations are the cubical maps $f \from X \to Y$ such that $\cosk^{n+1} f \from \cosk^{n+1} X \to \cosk^{n+1} Y$ is a Kan fibration; and
        \item weak equivalences are the $n$-equivalences.
    \end{itemize}
\end{theorem}
\begin{proof}
    This is the full subcategory of a model category consisting of the fibrant objects, hence forms a fibration category \cite[Ex.~1.1]{brown}. 
\end{proof}
\begin{theorem}[{\cite[Thm.~3.18]{kapulkin-mavinkurve:n-types}}]
    For $n \geq 0$, the category $\Graph$ of graphs admits a fibration category structure where:
    \begin{itemize}
        \item fibrations are the graph maps $f \from G \to H$ such that 
        \[ \cosk^{n+1} \gnerve[\infty] f \from \cosk^{n+1} \gnerve[\infty] G \to \cosk^{n+1} \gnerve[\infty] H \]
        is a Kan fibration; and
        \item weak equivalences are the $n$-equivalences of graphs.
    \end{itemize}
    Moreover, the nerve functor restricts to an exact functor $\gnerve[\infty] \from \Graph \to \cSet_n$. \qed
\end{theorem}
% Our results show that the nerve functor is essentially surjective after localizing at $n$-equivalences, for any $n$.
Our results make partial progress towards the discrete homotopy hypothesis by showing that the functor
\[ \Ho (\gnerve[\infty]) \from \Ho(\Graph) \to \Ho (\cSet_n) \]
between the two localizations at $n$-equivalences is an equivalence (\cref{nerve-equiv-main-thm}).

\section{Double mapping cylinders} \label{sec:mapping-cyl}
In this section, we discuss double mapping cylinders of graphs and cubical sets.
We begin by first identifying and ``abstracting away'' the key input needed for our construction.
This input is what we call an \emph{$n$-skeletal pushout} of graphs.
In brief, an $n$-skeletal pushout of graphs is a square that becomes a pushout after applying the $n$-skeleton functor to each 1-nerve.
The double mapping cylinder construction will provide our most prominent examples of $n$-skeletal pushouts of graphs.

% The ``combinatorial input'' to our proofs is centered around finding sufficient conditions for when a commutative square of graphs induces a homotopy pushout of $n$-types after applying the nerve functor.
% To this end, we introduce the notion of \emph{$n$-skeletal pushouts} of graphs.
% The key to proving our main theorem is a result that guarantees a pushout of graphs becomes a pushout of cubical sets after applying $\gnerve[1]$.
\subsection{$n$-Skeletal pushouts}

\begin{definition}
    Given $n \geq 0$, we say a commutative square
    \[ \begin{tikzcd}
        G \ar[r] \ar[d] & H \ar[d] \\
        K \ar[r] & L
    \end{tikzcd} \]
    in $\Graph$ is an \emph{$n$-skeletal pushout} if, after applying $\gnerve[1]$ and $\Sk^n$, the resulting square
    \[ \begin{tikzcd}
        \Sk^n \gnerve[1] G \ar[r] \ar[d] \ar[rd, phantom, "\potick" very near end] & \Sk^n \gnerve[1] H \ar[d] \\
        \Sk^n \gnerve[1] K \ar[r] & \Sk^n \gnerve[1] L
    \end{tikzcd} \]
    in $\cSet$ is a pushout.
\end{definition}
Let us record some basic facts about $n$-skeletal pushouts of graphs.
\begin{proposition}
    If
    \[ \begin{tikzcd}
        G \ar[r] \ar[d] & H \ar[d] \\
        K \ar[r] & L
    \end{tikzcd} \]
    is an $n$-skeletal pushout of graphs then it is an $m$-skeletal pushout for all $m \leq n$.
\end{proposition}
\begin{proof}
    Follows since $\Sk^m \cong \Sk^m \circ \Sk^n$, and $\Sk^m$ preserves pushouts.
\end{proof}
\begin{proposition} \label{0-skeletal-po-iff}
    A commutative square in $\Graph$
    \[ \begin{tikzcd}
        G \ar[r] \ar[d] & H \ar[d] \\
        K \ar[r] & L
    \end{tikzcd} \]
    is a 0-skeletal pushout if and only if it is a pushout after applying the vertex set functor $V \from \Graph \to \Set$.
\end{proposition}
\begin{proof}
    Follows since $\Sk^0 \from \cSet \to \cSet$ is naturally isomorphic to the composite of $V$ with the functor $D \from \Set \to \cSet$ that regards a set as a discrete cubical set.
    The functor $D$ is both conservative and cocontinuous, hence it preserves and reflects pushouts. 
\end{proof}
\begin{proposition} \label{1-skeletal-po-implies-po}
    If a commutative square in $\Graph$ is a 1-skeletal pushout then it is a pushout in $\Graph$.
\end{proposition}
\begin{proof}
    If the square
    \[ \begin{tikzcd}
        \Sk^1 \gnerve[1] G \ar[r] \ar[d] \ar[rd, phantom, "\potick" very near end] & \Sk^1 \gnerve[1] H \ar[d] \\
        \Sk^1 \gnerve[1] K \ar[r] & \Sk^1 \gnerve[1] L
    \end{tikzcd} \]
    is a pushout then the squares
    \[ \begin{tikzcd}
        \Graph(I_0, G) \ar[r] \ar[d] \ar[rd, phantom, "\potick" very near end] & \Graph(I_0, H) \ar[d] \\
        \Graph(I_0, K) \ar[r] & \Graph(I_0, L)
    \end{tikzcd} \qquad \begin{tikzcd}
        \Graph(I_1, G) \ar[r] \ar[d] \ar[rd, phantom, "\potick" very near end] & \Graph(I_1, H) \ar[d] \\
        \Graph(I_1, K) \ar[r] & \Graph(I_1, L)
    \end{tikzcd} \]
    are pushouts in $\Set$.
    It follows that the square 
    \[ \begin{tikzcd}
        G \ar[r] \ar[d] & H \ar[d] \\
        K \ar[r] & L
    \end{tikzcd} \]
    is a pushout in $\Graph$.
    % The 1-nerve admits a left adjoint $L \from \cSet \to \Graph$ which factors through the functor $\Sk^{1} \from \cSet \to \cSet$.
    % The counit of this adjunction is an isomorphism, from which it follows that if a square becomes a pushout after applying $\Sk^1 \gnerve[1]$ then the original square must be a pushout.
    % % Therefore, if a square becomes a pushout after applying $\Sk^1 \gnerve[1]$ then it 
\end{proof}
\begin{remark}
    The converse to \cref{1-skeletal-po-implies-po} is false in general.
    Below, we depict two pushouts of graphs
    
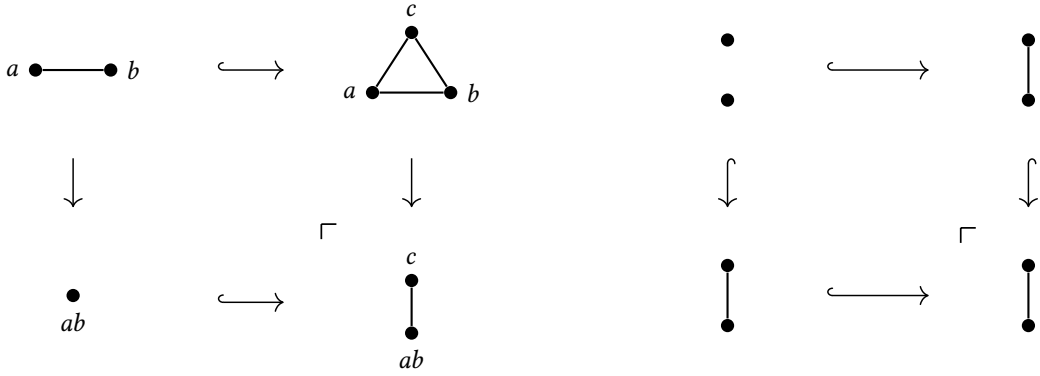
\[ \begin{tikzpicture}[commutative diagrams/.cd, every diagram]
        % Pushout 1
        % Top left graph
        \node[coordinate] (TLcenter) {};
        \path (TLcenter) -- +(-0.5, 0) node[vertex, label={west:$a$}] (TLa) {};
        \path (TLcenter) -- +(0.5, 0) node[vertex, label={east:$b$}] (TLb) {};
        \draw[thick] 
            (TLa) -- (TLb);
        
        % Top right graph
        \path (TLcenter) -- +(4.5, 0) node[coordinate] (TRcenter) {};
        \path (TRcenter) -- +(90:0.5) node[vertex, label={north:$c$}] (TRc) {};
        \path (TRcenter) -- +(210:0.6) node[vertex, label={west:$a$}] (TRa) {};
        \path (TRcenter) -- +(-30:0.6) node[vertex, label={east:$b$}] (TRb) {};
        \draw[thick]
            (TRa) -- (TRb)
            (TRa) -- (TRc)
            (TRb) -- (TRc);

        % Bottom left graph
        \path (TLcenter) -- +(0, -3) node[coordinate] (BLcenter) {};
        \node[vertex, label={south:$ab$}] (BLab) at (BLcenter) {};

        % Bottom right graph
        \path (TRcenter) -- +(0, -3) node[coordinate] (BRcenter) {};
        \path (BRcenter) -- +(0, -0.5) node[vertex, label={south:$ab$}] (BRab) {};
        \path (BRcenter) -- +(0, 0.2) node[vertex, label={north:$c$}] (BRc) {};
        \draw[thick]
            (BRc) -- (BRab);
        
        % Arrows for pushout 1
        \path (TLcenter) -- node[pos=0.4, draw=none] (TLTRstart) {} node[pos=0.65, draw=none] (TLTRend) {} (TRcenter);
        \path (TLcenter) -- node[pos=0.35, draw=none] (TLBLstart) {} node[pos=0.65, draw=none] (TLBLend) {} (BLcenter);
        \path (TRcenter) -- node[pos=0.35, draw=none] (TRBRstart) {} node[pos=0.65, draw=none] (TRBRend) {} (BRcenter);
        \path (BLcenter) -- node[pos=0.4, draw=none, yshift=-0.6ex] (BLBRstart) {} node[pos=0.65, draw=none, yshift=-0.6ex] (BLBRend) {} (BRcenter);
        \draw[commutative diagrams/.cd, every arrow, every label]
            (TLTRstart) edge [commutative diagrams/hook] (TLTRend)
            (TLBLstart) edge (TLBLend)
            (TRBRstart) edge (TRBRend)
            (BLBRstart) edge [commutative diagrams/hook] (BLBRend);
        % Pushout tick
        \path (BRcenter) -- +(-1.2, 0.95) node[coordinate] (corner) {};
        \path[draw, line width=0.65pt] 
            (corner) -- +(0.2, 0)
            (corner) -- +(0, -0.2);

        % Pushout 2
        % Top left
        \path (TRcenter) -- +(4.2, 0) node[coordinate] (TLcenterT) {};
        \path (TLcenterT) -- +(0, 0.4) node[vertex] (TLaT) {};
        \path (TLcenterT) -- +(0, -0.4) node[vertex] (TLbT) {};
        % \draw[thick]
        %     (TLaT) -- (TLbT);

        % Top right
        \path (TLcenterT) -- +(4, 0) node[coordinate] (TRcenterT) {};
        \path (TRcenterT) -- +(0, 0.4) node[vertex] (TRaT) {};
        \path (TRcenterT) -- +(0, -0.4) node[vertex] (TRbT) {};
        \draw[thick]
            (TRaT) -- (TRbT);

        % Bottom left
        \path (TLcenterT) -- +(0, -3) node[coordinate] (BLcenterT) {};
        \path (BLcenterT) -- +(0, 0.4) node[vertex] (BLaT) {};
        \path (BLcenterT) -- +(0, -0.4) node[vertex] (BLbT) {};
        \draw[thick]
            (BLaT) -- (BLbT);
        
        % Bottom right
        \path (TRcenterT) -- +(0, -3) node[coordinate] (BRcenterT) {};
        \path (BRcenterT) -- +(0, 0.4) node[vertex] (BRaT) {};
        \path (BRcenterT) -- +(0, -0.4) node[vertex] (BRbT) {};
        \draw[thick]
            (BRaT) -- (BRbT);
        
        % Arrows for pushout 2
        \path (TLcenterT) -- node[pos=0.3, draw=none] (TLTRstartT) {} node[pos=0.7, draw=none] (TLTRendT) {} (TRcenterT);
        \path (TLcenterT) -- node[pos=0.35, draw=none] (TLBLstartT) {} node[pos=0.65, draw=none] (TLBLendT) {} (BLcenterT);
        \path (TRcenterT) -- node[pos=0.35, draw=none] (TRBRstartT) {} node[pos=0.65, draw=none] (TRBRendT) {} (BRcenterT);
        \path (BLcenterT) -- node[pos=0.3, draw=none] (BLBRstartT) {} node[pos=0.7, draw=none] (BLBRendT) {} (BRcenterT);
        \draw[commutative diagrams/.cd, every arrow, every label]
            (TLTRstartT) edge [commutative diagrams/hook] (TLTRendT)
            (TLBLstartT) edge [commutative diagrams/hook] (TLBLendT)
            (TRBRstartT) edge [commutative diagrams/hook] (TRBRendT)
            (BLBRstartT) edge [commutative diagrams/hook] (BLBRendT);
        % Pushout tick
        \path (BRcenterT) -- +(-0.9, 0.9) node[coordinate] (cornerT) {};
        \path[draw, line width=0.65pt] 
            (cornerT) -- +(0.2, 0)
            (cornerT) -- +(0, -0.2);
    \end{tikzpicture} \]
    which are not 1-skeletal pushouts.
    Conceptually, this is explained by the fact that, after applying $\gnerve[1]$, the corresponding pushouts of cubical sets contain ``parallel edges'', i.e.\ distinct 1-cubes with the same source and target.
    One can show that if this does not occur (i.e. if the pushout $\Sk^1 H \push_{\Sk^1 G} \Sk^1 K$ contains no distinct 1-cubes with the same source and target) then the pushout of graphs $H \push_{G} K$ is a 1-skeletal pushout.
\end{remark}

The main benefit of $n$-skeletal pushouts is the following result.
\begin{theorem} \label{n-skeletal-po-is-hopo}
    Given $n \geq 0$, if
    \[ \begin{tikzcd}
        G \ar[r] \ar[d] & H \ar[d] \\
        K \ar[r] & L
    \end{tikzcd} \]
    is an $(n+1)$-skeletal pushout and either $G \to H$ or $G \to K$ is a monomorphism then, for any $m \in \NN \cup \{ \infty \}$, applying the functor $\gnerve[m]$ yields a homotopy pushout
    \[ \begin{tikzcd}
        \gnerve[m] G \ar[r] \ar[d] & \gnerve[m] H \ar[d] \\
        \gnerve[m] K \ar[r] & \gnerve[m] L
    \end{tikzcd} \]
    of $n$-types.
\end{theorem}
\begin{proof}
    For $m \leq m'$ in $\NN \cup \{ \infty \}$, there are natural weak equivalences $\gnerve[m] \ito \gnerve[m']$, so it suffices to consider the case $m = 1$.
    By assumption, the square 
    \[ \begin{tikzcd}
        \Sk^{n+1} \gnerve[1] G \ar[r] \ar[d] & \Sk^{n+1} \gnerve[1] H \ar[d] \\
        \Sk^{n+1} \gnerve[1] K \ar[r] & \Sk^{n+1} \gnerve[1] L
    \end{tikzcd} \]
    is a pushout along a monomorphism, hence a homotopy pushout in the Grothendieck model structure.
    This implies it is a homotopy pushout in the Cisinski model structure for $n$-types.
    Since the natural inclusion $\Sk^{n+1} \Rightarrow \id[\cSet]$ is an $n$-equivalence \cite[Lem.~2.22]{kapulkin-mavinkurve:n-types}, the square
    \[ \begin{tikzcd}
        \gnerve[1] G \ar[r] \ar[d] & \gnerve[1] H \ar[d] \\
        \gnerve[1] K \ar[r] & \gnerve[1] L
    \end{tikzcd} \]
    is a homotopy pushout of $n$-types.
\end{proof}

In practice, finding conditions for when a square of graphs is an $n$-skeletal pushout is rather difficult.
To the best knowledge of the authors, the following result encapsulates all known examples of $n$-skeletal pushouts.
\begin{lemma} \label{n-skeletal-po-condition-general}
    Let
    \[ \begin{tikzcd}
        G \ar[r] \ar[d] & H \ar[d] \\
        K \ar[r] & L
    \end{tikzcd} \]
    be a commutative square of graphs and $n \geq 0$ be a non-negative integer.
    If
    \begin{enumerate}
        \item every morphism in the square is a monomorphism;
        \item the square is a pullback in $\Graph$; and
        \item every morphism $\gcube{1}{n} \to L$ from an $n$-cube to $L$ factors through either $H \to L$ or $K \to L$,
    \end{enumerate}
    then this square is an $n$-skeletal pushout.
\end{lemma}
\begin{proof}
    To begin, we first prove the following strengthening of condition (3)
    \begin{itemize}
        \item[3'.] For all $k \in \{ 0, \dots, n \}$, every morphism $\gcube{1}{k} \to L$ factors through either $H \to L$ or $K \to L$.
    \end{itemize}
    To prove this, fix $k$ and consider such a morphism $u \from \gcube{1}{k} \to L$.
    Let $\pi \from \gcube{1}{n} \to \gcube{1}{k}$ denote the projection to the first $k$ coordinates
    \[ \pi(x_1, \dots, x_n) := (x_1, \dots, x_k). \]
    By condition (3), the composite $u \pi \from \gcube{1}{n} \to L$ factors through either $H$ or $K$.
    Without loss of generality, assume it factors through $H$.
    We now obtain a commutative square
    \[ \begin{tikzcd}
        \gcube{1}{n} \ar[r] \ar[d, "\pi"'] & H \ar[d, hook] \\
        \gcube{1}{k} \ar[r, "u"] & L
    \end{tikzcd} \]
    The product projection $\pi$ is a split epimorphism and the right map is a monomorphism, hence this square admits a lift, as desired.

    Returning to the theorem statement, it suffices to show that the commutative square
    \[ \begin{tikzcd}
        {\Graph(\gcube{1}{k}, G)} \ar[r, hook] \ar[d, hook] & {\Graph(\gcube{1}{k}, H)} \ar[d, hook] \\
        {\Graph(\gcube{1}{k}, K)} \ar[r, hook] & {\Graph(\gcube{1}{k}, L)}
    \end{tikzcd} \]
    induced by post-composition is a pushout in $\Set$ for all $k \in \{ 0, \dots, n \}$.
    That is, we will show the function
    \[ \varphi \from \Graph(\gcube{1}{k}, H) \push_{\Graph(\gcube{1}{k}, G)} \Graph(\gcube{1}{k}, K) \to \Graph(\gcube{1}{k}, L) \]
    is a bijection.

    By condition (3'), $\varphi$ is surjective.
    To verify injectivity, fix two elements $f, g$ in the left set such that $\varphi(f) = \varphi(g)$.
    Recall that the functions
    \[ \begin{tikzcd}[column sep=0em]
        \Graph(\gcube{1}{k}, H) \ar[rd, "\lambda_H"] & {} &  \Graph(\gcube{1}{k}, K) \ar[ld, "\lambda_K"'] \\
        {} & \Graph(\gcube{1}{k}, H) \push_{\Graph(\gcube{1}{k}, G)} \Graph(\gcube{1}{k}, K) & {}
    \end{tikzcd} \]
    % \[ \Graph(\gcube{1}{k}, Y) \to \Graph(\gcube{1}{k}, X) \push_{\Graph(\gcube{1}{k}, X)} \Graph(\gcube{1}{k}, Z) \text{ and } \Graph(\gcube{1}{k}, Z) \to \Graph(\gcube{1}{k}, X) \push_{\Graph(\gcube{1}{k}, X)} \Graph(\gcube{1}{k}, Z) \]
    are jointly surjective.
    We proceed by case analysis:
    \begin{itemize}
        \item Suppose $f$ and $g$ are both in the image of $\lambda_H$.
        This function fits into a commutative triangle 
        \[ \begin{tikzcd}
            \Graph(\gcube{1}{k}, H) \ar[r, "\lambda_H"] \ar[rd, hook] & {\Graph(\gcube{1}{k}, H) \push_{\Graph(\gcube{1}{k}, G)} \Graph(\gcube{1}{k}, K)} \ar[d, "\varphi"] \\
            {} & {\Graph(\gcube{1}{k}, L)}
        \end{tikzcd} \]
        The composite map is injective by condition (1), from which it follows that $f = g$.
        \item Suppose $f$ and $g$ are both in the image of $\lambda_K$.
        This proceeds analogously to the previous case.
        \item Suppose $f$ is in the image of $\lambda_H$ but not $\lambda_K$, and $g$ is in the image of $\lambda_K$ but not $\lambda_H$.
        We will prove this case leads to a contradiction.

        Let $\overline{f} \from \gcube{1}{k} \to H$ and $\overline{g} \from \gcube{1}{k} \to K$ be maps such that $\lambda_H(\overline{f}) = f$ and $\lambda_K (\overline{g}) = g$.
        By assumption, the square
        \[ \begin{tikzcd}
            \gcube{1}{k} \ar[r, "\overline{f}"] \ar[d, "\overline{g}"'] & H \ar[d, hook] \\
            K \ar[r, hook] & L
        \end{tikzcd} \]
        commutes.
        Condition (2) asserts that $G$ is a pullback, hence both $\overline{f}$ and $\overline{g}$ must factor through $G$.
        This contradicts our assumptions that $f$ is not in the image of $\lambda_K$ and $g$ is not in the image of $\lambda_H$. \qedhere
    \end{itemize}
\end{proof}

We highlight some example applications of \cref{n-skeletal-po-condition-general}.
Recall that the \emph{distance} between two vertices $v$ and $w$ in a graph is the length of the shortest path from $v$ to $w$.
If there is no path from $v$ to $w$ then we declare their distance to be $\infty$.
% In practice, we will often invoke the following corollary of \cref{n-skeletal-po-condition-general}, which provides a convenient rephrasing of some of the technical conditions while retaining most of the generality.
\begin{corollary} \label{n-skeletal-po-condition-subgraphs}
    Let $n \geq 0$ be a non-negative integer.
    Suppose $A$ and $B$ are induced subgraphs of a graph $G$ satisfying the following conditions:
    \begin{itemize}
        \item $A \cup B = G$; and
        \item for any two vertices $v, w \in X$, if $v \not\in A$ and $w \not\in B$ then the distance from $v$ to $w$ is at least $n+1$.
    \end{itemize}
    Then, the square
    \[ \begin{tikzcd}
        A \cap B \ar[r, hook] \ar[d, hook] & A \ar[d, hook] \\
        B \ar[r, hook] & G
    \end{tikzcd} \]
    is an $n$-skeletal pushout.
\end{corollary}
\begin{proof}
    We need only show this square satisfies the assumptions of \cref{n-skeletal-po-condition-general}.
    Conditions (1) and (2) are immediate.
    For (3), let $u \from \gcube{1}{n} \to G$ be a map.
    Observe that any two vertices in the image of $u$ are connected by a path of length at most $n$.
    Thus, the image of $u$ is entirely contained in one of $A$ or $B$.
    As $A$ and $B$ are induced subgraphs, we have the desired factorization of $u$.
\end{proof}

\subsection{Mapping cylinders}

Recall that, in the ``standard'' homotopy theory of spaces, homotopy pushouts are constructed via the double-mapping cylinder, a construction which takes two continuous functions $f \from S \to T$ and $g \from S \to U$ and ouputs the space obtained by taking a cylinder on $S$ and gluing a copy of $T$ and $U$ along the ends of cylinder, with $f$ and $g$ playing the role of the attaching map.
An analogous construction can be carried out in an arbitrary model category to construct homotopy pushouts.

In this subsection, we develop the notion of double-mapping cylinders in the category of graphs.
We will see that double-mapping cylinders of sufficient length yield examples of $n$-skeletal pushouts of graphs, and these examples are crucial to our understanding of the homotopy theory of $n$-types in the category of graphs.

Currently, there is no known model structure (or cofibration category structure) on the category of graphs whose weak equivalences are the $n$-equivalences, nor is there one whose weak equivalences are the weak homotopy equivalences.
As such, many standard properties of double mapping cylinders must be proven ``from scratch'' in our setting.

\begin{definition} \label{def:double-mapping-cylinder}
    Let $f \from G \to H$ and $g \from G \to K$ be graph maps.
    Given a non-negative integer $m \geq 0$, the \emph{double mapping cylinder} of $f$ and $g$ of length $m$ is the colimit:
    \[ \Cyl{m}(f, g) := \colim \left( \begin{tikzcd}
        {} & G \ar[ld, "f"'] \ar[rd, "i_0", hook] & {} & G \ar[ld, "i_m", hook, swap] \ar[rd, "g"] & {} \\
        H & {} & G \gtimes I_m & {} & K
    \end{tikzcd} \right) \]
\end{definition}
Explicitly, $\Cyl{m}(f, g)$ is the quotient of the disjoint union
\[ \Cyl{m}(f, g) = \, H \sqcup (G \gtimes I_m) \sqcup K \, \Big/ \sim \]
given by the identifications $(v, 0) \sim f(v)$ and $(v, n) \sim g(v)$.
% In particular, if $n = 0$ then $\Cyl{0}(f, g)$ is the pushout $Y \push_{X} Z$ of $f$ and $g$.
\begin{example}
    We summarize the following ``edge case'' examples.
    \begin{enumerate}
        \item If $n = 0$ then $\Cyl{0}(f, g)$ is the pushout $H \push_{G} K$ of $f$ and $g$.
        \item If $f = \id[G] = g$ then $\Cyl{m}(\id[G], \id[G])$ is the cylinder $G \gtimes I_m$ on $G$.
        In this case, the structure morphism $G \gtimes I_m \to \Cyl{m}(\id[G], \id[G])$ is an isomorphism.
        \item The double-mapping cylinder is symmetric in its arguments up to isomorphism.
        The isomorphism is given by the formula
        \[ \begin{array}{c@{\qquad}c@{\qquad}c}
            (v_G, t) \mapsto (v_G, n-t) & v_H \mapsto v_H & v_K \mapsto v_K
        \end{array} \]
        defined on the disjoint union $(G \gtimes I_m) \sqcup H \sqcup K$.
        \item If $f = \id[G]$ and $g$ is the unique map $\bang \from G \to I_0$ then $\Cyl{m}(\id[G], \bang)$ is the cone on $G$ of length $n$, i.e.\ the quotient
        \[ \big( G \gtimes I_m \big) \Big/ (v, n) \sim (w, n) \text{ for all } v, w \in G . \]
    \end{enumerate}
\end{example}
% \begin{example}
%     If $n = 0$ then $\Cyl{m}(f, g)$ recovers with the pushout $Y \push_{X} Z$ of $f$ and $g$.
% \end{example}
% \begin{example}
%     If $f$ and $g$ are both the identity then $\Cyl{m}(f, g)$ is isomorphic to the cylinder $G \gtimes I_m$ on $X$.
%     More precisely, this isomorphism is given by the structure morphism $G \gtimes I_m \to \Cyl{m}(\id[G], \id[G])$.
% \end{example}
\begin{example} \label{ex:boundary-as-cyl}
    Recall the \emph{boundary} of the $n$-cube of length $m$ is the subgraph of $\gcube{m}{n}$ defined by
    \[ \bd \gcube{m}{n} := \{ (x_1, \dots, x_n) \in \gcube{m}{n} \mid x_i = 0 \text{ or } m \text{ for some $i$} \}. \]
    An edge $(x_1, \dots, x_n) \sim (y_1, \dots, y_n)$ in $\gcube{m}{n}$ is contained in the subgraph $\bd \gcube{m}{n}$ if either $x_i = 0 = y_i$ or $x_i = m = y_i$ for some $i$ (this makes $\bd \gcube{m}{n}$ an induced subgraph in all cases except $m = n = 1$).
    Writing $i^n_m \from \bd \gcube{m}{n} \to \gcube{m}{n}$ for the inclusion map, we have an isomorphism $\Cyl{m}(i^n_m, i^n_m) \cong \bd \gcube{m}{n+1}$.

    Similarly, the \emph{$n$-dimensional open box} of length $m$ is the induced subgraph $\obox{n}{m} \subseteq \bd \gcube{m}{n}$ defined by
    \[ \obox{n}{m} := \bd \gcube{m}{n} - \big\{ (x_1, \dots, x_n) \mid x_n = m \text{ and } x_i \neq 0, m \text{ for all } i \neq n \big\}. \]
    We now have an isomorphism $\Cyl{m}(i^n_m, \id[\bd \gcube{m}{n}]) \cong \obox{n+1}{m}$.

    \begin{figure}[H]
        \centering
        \subfloat[The graph $\bd \gcube{2}{2}$.]{
            \begin{tikzpicture}[scale=0.6]
                \node[vertex, minimum size=4pt] at (0, 0) {};
                \node[vertex, minimum size=4pt] at (1, 0) {};
                \node[vertex, minimum size=4pt] at (2, 0) {};
                \node[vertex, minimum size=4pt] at (0, 1) {};
                % \node[vertex] at (1, 1) {};
                \node[vertex, minimum size=4pt] at (2, 1) {};
                \node[vertex, minimum size=4pt] at (0, 2) {};
                \node[vertex, minimum size=4pt] at (1, 2) {};
                \node[vertex, minimum size=4pt] at (2, 2) {};
                
                \draw
                    (0, 0) to (1, 0)
                    (1, 0) to (2, 0)
                    (0, 0) to (0, 1)
                    (0, 1) to (0, 2)
                    (0, 2) to (1, 2)
                    (1, 2) to (2, 2)
                    (2, 2) to (2, 1)
                    (2, 1) to (2, 0);
                
                % This artificially increases the size of this picture/float
                \path[draw=none]
                    (-4, -1) rectangle (6, 3);
            \end{tikzpicture}
        }
        \subfloat[The graph ${\reali[2]{\obox{2}{}}}$.]{
            \begin{tikzpicture}[scale=0.6]
                \node[vertex, minimum size=4pt] at (0, 0) {};
                % \node[vertex, minimum size=4pt] at (1, 0) {};
                \node[vertex, minimum size=4pt] at (2, 0) {};
                \node[vertex, minimum size=4pt] at (0, 1) {};
                % \node[vertex] at (1, 1) {};
                \node[vertex, minimum size=4pt] at (2, 1) {};
                \node[vertex, minimum size=4pt] at (0, 2) {};
                \node[vertex, minimum size=4pt] at (1, 2) {};
                \node[vertex, minimum size=4pt] at (2, 2) {};
                \draw
                    (0, 0) to (0, 1)
                    (0, 1) to (0, 2)
                    (0, 2) to (1, 2)
                    (1, 2) to (2, 2)
                    (2, 2) to (2, 1)
                    (2, 1) to (2, 0);
                
                % This artificially increases the size of this picture/float
                \path[draw=none]
                    (-4, -1) rectangle (6, 3);
            \end{tikzpicture}
        }

        \vspace{2em}

        \subfloat[The graph $\bd \gcube{2}{3}$.]{
            \begin{tikzpicture}[scale=1.5]
                % Back layer
                \node[vertex, minimum size=4pt] (ATL) at (0, 0) {};
                \path (ATL) -- +(1, 0) node[vertex, minimum size=4pt] (AT) {};
                \path (AT) -- +(1, 0) node[vertex, minimum size=4pt] (ATR) {};
                \path (ATL) -- +(0, -1) node[vertex, minimum size=4pt] (AL) {};
                \path (AL) -- +(1, 0) node[vertex, minimum size=4pt] (AC) {};
                \path (AC) -- +(1, 0) node[vertex, minimum size=4pt] (AR) {};
                \path (AL) -- +(0, -1) node[vertex, minimum size=4pt] (ABL) {};
                \path (ABL) -- +(1, 0) node[vertex, minimum size=4pt] (AB) {};
                \path (AB) -- +(1, 0) node[vertex, minimum size=4pt] (ABR) {};
                \draw
                    (ATL) to (ATR)
                    (AL) to (AR)
                    (ABL) to (ABR)
                    (ATL) to (ABL)
                    (AT) to (AB)
                    (ATR) to (ABR);

                % Middle layer
                \path (ATL) -- +(0.3, -0.3) node[vertex, minimum size=4pt] (BTL) {};
                \path (BTL) -- +(1, 0) node[vertex, minimum size=4pt] (BT) {};
                \path (BT) -- +(1, 0) node[vertex, minimum size=4pt] (BTR) {};
                \path (BTL) -- +(0, -1) node[vertex, minimum size=4pt] (BL) {};
                \path (BL) -- +(1, 0) node[vertex, fill=none, minimum size=4pt] (BC) {};
                \path (BC) -- +(1, 0) node[vertex, minimum size=4pt] (BR) {};
                \path (BL) -- +(0, -1) node[vertex, minimum size=4pt] (BBL) {};
                \path (BBL) -- +(1, 0) node[vertex, minimum size=4pt] (BB) {};
                \path (BB) -- +(1, 0) node[vertex, minimum size=4pt] (BBR) {};
                \draw
                    (BTL) edge[line width=4pt, white] (BTR) % extra white border around this line
                    (BTL) edge (BTR)
                    (BTR) to (BBR)
                    (BBR) to (BBL)
                    (BBL) edge[line width=4pt, white] (BTL)% extra white border around this line
                    (BBL) edge (BTL);
                % Redraw vertices that are whited out by the extra borders
                \node[vertex, minimum size=4pt] at (BT) {};
                \node[vertex, minimum size=4pt] at (BL) {};

                % Back layer
                \path (BTL) -- +(0.3, -0.3) node[vertex, minimum size=4pt] (CTL) {};
                \path (CTL) -- +(1, 0) node[vertex, minimum size=4pt] (CT) {};
                \path (CT) -- +(1, 0) node[vertex, minimum size=4pt] (CTR) {};
                \path (CTL) -- +(0, -1) node[vertex, minimum size=4pt] (CL) {};
                \path (CL) -- +(1, 0) node[vertex, minimum size=4pt] (CC) {};
                \path (CC) -- +(1, 0) node[vertex, minimum size=4pt] (CR) {};
                \path (CL) -- +(0, -1) node[vertex, minimum size=4pt] (CBL) {};
                \path (CBL) -- +(1, 0) node[vertex, minimum size=4pt] (CB) {};
                \path (CB) -- +(1, 0) node[vertex, minimum size=4pt] (CBR) {};
                \draw
                    (CTL) edge [line width=4pt, white] (CTR) % extra white border around this line
                    (CTL) edge (CTR)
                    (CL) edge [line width=4pt, white] (CR) % extra white border around this line
                    (CL) edge (CR)
                    (CBL) to (CBR)
                    (CTL) edge [line width=4pt, white] (CBL) % extra white border around this line
                    (CTL) edge (CBL)
                    (CT) edge [line width=4pt, white] (CB) % extra white border around this line
                    (CT) edge (CB)
                    (CTR) to (CBR);
                % Redraw vertices that are whited out by the extra borders
                \node[vertex, minimum size=4pt] at (CT) {};
                \node[vertex, minimum size=4pt] at (CL) {};
                \node[vertex, minimum size=4pt] at (CC) {};
                
                % Edges between layers
                \draw 
                    (ATL) to (CTL)
                    (AT) to (CT)
                    (ATR) to (CTR)
                    (AL) to (CL)
                    (AR) to (CR)
                    (ABL) to (CBL)
                    (AB) to (CB)
                    (ABR) to (CBR);
                
                % % This artificially increases the size of this picture/float
                % \path[draw=red]
                %     (-5, -1) rectangle (7, 3);
            \end{tikzpicture}
        }
        \hspace{4.9em}
        \subfloat[The graph ${\obox{3}{2}}$.]{
            \begin{tikzpicture}[scale=1.5]
                % Back layer
                \node[vertex, minimum size=4pt] (ATL) at (0, 0) {};
                \path (ATL) -- +(1, 0) node[vertex, minimum size=4pt] (AT) {};
                \path (AT) -- +(1, 0) node[vertex, minimum size=4pt] (ATR) {};
                \path (ATL) -- +(0, -1) node[vertex, minimum size=4pt] (AL) {};
                \path (AL) -- +(1, 0) node[vertex, minimum size=4pt] (AC) {};
                \path (AC) -- +(1, 0) node[vertex, minimum size=4pt] (AR) {};
                \path (AL) -- +(0, -1) node[vertex, minimum size=4pt] (ABL) {};
                \path (ABL) -- +(1, 0) node[vertex, minimum size=4pt] (AB) {};
                \path (AB) -- +(1, 0) node[vertex, minimum size=4pt] (ABR) {};
                \draw
                    (ATL) to (ATR)
                    (AL) to (AR)
                    (ABL) to (ABR)
                    (ATL) to (ABL)
                    (AT) to (AB)
                    (ATR) to (ABR);

                % Middle layer
                \path (ATL) -- +(0.3, -0.3) node[vertex, minimum size=4pt] (BTL) {};
                \path (BTL) -- +(1, 0) node[vertex, minimum size=4pt] (BT) {};
                \path (BT) -- +(1, 0) node[vertex, minimum size=4pt] (BTR) {};
                \path (BTL) -- +(0, -1) node[vertex, minimum size=4pt] (BL) {};
                \path (BL) -- +(1, 0) node[vertex, fill=none, minimum size=4pt, fill=none] (BC) {};
                \path (BC) -- +(1, 0) node[vertex, minimum size=4pt] (BR) {};
                \path (BL) -- +(0, -1) node[vertex, minimum size=4pt] (BBL) {};
                \path (BBL) -- +(1, 0) node[vertex, minimum size=4pt, fill=none] (BB) {};
                \path (BB) -- +(1, 0) node[vertex, minimum size=4pt] (BBR) {};
                \draw
                    (BTL) edge[line width=4pt, white] (BTR) % extra white border around this line
                    (BTL) edge (BTR)
                    (BTR) to (BBR)
                    % (BBR) to (BBL)
                    % (BBL) edge[line width=4pt, white] (BTL)% extra white border around this line
                    (BBL) edge (BTL);
                % Redraw vertices that are whited out by the extra borders
                \node[vertex, minimum size=4pt] at (BT) {};
                \node[vertex, minimum size=4pt] at (BL) {};

                % Back layer
                \path (BTL) -- +(0.3, -0.3) node[vertex, minimum size=4pt] (CTL) {};
                \path (CTL) -- +(1, 0) node[vertex, minimum size=4pt] (CT) {};
                \path (CT) -- +(1, 0) node[vertex, minimum size=4pt] (CTR) {};
                \path (CTL) -- +(0, -1) node[vertex, minimum size=4pt] (CL) {};
                \path (CL) -- +(1, 0) node[vertex, minimum size=4pt] (CC) {};
                \path (CC) -- +(1, 0) node[vertex, minimum size=4pt] (CR) {};
                \path (CL) -- +(0, -1) node[vertex, minimum size=4pt] (CBL) {};
                \path (CBL) -- +(1, 0) node[vertex, minimum size=4pt] (CB) {};
                \path (CB) -- +(1, 0) node[vertex, minimum size=4pt] (CBR) {};
                \draw
                    (CTL) edge [line width=4pt, white] (CTR) % extra white border around this line
                    (CTL) edge (CTR)
                    (CL) edge [line width=4pt, white] (CR) % extra white border around this line
                    (CL) edge (CR)
                    (CBL) to (CBR)
                    (CTL) edge [line width=4pt, white] (CBL) % extra white border around this line
                    (CTL) edge (CBL)
                    (CT) edge [line width=4pt, white] (CB) % extra white border around this line
                    (CT) edge (CB)
                    (CTR) to (CBR);
                % Redraw vertices that are whited out by the extra borders
                \node[vertex, minimum size=4pt] at (CT) {};
                \node[vertex, minimum size=4pt] at (CL) {};
                \node[vertex, minimum size=4pt] at (CC) {};
                
                % Edges between layers
                \draw 
                    (ATL) to (CTL)
                    (AT) to (CT)
                    (ATR) to (CTR)
                    (AL) to (CL)
                    (AR) to (CR)
                    (ABL) to (CBL)
                    % (AB) to (CB)
                    (ABR) to (CBR);

                % % This artificially increases the size of this picture/float
                % \path[draw=red]
                % (-5, -1) rectangle (7, 3);
            \end{tikzpicture}
        }
        \caption{The double mapping cylinders described in \cref{ex:boundary-as-cyl}.}
    \end{figure}
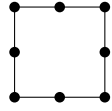
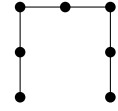
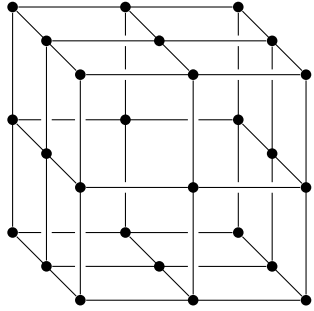
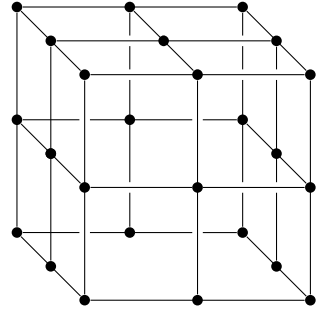
\end{example}
\begin{example}
    Let $\bang \from G \to I_0$ denote the unique map.
    Given $m \geq 0$, the {suspension of length $m$} of a graph $G$, denoted $\susp[m]{G}$, is the double mapping cylinder $\Cyl{m}(\bang , \bang)$.
    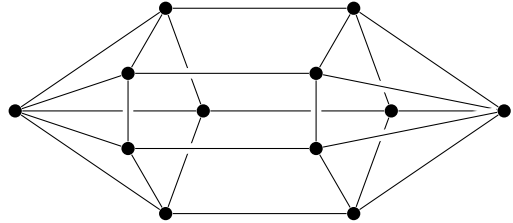
\begin{figure}[H]
        \centering
        \begin{tikzpicture}
            % A layer
            \node[vertex] (A1) {};
            \path (A1) -- +(-90:1) node[vertex] (A2) {};
            \path (A2) -- +(-60:1) node[vertex] (A3) {};
            \path (A1) -- +(60:1) node[vertex] (A5) {};
            \path (A5) -- node[coordinate] (Acenter) {} (A3);
            \path (Acenter) -- +(0:0.5) node[vertex] (A4) {};
            \draw (A2) -- (A3) -- (A4) -- (A5) -- (A1); % We draw (A1) to (A2) later
            % B layer
            \path (A1) -- +(2.5, 0) node[vertex] (B1) {};
            \path (B1) -- +(-90:1) node[vertex] (B2) {};
            \path (B2) -- +(-60:1) node[vertex] (B3) {};
            \path (B1) -- +(60:1) node[vertex] (B5) {};
            \path (B5) -- node[coordinate] (Bcenter) {} (B3);
            \path (Bcenter) -- +(0:0.5) node[vertex] (B4) {};
            \draw (B2) -- (B3) -- (B4) -- (B5) -- (B1); % We draw (B1) -- (B2) later
            % North pole
            \path (Acenter) -- +(-2, 0) node[vertex] (N) {}; 
            % South pole
            \path (Bcenter) -- +(2, 0) node[vertex] (S) {}; 

            % North pole edges
            \draw
                (N) -- (A1)
                (N) -- (A2)
                (N) -- (A3)
                (N) -- (A4)
                (N) -- (A5);
            % South pole edges
            \draw
                (S) edge[line width=4pt, white] (B1)
                (S) edge (B1)
                (S) edge[line width=4pt, white] (B2)
                (S) edge (B2)
                (S) edge (B3)
                (S) -- (B4)
                (S) edge (B5);
            % In between layer edges
            \draw
                (A1) edge[line width=4pt, white] (B1)
                (A1) edge (B1)
                (A2) edge[line width=4pt, white] (B2)
                (A2) edge (B2)
                (A3) -- (B3)
                (A4) -- (B4)
                (A5) -- (B5);
            % Redraw (A1) -- (A2) and (B1) -- (B2)
            \draw
                (A1) edge[line width=4pt, white] (A2)
                (A1) edge (A2)
                (B1) edge[line width=4pt, white] (B2)
                (B1) edge (B2);
        \end{tikzpicture}
        \caption{The length-3 suspension $\susp[3]{C_5}$ of the 5-cycle.}
    \end{figure}
\end{example}

Let $\ell_0(f, g)$ and $r_0(f, g)$ denote the structure morphisms
\[ \ell_0(f, g) \from H \to \Cyl{m}(f, g) \quad r_0(f, g) \from K \to \Cyl{m}(f, g). \]
% We observe the following isomorphisms.
% \begin{proposition}
%     For any $n \geq 0$, the map $$
% \end{proposition}
These can be used to write the double mapping cylinder as a pushout.
\begin{proposition} \label{cyl-as-po}
    The squares
    \[ \begin{tikzcd}[column sep = 6em, row sep = 4.5em]
        G \sqcup G \ar[r, "{[i_0, I_m]}"] \ar[d, "{f \, \sqcup \, g}"'] & G \gtimes I_m \ar[d] \\
        H \sqcup K \ar[r, "{[\ell_0(f, g), r_0(f, g)]}"] & \Cyl{m}(f, g)
    \end{tikzcd} 
    \quad \begin{tikzcd}[column sep = 3.8em, row sep = 3.2em]
        G \ar[d, "f"'] \ar[r, "{\ell_0(\id[G], g)}"] & \Cyl{m}(\id[G], g) \ar[d] \\
        H \ar[r, "{\ell_0(f, g)}"] & {\Cyl{m}(f, g)}
    \end{tikzcd} 
    \quad \begin{tikzcd}[column sep = 3.8em, row sep = 3.2em]
        G \ar[d, "g"'] \ar[r, "{r_0(f, \id[G])}"] & \Cyl{m}(f, \id[G]) \ar[d] \\
        K \ar[r, "{r_0(f, g)}"] & {\Cyl{m}(f, g)}
    \end{tikzcd} 
    % \quad \begin{tikzcd}[column sep = 3.8em, row sep = 3.2em]
    %     X \ar[r, "I_m"] \ar[d, "g"'] & G \gtimes I_m \ar[d] \\
    %     Z \ar[r, "{r_0(\id[G], g)}"] & {\Cyl{m}(\id[G], g)}
    % \end{tikzcd}
    \]
    are pushouts.
\end{proposition}
\begin{proof}
    Use the explicit description of pushouts in $\Graph$.
\end{proof}

% In the setting of topological spaces, the role of double mapping cylinders is to provide an explicit model for the homotopy pushout of $f$ and $g$.
% We now move towards developing 

% We move towards developing the ``homotopical'' properties of double mapping cylinders.
When $g = \id[G]$, the map $\ell_0(f, \id[G])$ admits a retraction
\[ \sigma(f) \from \Cyl{m}(f, \id[G]) \to H \]
which is constructed as the map out of the pushout
\[ \begin{tikzcd}
    G \ar[r, "i_0"] \ar[d, "f"'] \ar[rd, phantom, "\potick" very near end] & G \gtimes I_m \ar[d] \ar[r, "\pi_X"] & G \ar[d, "f"] \\
    H \ar[r] & {\Cyl{m}(f, \id[G])} \ar[r, dotted, "{\sigma(f)}"] & H
\end{tikzcd} \]
Similarly, if $f = \id[G]$ then $r_0(\id[G], g)$ admits a retraction.
We overload notation and write $\sigma(g) \from \Cyl{m}(\id[G], g) \to Z$ for this morphism as well.
\begin{proposition} \label{cyl-inclusion-is-htpy-equiv}
    The maps $\sigma(f)$ and $\sigma(g)$ are homotopy inverses of $\ell_0(f, \id[G])$ and $r_0(\id[G], g)$, respectively.
\end{proposition}
\begin{proof}
    We prove the statement for $\sigma(f)$; the argument for $\sigma(g)$ is analogous.
    By construction, $\sigma(f) \circ \ell_0(f, \id[G]) = \id[Y]$.
    It remains only to define the homotopy $\ell_0(f, \id[G]) \circ \sigma(f) \sim \id[{\Cyl{m}(f, \id[G])}]$.
    
    The functor $\uvar \gtimes I_m \from \Graph \to \Graph$ preserves pushouts, hence the square
    \[ \begin{tikzcd}[column sep = 5.5em, row sep = 4em]
        G \gtimes I_m \ar[d, "f \gtimes I_m"'] \ar[r, "i_0 \gtimes I_m", hook] \ar[rd, phantom, "\potick" very near end] & (G \gtimes I_m) \gtimes I_m \ar[d] \\
        H \gtimes I_m \ar[r, hook, "{\ell_0(f, \id[G]) \gtimes I_m}"] & \Cyl{m}(f, \id[G]) \gtimes I_m
    \end{tikzcd} \]
    is a pushout.
    Define a map $\alpha \from G \gtimes I_m \gtimes I_m \to G \gtimes I_m$ by
    \[ \alpha(v, t, t') := (v, \min(t, t')). \]
    Observe that the left and top squares of the cube
    \[ \begin{tikzcd}
        G \gtimes I_m \ar[rr] \ar[dd] \ar[rd, "\pi_G"] &[+1ex] {} & (G \gtimes I_m) \gtimes I_m \ar[dd] \ar[rd, "\alpha"] &[-3ex] {} \\
        {} & X \ar[rr, crossing over] & {} & G \gtimes I_m \ar[dd] \\
        H \gtimes I_m \ar[rr] \ar[rd, "\pi_H"] & {} & \Cyl{m}(f, \id[G]) \gtimes I_m \ar[rd, dotted, "\beta"] & {} \\
        {} & Y \ar[rr] \ar[from=uu, crossing over] & {} & \Cyl{m}(f, \id[G]) 
    \end{tikzcd} \]
    commute, from which we define the map $\beta \from \Cyl{m}(f) \gtimes I_m \to \Cyl{m}(f)$ as the induced map between pushouts.
    An explicit formula for $\beta$ is given by
    \begin{align*}
        \beta \big( (v_G, t), t' \big) &:= \big( v_G, \min(t, t') \big) \\
        \beta(v_H, t') &:= v_H.
    \end{align*}
    If $t' = m$ then $\beta(\uvar, m) = \id[\Cyl{m}(f)]$.
    If $t' = 0$ then $\beta(\uvar, 0) = \ell_0(f, \id[G]) \circ \sigma(f)$.
    Thus, $\beta$ is the desired homotopy.
\end{proof}
% The colimit in \cref{def:double-mapping-cylinder} can be computed via iterated pushouts, yielding the following result.
% \begin{proposition}
%     Each of the squares
%     \[ \begin{tikzcd}
%         X & 
%     \end{tikzcd} \]
%     is a pushout.
%     Moreover, the diagram obtained by pasting these squares

%     commutes.
% \end{proposition}

% We highlight the isomorphism $\Cyl{m}(f, g) \xrightarrow{\cong} \Cyl{m}(g, f)$ defined by the expressions:
% \[ \begin{array}{c c c}
%     (x, t) \mapsto (x, n-t) & y \mapsto y & z \mapsto z.
% \end{array} \]

More generally, given a parameter $k \in \{ 0, \dots, m \}$, we define maps 
\[ \ell_k(f, g) \from \Cyl{k}(f, \id[G]) \to \Cyl{m}(f, g) \quad \text{and} \quad r_k(f, g) \from \Cyl{k}(\id[G], g) \to \Cyl{m}(f, g) \]
as the maps out of the pushout depicted in:
\[ \begin{tikzcd}
        G \ar[r, "i_{0}"] \ar[d, "f"'] \ar[rd, phantom, "\potick" very near end] &[-1ex] G \gtimes I_{k} \ar[r, "{(x, t) \mapsto (x, t)}"] \ar[d] &[+1.8ex] G \gtimes I_m \ar[d] \\
        H \ar[r] & {\Cyl{k}(f, \id[G])} \ar[r, dotted, "{\ell_k(f, g)}"] & {\Cyl{m}(f, g)}
\end{tikzcd} \qquad \begin{tikzcd}
        G \ar[r, "i_{k}"] \ar[d, "g"'] \ar[rd, phantom, "\potick" very near end] &[-1ex] G \gtimes I_{k} \ar[r, "{(x, t) \mapsto (x, t+m-k)}"] \ar[d] &[+5.5ex] G \gtimes I_m \ar[d] \\
        K \ar[r] & {\Cyl{k}(\id[G], g)} \ar[r, dotted, "{r_k(f, g)}"] & {\Cyl{m}(f, g)}
\end{tikzcd} \]
% Explicit formulas for $\ell_k(f, g)$ and $r_k(f, g)$ are given by
% \[ \begin{array}{c c}
%     \big( \ell_k(f, g) \big)(x, t) = (x, t) & \big( r_k(f, g) \big)(x, t) = (x, t+n-k) \\
%     \big( \ell_k(f, g) \big)(y) = y & \big( r_k(f, g) \big)(z) = z.
% \end{array} \]
When $k = 0$, these definitions coincide with the previous definitions of $\ell_0(f, g)$ and $r_0(f, g)$.
% By construction, $$
% When $f = \id[X]$ and $k = 0$, the map $r_0(\id[f], g)$ recovers the inclusion $i_g \from Y \ito \Cyl{m}(g)$.

From the explicit description of the double mapping cylinder, we deduce the following results.
\begin{proposition} \label{induced-subgraph-double-mapping-cylinder}
    Let $f \from G \to H$ and $g \from G \to K$ be graph maps.
    For $m \geq 2$ and $k \neq m$, both maps
    \begin{align*} 
        \ell_{k}(f, g) &\from \Cyl{k}(f, \id[G]) \to \Cyl{m}(f, g) \\
        r_{k}(f, g) &\from \Cyl{k}(\id[G], g) \to \Cyl{m}(f, g)
    \end{align*} 
    are induced subgraph inclusion. \qed
\end{proposition}
\begin{proposition} \label{induced-subgraph-double-mapping-cylinder-mono}
    Let $f \from G \to H$ and $g \from G \to K$ be graph maps.
    \begin{enumerate}
        \item If $g$ is a monomorphism then the map
        \[ \ell_m(f, g) \from \Cyl{m}(f, \id[G]) \to \Cyl{m}(f, g) \]
        is a monomorphism.
        Moreover, if $g$ is an induced subgraph inclusion then $\ell_m(f, g)$ is an induced subgraph inclusion.
        \item If $f$ is a monomorphism then the map
        \[ r_m(f, g) \from \Cyl{m}(\id[G], g) \to \Cyl{m}(f, g) \]
        is a monomorphism.
        Moreover, if $f$ an induced subgraph inclusion then $r_m(f, g)$ is an induced subgraph inclusion. \qed
    \end{enumerate}
\end{proposition}

% For $n \geq 2$ and $k \in \{ 1, \dots, n-1 \}$, the map $\ell_k(f, g)$ admits a retraction.

% we define a map $\sigma_k(f) \from \Cyl{m}(f, \id[G]) \to \Cyl{n-1}(f, \id[X])$ as the map out of the pushout:
% \[ \begin{tikzcd}
    
% \end{tikzcd} \]

For small values of $k$, the subgraph inclusions $\ell_k(f, g)$ and $r_k(f, g)$ have disjoint images inside $\Cyl{m}(f, g)$.
When $k \geq \lceil m / 2 \rceil$, the two subgraphs intersect, and their intersection is isomorphic to a cylinder on $G$.
More precisely, 
\begin{proposition}
    Let $f \from G \to H$ and $g \from G \to K$ be graph maps.
    \begin{enumerate}
        \item Let $p \geq 0$ be non-negative and $q_1, q_2 \geq 1$ be strictly positive.
        Then, the square
        \[ \begin{tikzcd}[column sep = 4em, row sep = 3em]
            G \gtimes I_{p} \ar[r, hook, "{r_{p}(f, \id)}"] \ar[d, hook, "{\ell_{p}(\id, g)}"'] \ar[rd, phantom, "\pbtick" very near start, "\potick" very near end] & \Cyl{p + q_1}(f, \id[G]) \ar[d, "{\ell_{p+q_1}(f, g)}", hook] \\
            \Cyl{p + q_2}(\id[G], g) \ar[r, "{r_{p+q_2}(f, g)}", hook] & \Cyl{p + q_1 + q_2}(f, g)
        \end{tikzcd} \]
        is a pullback. 
        \item If $f$ is injective then part (1) is additionally true when $q_1 = 0$.
        \item If $g$ is injective then part (1) is additionally true when $q_2 = 0$. \qed
    \end{enumerate}
\end{proposition}
% More precisely, the square
% \[ \begin{tikzcd}[column sep = 5.5em, row sep = 3.5em]
%     G \gtimes I_{n - 2(n-k)} \ar[r, hook, "{r_{n-2(n-k)}(f, \id)}"] \ar[d, hook, "{\ell_{n-2(n-k)}(\id, g)}"'] \ar[rd, phantom, "\pbtick" very near start] & \Cyl{k}(f, \id[G]) \ar[d, "{\ell_{k}(f, g)}"] \\
%     \Cyl{k}(\id[G], g) \ar[r, "{r_{k}(f, g)}"] & \Cyl{m}(f, g)
% \end{tikzcd} \]
% is both a pushout and a pullback for any $k \in \{ \lceil n / 2 \rceil, \dots, n-1 \}$.
% We introduce the change of variables
% \[ p := n - 2(n-k) \quad \text{and} \quad q := n-k. \]
% Now, for any $p \geq 0$ and $q \geq 1$, the square
% \[ \begin{tikzcd}[column sep = 4em, row sep = 3em]
%     G \gtimes I_{p} \ar[r, hook, "{r_{p}(f, \id)}"] \ar[d, hook, "{\ell_{p}(\id, g)}"'] \ar[rd, phantom, "\pbtick" very near start] & \Cyl{p + q}(f, \id[X]) \ar[d, "{\ell_{p+q}(f, g)}", hook] \\
%     \Cyl{p + q}(\id[X], g) \ar[r, "{r_{p+q}(f, g)}", hook] & \Cyl{p + 2q}(f, g)
% \end{tikzcd} \]
% is both a pushout and a pullback.
% (In fact, this square is also a pushout, though we do not make use of this.)
% We introduce a change of variables $j := n - k$ to rewrite this square as:
% \[ \begin{tikzcd}[column sep = 5.5em, row sep = 3.5em]
%     G \gtimes I_{} \ar[r, hook, "{r_{n-2(n-k)}(f, \id)}"] \ar[d, hook, "{\ell_{n-2(n-k)}(\id, g)}"'] \ar[rd, phantom, "\pbtick" very near start] & \Cyl{k}(f) \ar[d, "{\ell_{k}(f, g)}"] \\
%     \Cyl{k}(g) \ar[r, "{r_{k}(f, g)}"] & \Cyl{m}(f, g)
% \end{tikzcd} \]

% Using what we know about $n$-skeletal pushouts of graphs, we may now 
\begin{lemma} \label{double-mapping-cylinder-is-n-skeletal-po}
    Let $f \from G \to H$ and $g \from G \to K$ be graph maps.
    \begin{enumerate}
        \item For non-negative $p \geq 0$ and strictly positive $q_1, q_2 \geq 1$, the square
        \[ \begin{tikzcd}[column sep = 4em, row sep = 3em]
            G \gtimes I_{p} \ar[r, hook, "{r_{p}(f, \id)}"] \ar[d, hook, "{\ell_{p}(\id, g)}"'] \ar[rd, phantom, "\pbtick" very near start] & \Cyl{p+q_1}(f, \id[G]) \ar[d, "{\ell_{p+q_1}(f, g)}", hook] \\
            \Cyl{p+q_2}(\id[G], g) \ar[r, "{r_{p+q_2}(f, g)}", hook] & \Cyl{p + q_1 + q_2}(f, g)
        \end{tikzcd} \]
        is a $(p+1)$-skeletal pushout.
        \item If $f$ is injective then part (1) is additionally true when $q_1 = 0$.
        \item If $g$ is injective then part (1) is additionally true when $q_2 = 0$.
    \end{enumerate}
\end{lemma}
\begin{proof}
    Part (1) follows from \cref{n-skeletal-po-condition-general}, provided that we verify condition (3).
    
    We first observe that, for any $m \geq 1$, if $[v, t]$ and $[v', t']$ are vertices in the image of the structure morphism $G \gtimes I_m \to \Cyl{m}(f, g)$ then the distance from $[v, t]$ to $[v', t']$ in $\Cyl{m}(f, g)$ satisfies
    \[ \operatorname{dist} \big( [v, t], [v', t'] \big) \geq |t - t'|. \]
    % is bounded above by the absolute value $|t - t'|$ of the distance between $t$ and $t'$
    This follows from the description of edges in the box product $G \gtimes I_m$.
    We deduce that if $v \in \Cyl{p+q_1+q_2}(f, g)$ is a vertex which is not in the image of $\ell_{p+q_1}(f, g)$ and $v' \in \Cyl{p+q_1+q_2}(f, g)$ is a vertex which is not in the image of $r_{p+q_2}(f, g)$ then the distance from $v$ to $v'$ is at least $p+2$.
    This now implies condition (3) (c.f.\ \cref{n-skeletal-po-condition-subgraphs}).

    Parts (2) is similar; suppose $f$ is a monomorphism and $q_1 = 0$.
    If $v' \in \Cyl{p+q_2}(f, g)$ is not in the image of $r_{p+q_2}(f, g)$ then it must be of the form $\ell_0(f, g)(v_H)$ for some vertex $v_H \in H$ in the complement of the image of $f$ (here we are using that $f$ is a monomorphism).
    In this case, the distance from $v'$ to a vertex of the form $[v, p+1]$ satisfies
    \begin{align*} 
        \operatorname{dist} \big( v', [v, p + 1] \big) &\geq \operatorname{dist} \big( v_H, f(v) \big) + p + 1 \\
        &\geq 1 + p + 1 \\
        &= p + 2. 
    \end{align*}
    As before, this implies condition (3).

    The proof of part (3) is analogous to the proof of part (2). 
\end{proof}

\begin{corollary} \label{double-mapping-cyl-inclusions-hopo}
    Let $f \from G \to H$ and $g \from G \to K$ be graph maps.
    \begin{enumerate}
        \item Given non-negative $p \geq 0$, strictly positive $q_1, q_2 \geq 1$, and $m \in \NN \cup \{ \infty \}$, the square
        \[ \begin{tikzcd}
            \gnerve[m] \big( G \gtimes I_p \big) \ar[r] \ar[d] & \gnerve[m] \big( \! \Cyl{p+q_1}(f, \id[G]) \big) \ar[d] \\
            \gnerve[m] \big( \! \Cyl{p+q_2}(\id[G], g) \big) \ar[r] & \gnerve[m] \big( \! \Cyl{p+q_1 + q_2}(f, g) \big)
        \end{tikzcd} \]
        is a homotopy pushout of $p$-types.
        \item If $f$ is injective then part (1) is additionally true when $q_1 = 0$.
        \item If $g$ is injective then part (1) is additionally true when $q_2 = 0$.
    \end{enumerate}
\end{corollary}
\begin{proof}
    Combine \cref{double-mapping-cylinder-is-n-skeletal-po} with \cref{n-skeletal-po-is-hopo}.
\end{proof}
% \begin{theorem}
%     Let $f \from G \to H$ and $g \from G \to K$ be graph maps.
% \end{theorem}
% \begin{definition}
%     Given a graph map $f \from X \to Y$, the \emph{mapping cylinder factorization} of $f$ of length $n$ is the pair of maps appearing on the left and right of the commutative triangle:
%     \[ \begin{tikzcd}
%         {} & \Cyl{m}(f) \ar[dr, "r_f"] & {} \\
%         X \ar[rr, "f"] \ar[ur, "{r_n(\id, f)}"] & {} & Y
%     \end{tikzcd} \]
% \end{definition}
\begin{theorem} \label{cyl-factorization-hopo}
    Let $f \from G \to H$ and $g \from G \to K$ be graph maps.
    \begin{enumerate}
        \item Given non-negative $p \geq 0$, an integer $q \geq 2$, and $m \in \NN \cup \{ \infty \}$, the pushout
        \[ \begin{tikzcd}
            G \ar[d, "{\ell_{0}(\id, g)}"'] \ar[r, "f"] \ar[rd, phantom, "\potick" very near end] & H \ar[d, "{\ell_0(f, g)}"] \\
            \Cyl{p+q}({\id[G], g}) \ar[r] & \Cyl{p+q}(f, g)
        \end{tikzcd} \]
        of $f$ along $\ell_{0}(\id, g)$ becomes a homotopy pushout of $p$-types
        \[ \begin{tikzcd}
            \gnerve[m]{G} \ar[d, "{\gnerve[m]{(\ell_{0}(\id, g))}}"'] \ar[r, "{\gnerve[m]{(f)}}"] \ar[rd, phantom, "\potick" very near end] & \gnerve[m]{H} \ar[d, "{\gnerve[m]{(\ell_0(f, g))}}"] \\
            \gnerve[m]{\big(\! \Cyl{p+q}({\id[G], g}) \big)} \ar[r] & \gnerve[m]{\big(\! \Cyl{p+q}(f, g) \big)}
        \end{tikzcd} \] 
        after applying $\gnerve[m] \from \Graph \to \cSet$.
        \item If one of $f$ or $g$ is injective then part (1) is additionally true when $q = 1$.
        \item If both $f$ and $g$ are injective then part (1) is additionally true when $q = 0$.
    \end{enumerate}
\end{theorem}
\begin{proof}
    We first prove part (1).
    To this end, the first pushout square can be factored into two pushout squares:
    \[ \begin{tikzcd}
        G \ar[r, "f"] \ar[d, "{i_0}"', "\sim"] \ar[rd, phantom, "\potick" very near end] & H \ar[d, "{\ell_{0}(f, \id)}", "\sim"'] \\
        G \gtimes I_{p+q-1} \ar[r] \ar[d, "{\ell_{p+q-1}(\id, g)}"'] \ar[rd, phantom, "\potick" very near end] & \Cyl{p+q-1}(f, \id[G]) \ar[d, "{\ell_{p+q-1}(f, g)}"] \\
        \Cyl{p+q}(\id[G], g) \ar[r] & \Cyl{p+q}(f, g)
    \end{tikzcd} \]
    The top left and top right morphisms are homotopy equivalences by \cref{cyl-inclusion-is-htpy-equiv}.
    Thus, the composite square becomes a homotopy pushout if and only if the bottom square becomes a homotopy pushout.

    We paste an additional square to the left of the bottom square, yielding the diagram:
    \[ \begin{tikzcd}[column sep = 4.5em, row sep = 3.5em]
        {} & {} & {} \\[-2em]
        G \gtimes I_{p} \ar[r, "{r_{p}(\id[G], \id[G])}"] \ar[d, "{\ell_{p}(\id[G], g)}"'] & G \gtimes I_{p+q-1} \ar[r] \ar[d, "{\ell_{p+q-1}(\id[G], g)}" description] \ar[rd, phantom, "\potick" very near end] \ar[u, draw=none, "{r_{p}(f, \id[G])}"{name=Top, description, very near end, yshift=1ex}] & \Cyl{p+q-1}(f, \id[G]) \ar[d, "{\ell_{p+q-1}(f, g)}"] \\
        \Cyl{p+1}(\id[G], g) \ar[r, "{r_{p+1}(\id[G], g)}"] & \Cyl{p+q}(\id[G], g) \ar[r] \ar[d, draw=none, "{r_{p+q}(f, g)}"{name=Bott, description, very near end, yshift=-1ex}] & \Cyl{p+q}(f, g) \\[-2em]
        {} & {} & {}
        \arrow[from=2-1, to=2-3, rounded corners, 
			to path={ 
				|- (Top.south)
				-| (\tikztotarget)
			}
	    ]
        \arrow[from=3-1, to=3-3, rounded corners, 
			to path={ 
				|- (Bott.north)
				-| (\tikztotarget)
			}
	    ]
        % \arrow[from=2-1, to=2-3, bend left=15]
        % \arrow[from=3-1, to=3-3, bend right=15]
    \end{tikzcd} \]
    Both the left and composite squares become homotopy pushouts of $p$-types by \cref{double-mapping-cyl-inclusions-hopo}.
    Therefore, the right square becomes a homotopy pushout of $p$-types.

    Concerning part (2), if $f$ is injective then we repeat the proof above verbatim (with $q = 1$).
    If $g$ is injective then we repeat the proof above, substituting $\Cyl{p+1}(\id[G], g)$ with $\Cyl{p}(\id[G], g)$.
    For part (3), we additionally substitute all instances of $p+q-1$ with $p$.
\end{proof}

\subsection{Comparison with the cubical double-mapping cylinder}

We now compare the construction of the double mapping cylinder in $\Graph$ with the analogous construction in $\cSet$.

% \newcommand{\Sp}[1]{\mathrm{Sp}^n}
% To start, let $\Sp{n}$ denote the length-$n$ ``spine'' in cubical sets, i.e.\ the cubical set
% \[ \Sp{n} :=   \]
Given cubical maps $f \from X \to Y$ and $g \from X \to Z$ in $\cSet$, we write $\Cyl{}^{\boxcat}(f, g)$ for the double mapping cylinder in cubical sets, i.e.\ the colimit
\[ \Cyl{}^{\boxcat}(f, g) := \colim \left( \begin{tikzcd}[row sep = 2.2em, column sep = 3em]
    {} & X \ar[dl, "f"'] \ar[rd, "{\id \gprod \face{}{1, 0}}" description] & {} & X \ar[dl, "{\id \gprod \face{}{1,1}}" description] \ar[dr, "g"] & {} \\
    Y & {} & X \gprod \cube{1} & {} & Z
\end{tikzcd} \right) \]
We write $\ell^{\boxcat}(f, g)$ and $r^{\boxcat}(f, g)$ for the structure maps $Y \to \Cyl{}^{\boxcat}(f, g)$ and $Z \to \Cyl{}^{\boxcat}(f, g)$, respectively.
Just as in the case of double-mapping cylinders of graphs, if $g = \id[X]$ then $\ell^{\boxcat}(f, \id[X])$ admits a retraction $\sigma(f)$ which is also a homotopy inverse, and if $f = \id[X]$ then $r^{\boxcat}(\id[X], g)$ admits a retraction $\sigma(g)$ which is also a homotopy inverse.

We immediately observe the following:
\begin{proposition} \label{cyl-reali-iso}
    Let $\Lambda$ denote the category with three objects $0, 1, 2$ and morphisms $0 \to 1$ and $0 \to 2$.
    For $m \in \NN$, we have functors:
    \[ \Cyl{m} \from \Fun(\Lambda, \Graph) \to \Graph \qquad \Cyl{}^{\boxcat} \from \Fun(\Lambda, \cSet) \to \cSet \]
    defined by taking the double-mapping cylinder in $\Graph$ and in $\cSet$, respectively.
    Moreover, the square
    \[ \begin{tikzcd}
        \Fun(\Lambda, \cSet) \ar[r, "{\Cyl{}^{\boxcat}}"] \ar[d, "{\reali[m]{\uvar}}"'] & \cSet \ar[d, "{\reali[m]{\uvar}}"] \\
        \Fun(\Lambda, \Graph) \ar[r, "{\Cyl{m}}"] & \Graph
    \end{tikzcd} \]
    commutes up to natural isomorphism. \qed
\end{proposition}
\begin{remark}
    There is an analogue of \cref{cyl-reali-iso} which relates the geometric realization functor $\reali{\uvar} \from \cSet \to \Top$ and the double-mapping cylinder in topological spaces. 
\end{remark}
We use the isomorphism in \cref{cyl-reali-iso} to construct our first comparison map between the double-mapping cylinders in $\cSet$ and $\Graph$: it is given by the unit map $\eta_{\Cyl{}^{\boxcat}(f, g)}$, which we view as a morphism
\[ \eta \from \Cyl{}^{\boxcat}(f, g) \to \gnerve[m] \big( \Cyl{m}(\reali[m]{f}, \reali[m]{g}) \big) \]
which is moreover natural in $(f, g) \in \Fun(\Lambda, \cSet)$.

% The key property of the comparison map is that it is an $n$-equivalence, provided the realization length $m$ is sufficiently large with respect to $n$.
% When $f$ and $g$ are monomorphisms, this comparison map is an $n$-equivalence provided the realization length $m$ is sufficiently large with respect to $n$.
% Under suitable conditions, the comparison map is an $n$-equivalence.
We investigate what conditions guarantee the comparison map is an $n$-equivalence.
\begin{lemma} \label{cyl-cset-graph-equiv-mono}
    Let $f \from X \to Y$ and $g \from X \to Z$ be cubical monomorphisms.
    Given $n \geq 0$ and $m \geq n$, suppose the unit maps
    \[ \eta_X \from X \to \gnerve[m]{\reali[m]{X}} \quad \eta_Y \from Y \to \gnerve[m]{\reali[m]{Y}} \quad \eta_Z \from Z \to \gnerve[m]{\reali[m]{Z}} \] 
    are $n$-equivalences.
    Then, the comparison map
    \[ \eta \from \Cyl{}^{\boxcat}(f, g) \to \gnerve[m] \big( \Cyl{m}(\reali[m]{f}, \reali[m]{g}) \big) \]
    is an $n$-equivalence.
\end{lemma}
\begin{proof}
    The cubical double-mapping cylinder can be written as a pushout
    \[ \begin{tikzcd}
        X \gprod \cube{1} \ar[r] \ar[d] \ar[rd, phantom, "\potick" very near end] & \Cyl{}^{\boxcat}(f, \id[X]) \ar[d] \\
        \Cyl{}^{\boxcat}(\id[X], g) \ar[r] & \Cyl{}^{\boxcat}(f, g)
    \end{tikzcd} \]
    in $\cSet$.
    Applying the functor $\reali[m]{\uvar} \from \cSet \to \Graph$ to this square yields the diagram
    \[ \begin{tikzcd}
        \reali[m]{X} \gtimes I_m \ar[r, "{r_m(f, \id)}"] \ar[d, "{\ell_m(\id, g)}"'] & \Cyl{m}(\reali[m]{f}, \id) \ar[d, "{\ell_m(f, g)}"] \\
        \Cyl{m}(\id, \reali[m]{g}) \ar[r, "{r_m(f, g)}"] & \Cyl{m}(\reali[m]{f}, \reali[m]{g})
    \end{tikzcd} \]
    in $\Graph$.
    This square is a homotopy pushout of $n$-types by \cref{double-mapping-cyl-inclusions-hopo}.

    The comparison map $\eta$ induces a commutative cube
    \[ \begin{tikzcd}
        X \gprod \cube{1} \ar[rr] \ar[dd] \ar[rd, "\eta"] & {} & \Cyl{}^{\boxcat}(f, \id) \ar[dd] \ar[rd, "\eta"] & {} \\
        {} & \gnerve[m] \big( \reali[m]{X} \gtimes I_m \big) \ar[rr, crossing over] & {} & \gnerve[m] \big( \Cyl{m}(\reali[m]{f}, \id) \big) \ar[dd] \\
        \Cyl{}^{\boxcat}(\id, g) \ar[rr] \ar[rd, "\eta"] & {} & \Cyl{}^{\boxcat}(f, g) \ar[rd, "\eta"] & {} \\
        {} & \gnerve[m] \big( \Cyl{m}(\id, \reali[m]{g}) \big) \ar[rr] \ar[from=uu, crossing over] & {} & \gnerve[m] \big( \Cyl{m}(\reali[m]{f}, \reali[m]{g}) \big)
    \end{tikzcd} \]
    The front and back faces of this cube are homotopy pushouts of $n$-types: for the front square, this follows from \cref{double-mapping-cyl-inclusions-hopo}; for the back square, we observe this is a pushout of monomorphisms, hence a homotopy pushout in the Cisinski model structure for $n$-types.

    Three of the diagonal maps arise as the top arrow in the commutative squares:
    \[ \begin{tikzcd}[cramped, column sep = 1.8em]
        X \gprod \cube{1} \ar[r, "\eta"] \ar[d, "\pi_X"', "\sim"] & \gnerve[m]{\big( \reali[m]{X} \gtimes I_m \big)} \ar[d, "{\gnerve[m](\pi_{\reali[m]{X}})}", "\sim"'] \\
        X \ar[r, "\eta", "\sim"'] & \gnerve[m] \reali[m]{X}
    \end{tikzcd} \quad \begin{tikzcd}[cramped,  column sep = 1.5em]
        \Cyl{}^{\boxcat}(f, \id) \ar[r, "\eta"] \ar[d, "{\sigma(f)}"'] & \gnerve[m]{\big( \Cyl{m}(\reali[m]{f}, \id) \big)} \ar[d, "{\gnerve[m] (\sigma(\reali[m]{f}))}"] \\
        Y \ar[r, "\eta", "\sim"'] & \gnerve[m] \reali[m]{Y}
    \end{tikzcd} \quad \begin{tikzcd}[cramped,  column sep = 1.5em]
        \Cyl{}^{\boxcat}(\id, g) \ar[r, "\eta"] \ar[d, "{\sigma(g)}"', "\sim"] & \gnerve[m]{\big( \Cyl{m}(\id, \reali[m]{g}) \big)} \ar[d, "{\gnerve[m] (\sigma(\reali[m]{g}))}", "\sim"'] \\
        Z \ar[r, "\eta", "\sim"'] & \gnerve[m] \reali[m]{Z}
    \end{tikzcd} \]
    The vertical maps are homotopy equivalences; the bottom maps are $n$-equivalences by assumption, hence the comparison maps are $n$-equivalences.
\end{proof}

The proof of \cref{cyl-cset-graph-equiv-mono} crucially relies on the assumption that $f$ and $g$ are monomorphisms, the argument we provide does not work in the general case.
That said, we can construct a zig-zag of weak equivalences in the general case.

Observe that if $\alpha \from I_m \to I_{m'}$ is a graph map which is endpoint-preserving (i.e.\ $\alpha(0) = 0$ and $\alpha(m) = m'$) then there is an induced graph map $\alpha_* \from \Cyl{m}(f, g) \to \Cyl{m'}(f, g)$ for any $f$ and $g$.
This map arises as the map between colimits induced from the morphism of diagrams:
\[ \begin{tikzcd}
    H \ar[d, equal] & G \ar[l, "f"'] \ar[r, "i_0"] \ar[d, equal] & G \gtimes I_m \ar[d, "\id \gtimes \alpha"] & G \ar[l, "i_{m}"'] \ar[d, equal] \ar[r, "g"] & K \ar[d, equal, ""{name=A}] & \Cyl{m}(f, g) \ar[d, "\alpha_*", ""'{name=B}] \\
    H & G \ar[l, "f"'] \ar[r, "i_0"] & G \gtimes I_{m'} & G \ar[l, "{i_{m'}}"'] \ar[r, "g"] & K & \Cyl{m'}(f, g)
    \ar[from=A, to=B, phantom, "\leadsto"]
\end{tikzcd} \]
The map $\alpha_*$ provides a comparison map for double-mappig cylinders of different lengths.

Fix a non-negative integer $p \in \NN$.
Given two pairs of non-negative integers $q_1, q_1' \geq 0$ and $q_2, q_2' \geq 0$, let $\Pi$ denote the endpoint-preserving map 
\[ \begin{array}{r@{\ }l }
    \Pi & \from I_{p+q_1+q_1'+q_2+q_2'} \to I_{p+q_1+q_2} \\
    {} & {} \\
    \Pi(t) &:= \begin{cases}
        0 & \text{if } t \leq q_1' \\
        t - q_1' & \text{if } q_1' \leq t \leq p + q_1 + q_1' + q_2 \\
        p + q_1 + q_2 & \text{if } t \geq p + q_1 + q_1' + q_2
    \end{cases}
\end{array} \]
That is, $\Pi$ contracts the first $q_1'$ vertices and the last $q_2'$ vertices.
This map induces a map
\[ \Pi_* \from \Cyl{p+q_1+q_1'+q_2+q_2'}(f, g) \to \Cyl{p+q_1+q_2}(f, g) \]
which is natural in $(f, g) \in \Fun(\Lambda, \Graph)$.
\begin{proposition} \label{pi-map-n-equiv}
    Let $f \from G \to H$ and $g \from G \to K$ be graph maps.
    For non-negative $p, q_1', q_2' \geq 0$ and strictly positive $q_1, q_2 \geq 1$,
    \begin{enumerate}
        \item The map $\Pi_*$ is a $p$-equivalence.
        \item If $f$ is injective then part (1) is additionally true when $q_1 = 0$.
        \item If $g$ is injective then part (1) is additionally true when $q_2 = 0$.
    \end{enumerate}
\end{proposition}
\begin{proof}
    The maps $\Pi_*$ for various values of $q_1, q_1', q_2, q_2'$ induce a commutative cube
    \[ \begin{tikzcd}
        G \gtimes I_{p} \ar[rr] \ar[dd] \ar[rd, equal] & {} & \Cyl{p+q_1+q_1'}(f, \id[G]) \ar[dd] \ar[rd, "\Pi_*"] & {} \\
        {} & G \gtimes I_{p} \ar[rr, crossing over] & {} & \Cyl{p+q_1}(f, \id[G]) \ar[dd] \\
        \Cyl{p+q_2+q_2'}(\id[G], g) \ar[rr] \ar[rd, "\Pi_*"] & {} & \Cyl{p+q_1+q_1'+q_2+q_2'}(f, g) \ar[rd, "\Pi_*"] & {} \\
        {} & \Cyl{p+q_2}(\id[G], g) \ar[rr] \ar[from=uu, crossing over] & {} & \Cyl{p+q_1+q_2}(f, g)
    \end{tikzcd} \]
    More precisely, the map between the top right objects arises from setting $q_2 = q_2' = 0$; the map between the bottom left objects arises from setting $q_1 = q_1' = 0$.
    By \cref{double-mapping-cyl-inclusions-hopo}, both the front and back squares become homotopy pushouts of $p$-types after applying $\gnerve[m]$ for any $m \in \NN \cup \{ \infty \}$.
    As the functor $\gnerve[m]$ reflects $p$-equivalences, it suffices to show that each of the maps between the top left, top right, and bottom left maps are $p$-equivalences.

    The map between the top left objects is an isomorphism, hence a $p$-equivalence.
    Both the maps between the top right and bottom right objects arise as the top arrow in the two commutative triangles:
    \[ \begin{tikzcd}
        \Cyl{p+q_1+q_1'}(f, \id[G]) \ar[rr, "\Pi_*"] \ar[rd, "\sigma(f)"', "\sim"] & {} & \Cyl{p+q_1}(f, \id[G]) \ar[ld, "\sigma(f)", "\sim"'] \\
        {} & H & {}
    \end{tikzcd} \quad \begin{tikzcd}
        \Cyl{p+q_2+q_2'}(\id[G], g) \ar[rr, "\Pi_*"] \ar[rd, "\sigma(g)"', "\sim"] & {} & \Cyl{p+q_1}(\id[G], g) \ar[ld, "\sigma(g)", "\sim"'] \\
        {} & K & {}
    \end{tikzcd} \] 
    The vertical maps are homotopy equivalences by \cref{cyl-inclusion-is-htpy-equiv}, thus the top arrows are weak equivalences.
\end{proof}
We use the map $\Pi$ to construct the zig-zag comparing the double-mapping cylinders of a pair $(f, g)$ and its realization $(\reali[m]{f}, \reali[m]{g})$.
\begin{construction} \label{construction:zig-zag}
    Let $f \from X \to Y$ and $g \from X \to Z$ be cubical maps.
    Then, for any $m \in \NN \cup \{ \infty \}$, the cubical double-mapping cylinder $\Cyl{}^{\boxcat}(f, g)$ is connected to the $m$-nerve of the graph double mapping cylinder of length $m+2$ via the zig-zag:
    \[ \begin{tikzcd}
        \Cyl{}^{\boxcat}(f, g) &[+5em] \ar[l, "{\Cyl{}^{\boxcat}(\id[X], \sigma(f), \sigma(g))}"'] \Cyl{}^{\boxcat} \big( r^{\boxcat}(f, \id), \ell^{\boxcat}(\id, g) \big) \ar[r, "\eta"] &[-0.6em] \gnerve[m]{ \Big( \Cyl{m} \big( r_{0}(\reali[m]{f}, \id ), \ell_0(\id , \reali[m]{g}) \big) \Big)} \ar[ld, "\cong" description] \\
        {} & \gnerve[m] \big( \Cyl{3m}(\reali[m]{f}, \reali[m]{g}) \big) \ar[r, "{\gnerve[m]{(\Pi_*)}}"'] & \gnerve[m]{ \big( \Cyl{m+2}(\reali[m]{f}, \reali[m]{g}) \big)}
    \end{tikzcd} \]
    Defining each map individually:
    \begin{enumerate}
        \item the first map is functoriality of $\Cyl{}^{\boxcat}$ applied to the morphism of diagrams:
        \[ \begin{tikzcd}[column sep = 4em, row sep = 3em]
            \Cyl{}^{\boxcat}(f, \id[X]) \ar[d, "{\sigma(f)}"'] & X \ar[l, "{\ell^{\boxcat}(f, \id[X])}"'] \ar[r, "{r^{\boxcat}(\id[X], g)}"] \ar[d, equal] & \Cyl{}^{\boxcat}(\id[X], g) \ar[d, "{\sigma(g)}"] \\
            Y & X \ar[l, "f"'] \ar[r, "g"] & Z
        \end{tikzcd} \]
        in $\cSet$.
        \item The second map is the comparison map $\eta$; note that the $m$-realizations of the maps 
        \[ r^{\boxcat}(f, \id) \from X \to \Cyl{}^{\boxcat}(f, \id) \qquad \text{and} \qquad \ell^{\boxcat}(\id, g) \from X \to \Cyl{}^{\boxcat}(\id, g) \] 
        have been re-written as the graph maps 
        \[ r_0(\reali[m]{f}, \id)\from \reali[m]{X} \to \Cyl{m}(\reali[m]{f}, \id) \qquad \text{and} \qquad \ell_0(\id, \reali[m]{g}) \from \reali[m]{X} \to \Cyl{m}(\id, \reali[m]{g}), \] 
        respectively.
        \item The third map arises by applying $\gnerve[m]$ to an explicit isomorphism of graphs.
        Formally, this isomorphism is defined using the universal property of the colimit applied to the cone diagram:
        \[ \begin{tikzcd}
            {} & \reali[m]{X} \ar[ld, "{r_0(\reali[m]{f}, \id)}"'] \ar[rd, "i_0"] & {} & \reali[m]{X} \ar[rd, "{\ell_0(\reali[m]{f}, \id)}"] \ar[ld, "i_m"'] & {} \\
            \Cyl{m}(\reali[m]{f}, \id) \ar[rrdd, bend right, "{\ell_m(\reali[m]{f}, \reali[m]{g})}", near start] & {} & \reali[m]{X} \gtimes I_m \ar[d, "{(x, t) \mapsto (x, t+m)}"{description, xshift=1ex}] & {} & \Cyl{m}(\id, \reali[m]{g}) \ar[lldd, bend left, "{r_m(\reali[m]{f}, \reali[m]{g})}"', near start] \\[+1ex]
            {} & {} & \reali[m]{X} \gtimes I_{3m} \ar[d] & {} & {} \\[-1ex]
            {} & {} & \Cyl{3m}(\reali[m]{f}, \reali[m]{g}) & {} & {}
        \end{tikzcd} \]
        Intuitively, the graph $\Cyl{m} \big( r_{0}(\reali[m]{f}, \id ), \ell_0(\id , \reali[m]{g}) \big)$ attaches two mapping cylinders of length $m$ to each end of the cylinder $\reali[m]{X} \gtimes I_m$; this construction is isomorphic to the double-mapping cylinder $\Cyl{3m}(\reali[m]{f}, \reali[m]{g})$ of length $3m$.
        \item The fourth and final map is given by applying $\gnerve[m]$ to the map  $\Pi_*$ instantiated at $p = m$, $q_1 = q_2 = 1$, and $q_1' = q_2' = m-1$.
    \end{enumerate}
\end{construction}
\begin{theorem} \label{cyl-cset-graph-equiv-zigzag}
    Let $f \from X \to Y$ and $g \from X \to Z$ be cubical maps.
    Given $n \geq 0$ and $m \geq \min(n, 1)$, if the unit maps
    \[ \eta_X \from X \to \gnerve[m]{\reali[m]{X}} \quad \eta_Y \from Y \to \gnerve[m]{\reali[m]{Y}} \quad \eta_Z \from Z \to \gnerve[m]{\reali[m]{Z}} \] 
    are $n$-equivalences then each morphism in the zig-zag
    \[ \begin{tikzcd}
        \Cyl{}^{\boxcat}(f, g) &[+5em] \ar[l, "{\Cyl{}^{\boxcat}(\id[X], \sigma(f), \sigma(g))}"'] \Cyl{}^{\boxcat} \big( r^{\boxcat}(f, \id), \ell^{\boxcat}(\id, g) \big) \ar[r, "\eta"] &[-0.6em] \gnerve[m]{ \Big( \Cyl{m} \big( r_{0}(\reali[m]{f}, \id ), \ell_0(\id , \reali[m]{g}) \big) \Big)} \ar[ld, "\cong" description] \\
        {} & \gnerve[m] \big( \Cyl{3m}(\reali[m]{f}, \reali[m]{g}) \big) \ar[r, "{\gnerve[m]{(\Pi_*)}}"'] & \gnerve[m]{ \big( \Cyl{m+2}(\reali[m]{f}, \reali[m]{g}) \big)}
    \end{tikzcd} \]
    of \cref{construction:zig-zag} is an $n$-equivalence and natural in $(f, g) \in \Fun(\Lambda, \Graph)$.
\end{theorem}
\begin{proof}
    Naturality is straightforward to verify.
    The first map is a weak equivalence because the functor $\Cyl{}^{\boxcat}$ sends objectwise weak equivalences to weak equivalences.
    The second map is an $n$-equivalence by \cref{cyl-cset-graph-equiv-mono}; here we use the assumption that $m \geq 1$ so that $r^{\boxcat}(f, \id)$ and $\ell^{\boxcat}(\id, g)$ are monomorphisms.
    The third map is an isomorphism, hence an $n$-equivalence.
    The fourth map is an $n$-equivalence by \cref{pi-map-n-equiv}.
\end{proof}
We instantiate \cref{cyl-cset-graph-equiv-zigzag} to some examples of interest in \cref{sec:applications}.

 \section{The inverse construction} \label{sec:inverse-construction}

To show that the nerve functor $\gnerve[\infty] \from \Graph \to \cSet_n$ is an equivalence on localizations at $n$-equivalences, we need to give a construction going in the opposite direction.
We do this by first constructing a functor $F_! \from \cSeti \to \cSet$ from semicubical sets to cubical sets, then using the $m$-realization functor $\reali[m]{\uvar} \from \cSet \to \Graph$ to obtain a graph.
Note that every cubical set is weakly equivalent to a semicubical set (by \cref{nonsingular-approximation-cospan}), so it is no issue that the functor $F_!$ is not defined on the entire category of cubical sets.

The defining property of $F_!$ is that $F_! \cube{n}$ is isomorphic to a cone on $F_! \bd \cube{n} $, and the $n$-cube boundary inclusion $F_!  \bd \cube{n}  \to F_! \cube{n}$ coincides with the cone inclusion.
This allows us to use the comparison results for double-mapping cylinders in $\cSet$ and $\Graph$ to better understand the $m$-realization of $F_!$.

The functor $F_! \from \cSeti \to \cSet$ will be constructed as the extension by colimits of a diagram $F \from \boxcati \to \cSet$.
% Recall the \emph{latching object}
\begin{construction}
    We construct a diagram $F \from \boxcati \to \cSet$ by induction as a sequence of functors
    \[ F^n \from \boxcatin \to \cSet \]
    which commute with the inclusions
    \[ \begin{tikzcd}
        {\boxcatin[0]} \ar[r, hook] \ar[rrd, "F^0"'] & {\boxcatin[1]} \ar[r, hook] \ar[rd, "F^1"] & {\boxcatin[2]} \ar[r, hook] \ar[d, "F^2"] & {\dots} \\
        {} & {} & \cSet
    \end{tikzcd} \]
    Since $\boxcati$ is a direct category, by \cite[Lem.~3.10]{riehl-verity:reedy}, it suffices to produce an object $F^n \cube{n}$ in $\cSet$ together with a map
    \[ L_{\cube{n}} F^{n-1} \to F^n \cube{n} \]
    from the latching object of $\cube{n}$ with respect to $F^{n-1}$.

    We proceed by induction on $n$.
    For $n = 0$, we declare $F^0 \cube{0}$ to be the 0-cube $\cube{0}$.
    For $n \geq 1$, we declare $F^{n} \cube{n}$ to be the cone on the latching object $L_{\cube{n}} F^{n-1}$.
    That is, $F^n \cube{n}$ is defined by the pushout
    \[ \begin{tikzcd}
        L_{\cube{n}} F^{n-1} \ar[r] \ar[d, "i_1"'] \ar[rd, phantom, "\potick" very near end] & \cube{0} \ar[d] \\
        L_{\cube{n}} F^{n-1} \gprod \cube{1} \ar[r] & F^n \cube{n}
    \end{tikzcd} \]
    The latching map $L_{\cube{n}} F^{n-1} \to F^n \cube{n}$ is given by the composite
    \[ L_{\cube{n}} F^{n-1} \xrightarrow{i_0} L_{\cube{n}} F^{n-1} \gprod \cube{1} \to F^n \cube{n}. \]
    \Cref{fig:F-latching-map} depicts the latching map for small values of $n$.
    % Given 
\end{construction}

% \begin{example}
    
% \end{example}

Recall that a functor $D \from \boxcati \to \cSet$ is \emph{Reedy cofibrant} if the latching map $L_{\cube{n}} D \to D \cube{n}$ is a monomorphism.
\begin{proposition} \label{F-is-Reedy-cofib-replacement}
    The diagram $F \from \boxcati \to \cSet$ is a Reedy cofibrant replacement of the terminal diagram $\const{\cube{0}} \from \boxcati \to \cSet$.
\end{proposition}
\begin{proof}
    The diagram $F$ is Reedy cofibrant since the composite
    \[ L_{\cube{n}} F^{n-1} \xrightarrow{i_0} L_{\cube{n}} F^{n-1} \gprod \cube{1} \to F^n \cube{n} \]
    is a monomorphism.
    Since $F \cube{n}$ is defined as a pushout
    \[ \begin{tikzcd}
        L_{\cube{n}} F^{n-1} \ar[r] \ar[d, "i_1"'] \ar[rd, phantom, "\potick" very near end] & \cube{0} \ar[d] \\
        L_{\cube{n}} F^{n-1} \gprod \cube{1} \ar[r] & F^n \cube{n}
    \end{tikzcd} \]
    and the left map is an acyclic cofibration, the cone point inclusion $\cube{0} \to F^n \cube{n}$ is an acyclic cofibration, hence a weak equivalence.
\end{proof}

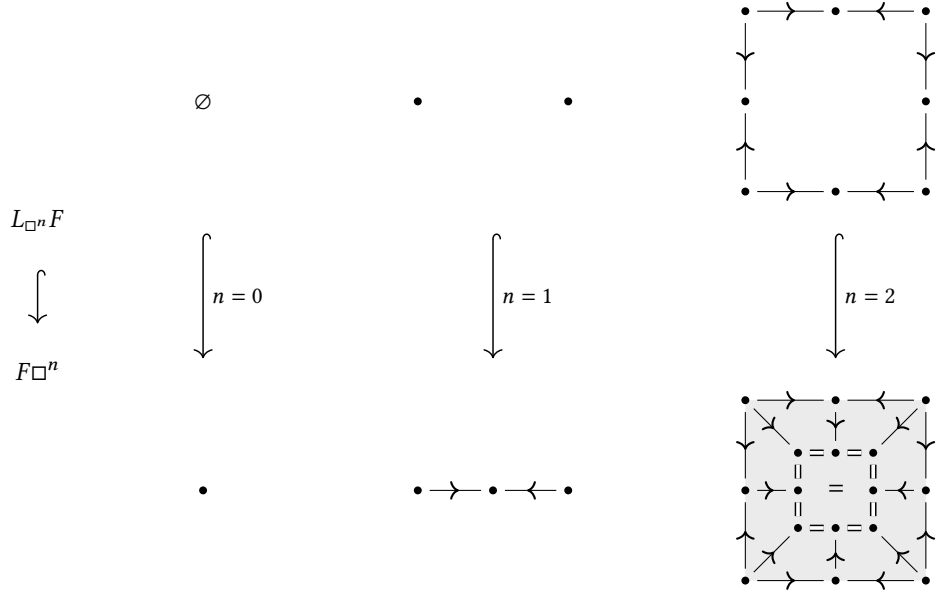
\begin{figure}[H]
    \centering
    \begin{tikzpicture}[
            cell/.style={rectangle, draw=none, outer sep=0pt, inner sep = 0pt, minimum width = 6em, minimum height = 1em},
            commutative diagrams/every diagram,
            decoration={markings,mark=at position 0.57 with {\arrow[thick]{>}}},
            scale=1
        ]
        % Invisible cells control where the arrow starts and ends
        \node[cell] (L0) {};
        \path (L0.east) -- +(2.8, 0) node[cell] (L1) {};
        \path (L1.east) -- +(3.5, 0) node[cell] (L2) {};
        \path (L0.south) -- +(0, -5) node[cell] (F0) {};
        \path (L1.south) -- +(0, -5) node[cell] (F1) {};
        \path (L2.south) -- +(0, -5) node[cell] (F2) {};
        \draw[commutative diagrams/.cd, every arrow, every label, hook, shorten=4.5em]
            (L0.south) edge node[coordinate] (CL) {} (F0.north)
            (L1.south) edge node[coordinate] (C) {} (F1.north)
            (L2.south) to node[coordinate] (CR) {} (F2.north);
        
        % L_0
        \node at (L0.center) {$\varnothing$};
        % F0
        \node[Zcube, label] at (F0.center) {};

        % L_1
        \path (L1.center) -- +(-1, 0) node[Zcube] (L1a) {};
        \path (L1.center) -- +(1, 0) node[Zcube] (L1c) {};
        % F1
        \node[Zcube] (F1b) at (F1.center) {};
        \path (F1.center) -- +(-1, 0) node[Zcube] (F1a) {};
        \path (F1.center) -- +(1, 0) node[Zcube] (F1c) {};
        \draw[postaction={decorate}] (F1a) -- (F1b);
        \draw[postaction={decorate}] (F1c) -- (F1b);

        % L2
        \path (L2.center) -- +(-1.2, 1.2) node[Zcube] (L2TL) {};
        \path (L2.center) -- +(0, 1.2) node[Zcube] (L2T) {};
        \path (L2.center) -- +(1.2, 1.2) node[Zcube] (L2TR) {};
        \path (L2.center) -- +(-1.2, 0) node[Zcube] (L2L) {};
        % \path (L2.center) -- +(0, 0) node[Zcube] (L2C) {};
        \path (L2.center) -- +(1.2, 0) node[Zcube] (L2R) {};
        \path (L2.center) -- +(-1.2, -1.2) node[Zcube] (L2BL) {};
        \path (L2.center) -- +(0, -1.2) node[Zcube] (L2B) {};
        \path (L2.center) -- +(1.2, -1.2) node[Zcube] (L2BR) {};
        \draw[postaction={decorate}] (L2TL) -- (L2T);
        \draw[postaction={decorate}] (L2TR) -- (L2T);
        \draw[postaction={decorate}] (L2TL) -- (L2L);
        \draw[postaction={decorate}] (L2TR) -- (L2R);
        \draw[postaction={decorate}] (L2BL) -- (L2B);
        \draw[postaction={decorate}] (L2BR) -- (L2B);
        \draw[postaction={decorate}] (L2BL) -- (L2L);
        \draw[postaction={decorate}] (L2BR) -- (L2R);

        % F2
        \path (F2.center) -- +(-1.2, 1.2) node[Zcube] (F2TL) {};
        \path (F2.center) -- +(0, 1.2) node[Zcube] (F2T) {};
        \path (F2.center) -- +(1.2, 1.2) node[Zcube] (F2TR) {};
        \path (F2.center) -- +(-1.2, 0) node[Zcube] (F2L) {};
        % \path (L2.center) -- +(0, 0) node[Zcube] (L2C) {};
        \path (F2.center) -- +(1.2, 0) node[Zcube] (F2R) {};
        \path (F2.center) -- +(-1.2, -1.2) node[Zcube] (F2BL) {};
        \path (F2.center) -- +(0, -1.2) node[Zcube] (F2B) {};
        \path (F2.center) -- +(1.2, -1.2) node[Zcube] (F2BR) {};
        \draw[postaction={decorate}] (F2TL) -- (F2T);
        \draw[postaction={decorate}] (F2TR) -- (F2T);
        \draw[postaction={decorate}] (F2TL) -- (F2L);
        \draw[postaction={decorate}] (F2TR) -- (F2R);
        \draw[postaction={decorate}] (F2BL) -- (F2B);
        \draw[postaction={decorate}] (F2BR) -- (F2B);
        \draw[postaction={decorate}] (F2BL) -- (F2L);
        \draw[postaction={decorate}] (F2BR) -- (F2R);
        % Inner layer
        \path (F2.center) -- +(-0.5, 0.5) node[Zcube] (F2TLi) {};
        \path (F2.center) -- +(0, 0.5) node[Zcube] (F2Ti) {};
        \path (F2.center) -- +(0.5, 0.5) node[Zcube] (F2TRi) {};
        \path (F2.center) -- +(-0.5, 0) node[Zcube] (F2Li) {};
        % \path (F2.center) -- +(0, 0) node[Zcube] (F2Ci) {};
        \path (F2.center) -- +(0.5, 0) node[Zcube] (F2Ri) {};
        \path (F2.center) -- +(-0.5, -0.5) node[Zcube] (F2BLi) {};
        \path (F2.center) -- +(0, -0.5) node[Zcube] (F2Bi) {};
        \path (F2.center) -- +(0.5, -0.5) node[Zcube] (F2BRi) {};
        \draw[commutative diagrams/.cd, every arrow, equal]
            (F2TLi) -- (F2Ti)
            (F2TRi) -- (F2Ti)
            (F2TLi) -- (F2Li)
            (F2TRi) -- (F2Ri)
            (F2BLi) -- (F2Bi)
            (F2BRi) -- (F2Bi)
            (F2BLi) -- (F2Li)
            (F2BRi) -- (F2Ri);
        \draw[postaction={decorate}] (F2TL) -- (F2TLi);
        \draw[postaction={decorate}] (F2T) -- (F2Ti);
        \draw[postaction={decorate}] (F2TR) -- (F2TRi);
        \draw[postaction={decorate}] (F2L) -- (F2Li);
        \draw[postaction={decorate}] (F2R) -- (F2Ri);
        \draw[postaction={decorate}] (F2BL) -- (F2BLi);
        \draw[postaction={decorate}] (F2B) -- (F2Bi);
        \draw[postaction={decorate}] (F2BR) -- (F2BRi);
        
        \begin{scope}[on background layer]
            \path[fill=grayfill]
                (F2TL.center) -- (F2T.center) -- (F2Ti.center) -- (F2TLi.center) -- cycle;
            \path[fill=grayfill]
                (F2TR.center) -- (F2T.center) -- (F2Ti.center) -- (F2TRi.center) -- cycle;
            \path[fill=grayfill]
                (F2TL.center) -- (F2L.center) -- (F2Li.center) -- (F2TLi.center) -- cycle;
            \path[fill=grayfill]
                (F2TR.center) -- (F2R.center) -- (F2Ri.center) -- (F2TRi.center) -- cycle;
            \path[fill=grayfill]
                (F2BL.center) -- (F2L.center) -- (F2Li.center) -- (F2BLi.center) -- cycle;
            \path[fill=grayfill]
                (F2BL.center) -- (F2B.center) -- (F2Bi.center) -- (F2BLi.center) -- cycle;
            \path[fill=grayfill]
                (F2BR.center) -- (F2B.center) -- (F2Bi.center) -- (F2BRi.center) -- cycle;
            \path[fill=grayfill]
                (F2BR.center) -- (F2R.center) -- (F2Ri.center) -- (F2BRi.center) -- cycle;
            \path[fill=grayfill]
                (F2TLi.center) rectangle node {$=$} (F2BRi.center);
        \end{scope}
    
        % Headings
        \path (CL) -- +(-2.2, 1) node (HT) {$L_{\cube{n}} F$};
        \path (CL) -- +(-2.2, -1) node (HB) {$F \cube{n}$};
        \draw[commutative diagrams/.cd, every arrow, every label, hook, shorten=1.1em]
            (HT) to (HB);
        \node[anchor=west, font=\small] at (CL.east) {$n = 0$}; 
        \node[anchor=west, font=\small] at (C.east) {$n = 1$}; 
        \node[anchor=west, font=\small] at (CR.east) {$n = 2$}; 
    \end{tikzpicture}
    % \begin{tabular}{r c c c}
    %     {} & $n = 0$ & $n = 1$ & $n = 2$ \\
    %     $L_{\cube{n}} F$ & {} & {} & {} \\
    %     $F \cube{n}$ & {} & {} & {}
    % \end{tabular}
    \caption{Depictions of the latching maps for $F$.} \label{fig:F-latching-map}
\end{figure}

Extending $F$ by colimits yields a cocontinuous functor
\[ F_! \from \cSeti \to \cSet . \]
By construction, the sequence $F_! \bd \cube{n} \to F_! \cube{n} \to \cube{0}$ is a mapping cylinder factorization.
That is,
\begin{proposition} \label{F-is-mapping-cone}
    For all $n \geq 0$, the image of the boundary inclusion $i \from \bd \cube{n} \to \cube{n}$ under $F_!$ factors through the bottom map of a pushout square as in the diagram:
    \[ \begin{tikzcd}
        {} & F_!(\bd \cube{n}) \ar[r] 
        \ar[d, hook, "i_1"'] \ar[rd, phantom, "\potick" very near end] & \cube{0} \ar[d] \\
        F_!(\bd \cube{n}) \ar[r, "i_0"] \ar[rr, bend right=18, "F_! i"'] & F_!(\bd \cube{n}) \gprod \cube{1} \ar[r] & F_!(\cube{n})
    \end{tikzcd} \]
\end{proposition}
\begin{proof}
    The boundary inclusion $F_! i \from F_! \bd \cube{n} \to F_! \cube{n}$ coincides with the latching map $L_{\cube{n}} F^{n-1} \to F^n \cube{n}$, from which the result follows.
\end{proof}

% Using the Reedy model structure on $\Fun(\boxcati, \cSet)$, we may choose a diagram $R \from \boxcati \to \cSet$ and an acyclic cofibration $i \weto R$ so that $R$ is both Reedy fibrant and Reedy cofibrant.
% % Since each $n$-cube $\cube{n}$ is weakly contractible, the map $R \to \const{\cube{0}}$ is an acylic fibration.
% Extending by colimits yields functors $i_!, R_! \from \cSeti \to \cSet$; the morphism of diagrams $i \to R$ induces a natural transformation $i_! \to R_!$.
% % We extend by colimits to obtain functors $i_!, R_!, F_! \from \cSeti \to \cSet$.
% \begin{theorem}
%     % The natural transformations
%     % where the map $i_! X \weto X_1$ is an acyclic cofibration.
%     % Moreover, the assignment $X \mapsto X_1$ is functorial in 
%     Let $X$ be a semicubical set.
%     \begin{enumerate}
%         \item There exists a natural transformation $F_! \to R_!$ such that $F_! X \to R_! X$ is a weak equivalence.
%         \item The natural map $i_! X \to R_! X$ is an acyclic cofibration.
%     \end{enumerate}
% \end{theorem}
Let $i$ denote the composite $\boxcati \ito \boxcat \xrightarrow{\yo} \cSet$, and $i_! \from \cSeti \to \cSet$ denote the extension by colimits.
\begin{theorem} \label{F-shriek-approx-cospan}
    There exists a functor $L \from \cSeti \to \cSet$ together with natural weak equivalences
    \[ i_! \weto L \weot F_! \]
    such that, 
    \begin{enumerate}
        \item $L$ admits a right adjoint; and
        \item for any semicubical set $X$, the map $i_! X \to L X$ is an acyclic cofibration.
    \end{enumerate}
\end{theorem}
\begin{proof}
    % Let $i$ denote the composite of inclusions $\boxcati \ito \boxcat \xrightarrow{\yo} \cSet$, which is a Reedy cofibrant object of $\cSet^{\boxcati}$.
    We fix a fibrant replacement $i_! \to \ell \to \const{\cube{0}}$ of $i_!$ in the Reedy model structure on $\Fun(\boxcati, \cSet)$.
    In particular, $i_! \to \ell$ is an acyclic cofibration, and $\ell$ is both fibrant and cofibrant.
    We define the functor $L$ to be the extension by colimits $L := \ell_!$.
    % Fix an acyclic cofibration $i \weto \ell$ in the Reedy model structure where $\ell$ is both fibrant and cofibrant, and define $L$ to be the extension by colimits 
    
    To construct the natural transformation $F_! \to \ell_!$, observe that the morphism of diagrams $\ell \to \const{\cube{0}}$ is an acylic fibration (since $i \to \const{\cube{0}}$ is a weak equivalence).
    As $F$ is a Reedy cofibrant diagram (\cref{F-is-Reedy-cofib-replacement}), the square
    \[ \begin{tikzcd}
        \const{\varnothing} \ar[r] \ar[d] & \ell \ar[d] \\
        F \ar[r] & \const{\cube{0}}
    \end{tikzcd} \]
    admits a lift $F \to \ell$, which is moreover a weak equivalence by 2-out-of-3.

    After taking extension by colimits, each of the functors $i_!$, $\ell_!$, and $F_!$ sends the boundary inclusion $\bd \cube{n} \to \cube{n}$ to a cofibration, since each starting diagram $i, \ell, F$ is Reedy cofibrant.
    Given a semicubical set $X$ and a natural number $n \geq 0$, we have a pushout square
    \[ \begin{tikzcd}
        \coprod\limits_{X_n} \bd \cube{n} \ar[r] \ar[d, hook] \ar[rd, phantom, "\potick" very near end] & \Sk^{n-1} X \ar[d, hook] \\ 
        \coprod\limits_{X_n} \cube{n} \ar[r] & \Sk^n  X
    \end{tikzcd} \]
    in $\cSeti$, which is preserved by each of $i_!$, $\ell_!$, and $F_!$.
    The maps $i \to \ell$ and $F \to \ell$ induce natural transformations between pushout diagrams.
    A skeletal induction argument therefore shows that the induced natural transformations $i_! \to \ell_! \leftarrow F_!$ are natural weak equivalences.
    We conclude by observing that the map $i_! \to \ell_!$ is a monomorphism because the starting morphism of diagrams $i \to \ell$ was a monomorphism by assumption.
\end{proof}

\begin{construction}
    Let $X$ be a semicubical set.
    Given a non-negative integer $m \geq 0$, we define a graph $\Gn{X}$ to be the realization of $F_! X$ of length $m$.
    That is,
    \[ \Gn{X} := \reali[m]{F_! X}. \]
    \Cref{ex:Gamma-ncube-012} showcases the graphs $\Gn{\cube{n}}$ for small values of $m$ and $n$.
    % We define the subgraph $U_{X, n} \subseteq G_{X, n}$ to be the image of the 
\end{construction}
The unit of the realization-nerve adjunction, when instantiated at $F_! X$, yields a map
\[ \eta^m_{F_! X} \from F_! X \to \gnerve[m] \Gn{X} \]
which is natural in the variable $X \in \cSeti$.

The following result holds by definition of $F_!$.
\begin{proposition} \label{Gn-cube-n-is-cyl}
    Let $m, n$ be non-negative.
    The graph $\Gn{\cube{n}}$ is isomorphic to the length-$m$ double mapping cylinder
    \[ \Gn{\cube{n}} \cong \Cyl{m} \big( \id[\Gn{\bd \cube{n}}], \, \bang_{\Gn{\bd \cube{n}}} \big) \]
    of $\id[\Gn{\bd \cube{n}}] \from \Gn{\bd \cube{n}} \to \Gn{\bd \cube{n}}$ and $\bang_{\Gn{\bd \cube{n}}} \from \Gn{\bd \cube{n}} \to I_0$.
    Under this identification, the map $\Gn{\bd \cube{n}} \ito \Gn{\cube{n}}$ induced by the cubical boundary inclusion coincides with the map $\ell_0(\id, \bang) \from \Gn{\bd \cube{n}} \to \Cyl{m}(\id, \bang)$.
\end{proposition}
\begin{proof}
    % The $m$-realization functor $\reali[m]{\uvar} \from \cSet \to \Graph$ preserves colimits and is (strong) monoidal with respect to the geometric product on $\cSet$ and the box product on $\Graph$.
    Follows from \cref{F-is-mapping-cone} and \cref{cyl-reali-iso}.
\end{proof}
\begin{corollary} \label{G-cube-contractible}
    For $m, n \geq 0$, the graph $\Gn{\cube{n}}$ is contractible.
\end{corollary}
\begin{proof}
    By \cref{cyl-inclusion-is-htpy-equiv}, the map $r_0(\id, \bang) \from I_0 \to \Gn{\cube{n}}$ is a homotopy equivalence.
\end{proof}
\begin{corollary} \label{Gn-bd-inclusion-hopo}
    Let $m$ and $n$ be a non-negative integers with $m \geq n+2$.
    Suppose $u \from \bd \cube{k} \to X$ is a map of semicubical sets.
    Then, for any $m \in \NN \cup \{ \infty \}$, the pushout 
    \[ \begin{tikzcd}
        \bd \cube{k} \ar[r, "u"] \ar[d, hook] \ar[rd, phantom, "\potick" very near end] & X \ar[d, hook] \\
        \cube{k} \ar[r] & X_u
    \end{tikzcd} \]
    in $\cSeti$ is sent to a homotopy pushout of $n$-types
    \[ \begin{tikzcd}[column sep = 3.8em]
        \gnerve[m] \Gn {(\bd \cube{n})} \ar[r, "{\gnerve[m] \Gn{(u)}}"] \ar[d, hook] & \gnerve[m] \Gn {(X)} \ar[d, hook] \\ 
        \gnerve[m] \Gn {(\cube{n})} \ar[r] & \gnerve[m] \Gn {(X_u)}
    \end{tikzcd} \]
    after applying the functor $\gnerve[m] \circ \ \Gn \from \cSeti \to \cSet$.
    Additionally, if $u$ is a monomorphism then the result holds when $m = n+1$.
\end{corollary}
\begin{proof}
    The functor $\Gn \from \cSet \to \Graph$ preserves pushouts, hence
    \[ \begin{tikzcd}[column sep = 2.8em]
        \Gn {(\bd \cube{k})} \ar[r, "{\Gn{(u)}}"] \ar[d, hook, "{\ell_0(\id, !)}"'] & \Gn {(X)} \ar[d, hook] \\ 
        \Gn {(\cube{k})} \ar[r] & \Gn {(X_u)}
    \end{tikzcd} \]
    is a pushout, and the left map is identified with $\ell_0(\id, \bang)$ by \cref{Gn-cube-n-is-cyl}.
    Apply \cref{cyl-factorization-hopo} with $f = \Gn{(u)}$ and $g = \bang$.
\end{proof}

Our main theorem in this section is the following:
\begin{theorem} \label{Gn-weq}
    Let $m, n$ be non-negative integers with $m \geq n+2$.
    For any semicubical set $X$, the map
    \[ \eta^m_{F_! X} \from F_! X \to \gnerve[m] \Gn{X} \]
    is an $n$-equivalence.
    Additionally, if $X$ is nonsingular then the result holds when $m = n+1$.
\end{theorem}
% Before proving \cref{Gn-weq}, we prove the following lemmas.
% \begin{lemma}
%     Let $n$ and $k$ be non-negative integers.
%     The graph $\Gn{\cube{n}}$ is contractible.
% \end{lemma}
% \begin{proof}
%     The $k$-realization functor $\reali[k]{\uvar} \from \cSet \to \Graph$ preserves colimits, so by \cref{F-is-mapping-cone}, the square
%     % We rewrite the square
%     % \[ \begin{tikzcd}
%     %     \Gn{\bd \cube{n}} \ar[r] \ar[d] \ar[rd, phantom, "\potick" very near end] & \Gn{\cube{0}} \ar[d] \\
%     %     \Gn{(\bd \cube{n} \gprod \cube{1})} \ar[r] & \Gn{\cube{n}}
%     % \end{tikzcd} \]
%     % as
%     \[ \begin{tikzcd}
%         \reali[k]{F_! \bd \cube{n}} \ar[r] \ar[d] \ar[rd, phantom, "\potick" very near end] & \reali[k]{\cube{0}} \ar[d] \\
%         \reali[k]{F_!(\bd \cube{n}) \gprod \cube{1}} \ar[r] & \reali[k]{F_! \cube{n}}
%     \end{tikzcd} \]
%     is a pushout.
%     We rewrite this square as
%     \[ \begin{tikzcd}
%         \Gn{(\bd \cube{n})} \ar[r] \ar[d] \ar[rd, phantom, "\potick" very near end] & I_0 \ar[d] \\
%         \Gn{(\bd \cube{n})} \gtimes I_k \ar[r] & \Gn{(\cube{n})}
%     \end{tikzcd} \]
%     and observe that this is the pushout that defines the cone on $\Gn{(\bd \cube{n})}$ of length $k$.
%     By \cref{cyl-inclusion-is-htpy-equiv}, the map $I_0 \to \Gn{\cube{n}}$ is a homotopy equivalence.
% \end{proof}
\begin{proof}%[Proof of \cref{Gn-weq}]
    It suffices to consider the case when $X \in \cSeti$ is finite.
    This is because an arbitrary semicubical set can be written as a filtered colimit of its finite subobjects.
    The functors $F_!$, $\reali[m]{\uvar}$ and $\gnerve[m]$ all preserve filtered colimits, and $n$-equivalences are stable under filtered colimits; thus, if $\eta_{F_! Y}$ is an $n$-equivalence for each finite subobject $Y \subseteq X$ then $\eta_{F_! X}$ is an $n$-equivalence.

    % We proceed by induction on $n$ and observe that, in the base case $n = 0$, the map $\eta^m_{F_! X}$ is a 0-equivalence for any $X$.
    % For the inductive step, fix $n \geq 1$ and suppose $\eta^m_{F_! X}$ is an $(n-1)$-equivalence for all $k \geq n+1$ and $X \in \cSeti$ finite.

    % Fix $k \geq n+2$ and $X \in \cSeti$ finite.
    % We wish to show $\eta^m_{F_! X}$ is an $n$-equivalence for all $k \geq n+2$ and $X \in \cSeti$ finite.
    % For this, 
    It suffices to show the map
    \[ \eta^m_{F_! (\Sk^k X)} \from F_! (\Sk^k X) \to \gnerve[m] \Gn{(\Sk^k X)} \]
    is an $n$-equivalence for all $k \geq 0$ (again, since $n$-equivalences are stable under filtered colimits).
    We proceed by induction on $k$, and observe that if $k = 0$ then $\eta^m_{F_! (\Sk^0 X)}$ is an isomorphism between discrete cubical sets.
    For the inductive step, fix $k \geq 0$ and suppose $\eta^m_{F_! \Sk^k X}$ is an $n$-equivalence for all $m \geq n+2$ and $X \in \cSeti$ finite.

    % Fix $k \geq n+2$ and $X$
    % Our goal is to show that $\eta^m_{F_! \Sk^{j+1} X}$ is an $n$-equivalence for all $k \geq n+2$ and $X \in \cSeti$ finite.
    % To this end, fix $k$ and $X$.
    Now, fix a finite semicubical set $X$.
    Since $X$ is finite, we may decompose $\Sk^{k+1} X$ into a sequence of inclusions
    \[ \Sk^{k} X = X_0 \subseteq X_1 \subseteq \dots \subseteq X_d = \Sk^{k+1} X \]
    where each inclusion map $X_i \ito X_{i+1}$ appears in a pushout
    \[ \begin{tikzcd}
        \bd \cube{j+1} \ar[r] \ar[d, hook] \ar[rd, phantom, "\potick" very near end] & X_i \ar[d, hook] \\
        \cube{j+1} \ar[r] & X_{i+1}
    \end{tikzcd} \]
    of semicubical sets.
    % for some $j \geq 0$.
    We perform induction on $i \in \{ 0, \dots, d \}$ and observe that, in the base case $i = 0$, the map $\eta^m_{X_0}$ is an isomorphism by our inductive hypothesis on $k$.

    Fix $i \in \{ 0, \dots, d-1 \}$ and assume $\eta^m_{X_i}$ is an $n$-equivalence.
    Both functors $F_! \from \cSeti \to \cSet$ and $\Gn{} \from \cSeti \to \Graph$ preserve colimits, hence we obtain two pushout squares
    \[ \begin{tikzcd}
        F_! \bd \cube{k+1} \ar[r] \ar[d, hook] \ar[rd, phantom, "\potick" very near end] & F_! X_i \ar[d, hook] \\
        F_! \cube{k+1} \ar[r] & F_! X_{i+1}
    \end{tikzcd} \qquad \begin{tikzcd}
        \Gn{\bd \cube{k+1}} \ar[r] \ar[d, hook] \ar[rd, phantom, "\potick" very near end] & \Gn{X_i} \ar[d, hook] \\
        \Gn{\cube{k+1}} \ar[r] & \Gn{X_{i+1}}
    \end{tikzcd} \]
    The natural transformation $\eta^m$ induces a commutative cube
    \[ \begin{tikzcd}
        F_! \bd \cube{k+1} \ar[rr] \ar[rd] \ar[dd] & {} & F_! X_i \ar[dd] \ar[rd] & {} \\
        {} & \gnerve[m]{(\Gn{(\bd \cube{k+1})})} \ar[rr] & {} & \gnerve[m]{(\Gn{(X_i)})} \ar[dd, crossing over] \\
        F_! \cube{k+1} \ar[rr] \ar[rd] & {} & F_! X_{i+1} \ar[rd] & {} \\
        {} & \gnerve[m]{(\Gn{(\cube{k+1})})} \ar[rr] \ar[from=uu, crossing over] & {} & \gnerve[m]{(\Gn{(X_{i+1})})}
    \end{tikzcd} \]
    Note the back square is a pushout along a monomorphism, hence it is a homotopy pushout of $n$-types.
    The front square is a homotopy pushout of $n$-types by \cref{Gn-bd-inclusion-hopo}.
    The map between the top right objects is an $n$-equivalence by the inductive hypothesis on $i$.
    The map between the top left objects is a $n$-equivalence by the inductive hypothesis on $k$.
    The map between the bottom left objects is a weak equivalence by \cref{G-cube-contractible}.
    Thus, the bottom right map is an $n$-equivalence as desired.
    % The functor $\Gn{\uvar} \from \cSeti \to \Graph$ preserves colimits, so we obtain a pushout
    % \[ \begin{tikzcd} 
    %     \Gn {\bd \cube{n+1}} \ar[r, hook] \ar[d, hook] & \Gn{X'} \ar[d, hook] \\
    %     \Gn{ \cube{n+1}} \ar[r, hook] & \Gn{X}
    % \end{tikzcd} \]
    % of graphs.
    % Since $n$-equivalences are stable under filtered colimits, it suffices to show
\end{proof}

\subsection{Explicit description in terms of $F$-sequences.}

We conclude this section with an element-wise description of the graphs $\Gn{\cube{n}}$.
This will be a closed-form description, as opposed to the inductive definition in terms of taking cones of the boundary $\Gn{\bd \cube{n}}$.

The vertex set of $\Gn{\cube{n}}$ can be described as the set of \emph{$F$-sequences} modulo \emph{$F$-reduction}, two concepts which we now define.
\begin{definition}
    For $m, n \geq 0$, an \emph{$F$-sequence} is a tuple $(\varphi, k, P, w)$ where:
    \begin{itemize}
        \item $\varphi$ is a function $\{ 1, \dots, n \} \to \{ -, + \}$, referred to as the \emph{sign} function
        \item $k$ is a non-negative integer, referred to as the \emph{length};
        \item $P$ is a surjective function $\{ 1, \dots, n \} \to \{ 1, \dots, k \}$, referred to as the \emph{partition}; and
        \item $w$ is a function $\{ 1, \dots, k \} \to \{ 0, \dots, m \}$, referred to as the \emph{weight function}.
    \end{itemize}
\end{definition}
Our convention is that if $n = 0$ then there is a unique $F$-sequence $(\bang, 0, \id[\varnothing], \bang)$, where $\bang$ always denotes the unique function from the empty set.

Given an $F$-sequence $(\varphi, k, P, w)$ and an element $i \in \{ 1, \dots, k \}$, we write $P_i$ for the pre-image of $i$ under the map $P$.
In this way, the surjection $P$ determines a partition of the set $\{ 1, \dots, n \}$ into $k$ classes; the \emph{weight} of an equivalence class $P_i$ is the value $w(i)$.
% We write $t_i$ for the value $t(i)$.
If we write the elements of $P_i$ as $P_i = \{ x_{i, 1}, \dots, x_{i, \ell(i)} \}$, then an $F$-sequence can be written as an expression of the form:
\[ \big( (w(1); x_{1, 1}^\pm, \dots, x_{1, \ell(1)}^\pm), (w(2); x_{2, 1}^\pm, \dots, x_{2, \ell(2)^\pm}), \dots, (w(k); x_{k, 1}^\pm, \dots, x_{k, \ell(k)}^\pm) \big) \]
where each $x_{i, j}$ is an element of the equivalence class $P_i$ (and thus an element of $\{ 1, \dots, n \}$) and the superscript denotes the value of the sign function at $x_{i, j}$.
For example, if $m = 8$ and $n = 3$, the expression
\[ \big( (7; 1^+, 3^-), (0; 2^-) \big) \]
corresponds to the $F$-sequence where:
\begin{itemize}
    \item the sign function is given by $1 \mapsto +$, $2 \mapsto -$, and $3 \mapsto -$;
    \item the partition is given by $P_1 = \{ 1, 3 \}$ and $P_2 = \{ 2 \}$; and
    \item the weight function is given by $1 \mapsto 7$ and $2 \mapsto 0$.
\end{itemize}
% Note the order of the elements $p_{i, 1}, \dots, p_{}$
Similarly, the expression
\[ \big( (4; 3^+), (1; 1^+), (4; 2^-) \big) \]
corresponds to the $F$-sequence where:
\begin{itemize}
    \item the sign function is given by $1 \mapsto +$, $2 \mapsto -$, and $3 \mapsto +$;
    \item the partition is given by $P_1 = \{ 3 \}$, $P_{2} = \{ 1 \}$, and $P_3 = \{ 2 \}$; and
    \item the weight function is given by $1 \mapsto 4$, $2 \mapsto 1$, and $3 \mapsto 4$.
\end{itemize}
Note the order of the elements $p_{i, 1} \dots, p_{i, \ell(i)}$ within a pair of parantheses does not matter, so $\big( (7; 1^+, 3^-), (0; 2^-) \big)$ and $\big( (7; 3^-, 1^+), (0; 2^-) \big)$ denote the same $F$-sequence.
However, swaps between parantheses denote distinct $F$-sequences, so $\big( (7; 1^+, 3^-), (0; 2^-) \big)$ and $\big( (0; 2^-), (7; 1^+, 3^-) \big)$ are not equal.
This is because re-arranging the order of parantheses changes both the surjection $P$ and the weight function $w$.
% Equivalently, an $F$-sequence is given by a partition of the set $\{ 1, \dots, n \}$ together with a total order on the set of equivalence classes of the partition and a ``weight function'' that assigns a value between 0 and $m$ to each class of the partition.
We introduce the notation $(w; P_i^\pm)_{i=1}^k$ as an alternative way to denote an arbitrary $F$-sequence $(\varphi, k, P, w)$.

We now construct an equivalence relation on the set of $F$-sequences, which we define to be the symmetric transitive closure of the relation given by the \emph{$F$-reduction rules}:
% To define the $F$-reductions, we first introduce some notation.
\begin{enumerate}
    \item[FR1.] Let $(\varphi, k, P, w)$ be an $F$-sequence with $k \geq 2$.
    If $w(i) = 0$ for some $i \in \{ 2, \dots, k \}$ then we define functions 
    \[ P_{\mathrm{FR1}} \from \{ 1, \dots, n \} \to \{ 1, \dots, k-1 \}, \quad w_{\mathrm{FR1}} \from \{ 1, \dots, k-1 \} \to \{ 0, \dots, m \} \] 
    by
    \[ \begin{array}{l@{\hspace{5em}}l}
        P_{\mathrm{FR1}}(x) := \begin{cases}
            P(x) & \text{if } x < i \\
            P(x-1) & \text{if } x \geq i.
        \end{cases} 
        &
        w_{\mathrm{FR1}}(j) := \begin{cases}
            w(j) & \text{if } j < i \\
            w(j+1) & \text{if } j \geq i.
        \end{cases}
    \end{array} \]
    The FR1-reduction rule asserts a relation $(\varphi, k, P, w) \sim_{\mathrm{FR1}} (\varphi, k-1, P_{\mathrm{FR1}}, w_{\mathrm{FR1}})$.

    In words, if an equivalence class $P_i \neq P_1$ has weight $0$ then we may merge the equivalence class $P_{i}$ into the equivalence class $P_{i-1}$.
    The weight of the newly-formed class $P_{i-1} \sqcup P_i$ is equal to the weight of $P_{i-1}$.
    \item[FR2.] Let $(\varphi, k, P, w)$ be an $F$-sequence.
    If $w(i) = m$ for some $i \in \{ 1, \dots, k \}$ then we define functions 
    \[ \varphi_{\mathrm{FR2}} \from \{ 1, \dots, n \} \to \{ -, + \}, \quad P_{\mathrm{FR2}} \from \{ 1, \dots, n \} \to \{ 1, \dots, i \}, \quad w_{\mathrm{FR2}} \from \{ 1, \dots, i \} \to \{ 0, \dots, m \} \] 
    by
    \[ \begin{array}{l@{\hspace{4em}}l@{\hspace{4em}}l}
        \varphi_{\mathrm{FR2}}(j) := \begin{cases}
            \varphi(j) & \text{if } P(j) < i \\
            + & \text{if } P(j) \geq i.
        \end{cases}
        &
        P_{\mathrm{FR2}}(j) := \begin{cases}
            P(j) & \text{if } P(j) < i \\
            i & \text{if } P(j) \geq i.
        \end{cases} 
        &
        w_{\mathrm{FR2}}(i) := w(i).
    \end{array} \]
    The FR2-reduction rule asserts a relation $(\varphi, k, P, w) \sim_{\mathrm{FR2}} (\varphi_{\mathrm{FR2}}, i, P_{\mathrm{FR2}}, w_{\mathrm{FR2}})$.
    
    In words, if some equivalence class $P_i$ has weight $m$ then we may merge the equivalence classes $P_i, P_{i+1}, \dots, P_k$ into a single equivalence class with weight $m$, such that every element in the newly-formed class $P_i \sqcup \dots \sqcup P_{k}$ has positive sign.
\end{enumerate}
We say two $F$-sequences are \emph{related via $F$-reductions} if there is a zig-zag of $F$-reduction rules connecting them.
\begin{definition}
    Let $(w; P_i^\pm)_{i=1}^k$ be an $F$-sequence.
    \begin{enumerate}
        \item We say $(w; P_i^\pm)_{i=1}^k$ is in \emph{reduced form} if:
        \begin{itemize}
            \item $w(i) \neq 0$ whenever $i \neq 0$,
            \item $w(i) \neq m$ whenever $i \neq k$, and
            \item if $w(k) = m$ then every element of $P_k$ has positive sign.
        \end{itemize}
        \item We say $(w; P_i^\pm)_{i=1}^k$ is in \emph{expanded form} if:
        \begin{itemize}
            \item each $P_i$ has exactly one element, and
            \item if $w(i) = m$ for some $i \in \{ 1, \dots, k \}$ then $w(j) = m$ for all $j \geq i$.
        \end{itemize}
    \end{enumerate}
\end{definition}
% For any $F$-sequence $(w; P_i^\pm)_{i=1}^k$, there is a sequence of $F$-reductions relating $(w; P_i^\pm)_{i=1}^k$ to an $F$-sequence in reduced form, and to an $F$-sequence in expanded form.
\begin{proposition} \label{unique-reduced-forms}
    Let $m, n \geq 0$ be non-negative.
    \begin{enumerate}
        \item Every $F$-sequence is related via $F$-reductions to a unique $F$-sequence in reduced form.
        \item Every $F$-sequence is related via $F$-reductions to an $F$-sequence in expanded form.
    \end{enumerate}
\end{proposition}
\begin{proof}
    Fix an $F$-sequence $(w; P_i^\pm)_{i=1}^k$.
    We present an algorithm for finding a reduced form and an expanded form, respectively.
    % Pasted from the survey paper
    % \begin{minipage}{0.5\textwidth}
    %     \begin{algorithm}[H]
    %         \caption*{\textbf{Algorithm}: }
    %         \begin{algorithmic}[1]
    %             \Statex \textbf{Inputs:} The adjacency matrix of a graph $X$
    %             \Statex \phantom{\textbf{Inputs:}} A singular $(n-1)$-cube $u$
    %             \Statex \phantom{\textbf{Inputs:}} A singular $(n-1)$-cube $u'$
    %             \Statex \textbf{Output:} True if $u$ and $u'$ are length-1 homotopic, 
    %             \Statex \phantom{\textbf{Output:}} False otherwise
    %             % \Statex \hspace*{-\leftmargin}\hrulefill
    
    %             \vspace{0.3em}
    
    %             \For{each vertex $i$ in $\gcube{1}{n}$}
    %                 \If{$u(i)$ is not connected to $u'(i)$ in $X$}
    %                     \State \Return False 
    %                 \EndIf
    %             \EndFor 
    %             \State \Return True
    %         \end{algorithmic}
    %     \end{algorithm}
    % \end{minipage}
    \emph{Algorithm for finding a reduced form $F$-sequence}.
    \begin{enumerate}
        \item[] Input: An $F$-sequence $(w; P_i^\pm)_{i=1}^k$.
        \item[] Output: An $F$-sequence $(w'; P_i'^\pm)_{i=1}^{k'}$ in reduced form that is related to $(w; P_i^\pm)_{i=1}^k$ via $F$-reductions.
        \item Define $(w'; P_i'^\pm)_{i=1}^{k'} := (w; P_i^\pm)_{i=1}^k$.
        \item If there exists $i \in \{ 1, \dots, k \}$ such that $w'(i) = m$:
        \begin{enumerate}
            \item Let $i_0$ be the minimal such $i$.
            \item Use FR2 to merge $P'_{i_0}, P'_{i_0 + 1}, \dots, P'_k$ into a single equivalence class with weight $m$, such that every element has positive sign.
            \item Set $k' := i_0$.
        \end{enumerate}
        \item While there exists $i \in \{ 2, \dots, k \}$ such that $w'(i) = 0$:
        \begin{enumerate}
            \item Use FR1 to merge $P'_i$ and $P'_{i-1}$ with weight $w(i-1)$.
            \item Decrement $k'$ by 1.
        \end{enumerate}
        \item Return $(w'; P_i'^\pm)_{i=1}^{k'}$.
    \end{enumerate}
    This algorithm terminates because each iteration of the while loop decreases the number of elements $i \in \{ 2, \dots, k \}$ satisfying $w(i) = 0$.
    By construction, the resulting $F$-sequence is related to $(w; P_i^\pm)_{i=1}^k$ via $F$-reductions, and moreover, the resulting $F$-sequence is in reduced form.

    It remains only to prove that the output $F$-sequence is the unique $F$-sequence in reduced form which is related to $(w; P_i^\pm)_{i=1}^k$ via $F$-reductions.
    Observe that, to the $F$-sequence $(w; P_i^\pm)_{i=1}^k$, we may associate a function $\Sigma \from \{ 1, \dots, n \} \to \NN$ defined by
    \[ \Sigma(j) := \begin{cases}
        \min \{ i \in \NN \mid i \leq P(j) \text{ and } w(i) = m \} & \text{if } \{ i \mid i \leq P(j) \text{ and } w(i) = m \} \text{ is nonempty} \\
        \max \{ i \in \NN \mid i \leq P(j) \text{ and } w(i) \neq 0 \} & \text{if } \{ i \mid i \leq P(j) \text{ and } w(i) = m \} \text{ is empty} \\
        {} & \quad \text{and } \{ i \mid i \leq P(j) \text{ and } w(i) \neq 0 \} \text{ is nonempty} \\
        1 & \text{otherwise.}
    \end{cases} \]
    We refer to $\Sigma$ as the \emph{signature function} of the $F$-sequence $(w; P_i^\pm)_{i=1}^k$.
    The signature function is invariant under $F$-reduction; i.e.\ if two $F$-sequences are related by a sequence of $F$-reductions then their siganture functions are equal.
    We further observe that the signature function takes values in the set $\{ 1, \dots, k \}$, hence the composite $w \circ \Sigma \from \{ 1, \dots, n \} \to \{ 0, \dots, m \}$ is well-defined and invariant under $F$-reduction.
    Lastly, if an $F$-sequence is in reduced form then $\Sigma(j) = P(j)$ for all $j \in \{ 1, \dots, n \}$.

    Now, suppose $(w'; P_i'^\pm)_{i=1}^{k'}$ and $(w''; P_i''^\pm)_{i=1}^{k''}$ are two $F$-sequences in reduced form that are related to $(w; P_i^\pm)_{i=1}^k$ via $F$-reductions.
    The length of any $F$-sequence in reduced form is given by the maximum value of its signature function $\max_j \Sigma(j)$, hence $k' = k''$.
    Moreover, we have that $P'(j) = \Sigma(j) = P''(j)$, hence we have an equality of partition functions $P' = P''$.
    To show that the weight functions are equal, fix $i \in \{ 1, \dots, k' \}$ and let $j \in \{ 1, \dots, n \}$ be such that $P'(j) = P''(j) = i$ (such a $j$ exists since $P' = P''$ is surjective).
    We use the fact that $P' = P'' = \Sigma$ and that $w \circ \Sigma$ is invariant under $F$-reduction to calculate
    \[ w'(i) = w'(P'(j)) = w'(\Sigma(j)) = w''(\Sigma(j)) = w''(P''(j)) = w''(i), \]
    hence $w' = w''$.
    Concerning the equality of sign functions, let $\varphi$, $\varphi'$ and $\varphi''$ denote the sign functions $\{ 1, \dots, n \} \to \{ -, + \}$ of $(w; P_i^\pm)_{i=1}^k$, $(w'; P_i'^\pm)_{i=1}^{k'}$ and $(w''; P_i''^\pm)_{i=1}^{k''}$, respectively.
    If $\varphi'(j) = \varphi(j)$ and $\varphi''(j) = \varphi(j)$ then clearly $\varphi'(j) = \varphi''(j)$.
    If $\varphi(j) \neq \varphi'(j)$ then it must be that $w(\Sigma(j)) = m$, since the sign of an element $j$ can only change by an FR2-reduction, which means the element $j$ must have been merged into an equivalence class of weight $m$, and $w(\Sigma(j)) = w(P'(j)) = m$.
    Similarly, if $\varphi(j) \neq \varphi''(j)$ then $w(\Sigma(j)) = m$.
    In either case, if $w(\Sigma(j)) = m$ then $\varphi'(j) = \varphi''(j) = +$ since both $\varphi'$ and $\varphi''$ are the sign functions of an $F$-sequence in reduced form.
    With this, we may deduce that $\varphi' = \varphi''$, therefore $(w'; P_i'^\pm)_{i=1}^{k'} = (w''; P_i''^\pm)_{i=1}^{k''}$.

    % To find the expanded form of an $F$-sequence, by part (1), we may assume the $F$-sequence is in reduced form.
    
    \emph{Algorithm for finding an expanded form $F$-sequence}.
    \begin{enumerate}
        \item[] Input: An $F$-sequence $(w; P_i^\pm)_{i=1}^k$.
        \item[] Output: An $F$-sequence $(w'; P_i'^\pm)_{i=1}^{k'}$ in expanded form that is related to $(w; P_i^\pm)_{i=1}^k$ by $F$-reductions.
        \item Define $(w'; P_i'^\pm)_{i=1}^{k'}$ to be the reduced form $F$-sequence outputted the previous algorithm on $(w; P_i^\pm)_{i=1}^k$.
        \item If $w(k') = m$:
        \begin{enumerate}
            \item Enumerate the elements of $P'_{k'}$ as $p'_{k', 1}, \dots, p'_{k', \ell(k')}$.
            \item Use FR2 in reverse to split $P'_{k'}$ into classes
            \[ P'_{k'} := \{ p'_{k', 1} \}, \quad P'_{k' + 1} := \{ p'_{k', 2} \}, \quad\dots, \quad P'_{k' + \ell(k') - 1} := \{ p'_{k', \ell(k')} \} \]
            all with weight $m$.
            \item Set $k' := k' + \ell(k') - 1$.
        \end{enumerate}
        \item While there exists $i \in \{ 1, \dots, k' \}$ such that $P'_i$ has more than one element:
        \begin{enumerate}
            \item Choose an element $j \in P'_i$.
            \item Use FR1 in reverse to split $P'_{i}$ into classes
            \[ (P'_{i})_{\mathrm{new}} := (P'_{i})_{\mathrm{old}} - \{ j \} \quad\text{and}\quad (P'_{i+1})_{\mathrm{new}} := \{ j \}, \]
            where $(P'_{i+1})_{\mathrm{new}}$ has weight 0.
            Note this shifts the partition and weight function so that 
            \[ \begin{array}{c c c}
                (P'_{i+2})_{\mathrm{new}} = (P'_{i+1})_{\mathrm{old}} &\text{and}& w'_{\mathrm{new}}(i+2) = w'_{\mathrm{old}}(i+1) \\
                (P'_{i+3})_{\mathrm{new}} = (P'_{i+2})_{\mathrm{old}} &\text{and}& w'_{\mathrm{new}}(i+3) = w'_{\mathrm{old}}(i+2) \\
                \vdots & {} & \vdots \\
                (P'_{k' + 1})_{\mathrm{new}} = (P'_{k'})_{\mathrm{old}} &\text{and}& w'_{\mathrm{new}}(k'+1) = w'_{\mathrm{old}}(k')
            \end{array} \]
            \item Increment $k'$ by 1.
        \end{enumerate}
        \item Return $(w'; P_i'^\pm)_{i=1}^{k'}$
    \end{enumerate} 
    This algorithm terminates because each iteration of the while loop decreases the number of elements in the equivalence class $P'_i$.
    By construction, the output $F$-sequence has exactly one element in each equivalence class $P'_i$, and is related to $(w; P_i^\pm)_{i=1}^k$ via $F$-reductions.
    If some $P'_i$ has weight $m$ then the class $P'_i$ was created in Step 2 of the algorithm, hence $P'_{i+1}, \dots, P'_{k'}$ all have weight $m$.
    Thus, the output $F$-sequence is in expanded form as desired.
\end{proof}
Given an $F$-sequence $(w; P_i^\pm)_{i=1}^k$, the \emph{reduced form of $(w; P_i^\pm)_{i=1}^k$} is the unique $F$-sequence in reduced form that is related to $(w; P_i^\pm)_{i=1}^k$ via $F$-reductions.
An \emph{expanded form of $(w; P_i^\pm)_{i=1}^k$} is an $F$-sequence in expanded form that is related to $(w; P_i^\pm)_{i=1}^k$ via $F$-reductions.
\begin{example}
    Consider the case of $m = 8$ and $n = 3$.
    The reduced form of the $F$-sequence $\big( (7; 1^+, 3^-), (0; 2^-) \big)$ is $\big( (7; 1^+, 2^-, 3^-) \big)$.
    Each of the 6 $F$-sequences
    \[ \begin{array}{c c}
        \big( (7; 1^+), (0; 2^-), (0; 3^-) \big) & \big( (7; 1^+), (0; 3^-), (0; 2^-) \big) \\
        \big( (7; 2^-), (0; 3^-), (0; 1^+) \big) & \big( (7; 2^-), (0; 1^-), (0; 3^-) \big) \\
        \big( (7; 3^-), (0; 1^+), (0; 2^-) \big) & \big( (7; 3^-), (0; 2^-), (0; 1^+) \big)
    \end{array} \] 
    is an expanded form of $\big( (7; 1^+, 3^-), (0; 2^-) \big)$.
    
    If we instead consider the $F$-sequence $\big( (8; 1^+, 3^-), (0; 2^-) \big)$ then any choice of sign function $\big( (8; 1^\pm, 3^\pm), (0; 2^\pm) \big)$ is related via $F$-reductions.
    There are 8 possible sign functions on the set $\{ 1, 2, 3 \}$, hence there are $6 \times 8 = 48$ distinct expanded forms of $\big( (8; 1^+, 3^-), (0; 2^-) \big)$.

    The $F$-sequence $\big( (4; 3^+), (1; 1^+), (4; 2^-) \big)$ is in both reduced form and expanded form.
    Moreover, the only $F$-sequence related to $\big( (4; 3^+), (1; 1^+), (4; 2^-) \big)$ via $F$-reductions is itself.
\end{example}

\begin{definition}
    Let $m, n \geq 0$ be non-negative.
    \begin{enumerate}
        \item Define a graph $F(m, n)$ as follows:
        \begin{itemize}
            \item the vertex set of $F(m, n)$ is the set of $F$-sequences modulo $F$-reduction.
            \item two vertices $(w_0; (P_0)_i^\pm)_{i=1}^{k_0}$ and $(w_1; (P_1)_i^\pm)_{i=1}^{k_1}$ are connected by an edge if there are expanded forms $(w'_0; (P'_0)_i^\pm)_{i=1}^{n}$ and $(w'_1; (P'_1)_i^\pm)_{i=1}^{n}$ such that:
            \begin{itemize}
                \item the sign functions of $(w'_0; (P'_0)_i^\pm)_{i=1}^{n}$ and $(w'_1; (P'_1)_i^\pm)_{i=1}^{n}$ are equal;
                \item the partition functions $P'_0 = P'_1$ are equal; and
                \item there exists $i \in \{ 1, \dots, n \}$ such that $|w'_0(i) - w'_1(i)| \leq 1$ and $w'_0(j) = w'_1(j)$ for all $j \neq i$ in $\{ 1, \dots, n \}$.
            \end{itemize}
        \end{itemize}
        \item The \emph{boundary of $F(m, n)$} is the induced subgraph $\bd F(m, n) \subseteq F(m, n)$ consisting of those vertices $(w; P_i^\pm)_{i=1}^{k}$ whose reduced form satisfies $w(1) = 0$.
        Our convention is that if $n = 0$ then $\bd F(m, 0) = \varnothing$.
        \item For $j = 1, \dots, n$ and $\eps = 0, 1$, we define a graph map $\face{}{j, \eps} \from F(m, n-1) \to F(m, n)$ by the assignment:
        \[ \face{}{j, \eps}\left( (w; P_i^\pm)_{i=1}^k \right) := (0; j^\eps) \cdot (w; P_i^\pm)_{i=1}^k
        % \begin{cases}
        %     (0; j^-) \cdot (w; P_i^\pm)_{i=1}^k & \text{if } \eps = 0 \\
        %     (0; j^+) \cdot (w; P_i^\pm)_{i=1}^k & \text{if } \eps = 1.
        % \end{cases} 
        \]
        Here, the notation $(0; j^\eps) \cdot (w; P_i^\pm)_{i=1}^k$ denotes the $F$-sequence obtained by inserting an equivalence class $P_1 := \{ j \}$ with weight 0; the sign of $j$ is $-$ if $\eps = 0$ and $+$ if $\eps = 1$.
        Note this alters the sign function $\varphi$, the partition function $P$, and the weight function $w$ so that:
        \begin{align*}
            \varphi_{\mathrm{new}}(k) &:= \begin{cases}
                \varphi_{\mathrm{old}}(k) & \text{if } k < j \\
                - & \text{if } k = j \text{ and } \eps = 0 \\
                + & \text{if } k = j \text{ and } \eps = 1 \\
                \varphi_{\mathrm{old}}(k-1) & \text{if } k > j.
            \end{cases} \\
            P_{\mathrm{new}}(k) &:= \begin{cases}
                1 & \text{if } k = j \\
                P_{\mathrm{old}}(k) + 1 & \text{otherwise.}
            \end{cases} \\
            w_{\mathrm{new}}(i) &:= \begin{cases}
                0 & \text{if } i = 0 \\
                w_{\mathrm{old}}(i-1) & \text{otherwise.}
            \end{cases}
        \end{align*}
        In particular, the length of the resulting $F$-sequence is $k + 1$.
    \end{enumerate}
\end{definition}
\begin{proposition}
    For $m \geq 0$ and $n \geq 1$, the function $\face{}{j, \eps} \from F(m, n-1) \to F(m, n)$ is a well-defined graph map.
\end{proposition}
\begin{proof}
    Let $(w; P_{i}^\pm)_{i=1}^k$ be an $F$-sequence in $F(m, n-1)$.
    For any subset $\{ i_1, \dots, i_\ell \} \in \{ 1, \dots, k \}$, if we merge the classes $P_{i_1}, \dots, P_{i_\ell}$ in $(w; P_i^\pm)_{i=1}^k$ then apply $\face{}{j, \eps}$, this is the same as applying $\face{}{j, \eps}$ first then merging the classes $P_{i_1+1}, \dots, P_{i_\ell + 1}$.
    From this, we deduce that $\face{}{j, \eps}$ is invariant under $F$-reduction, and that it preserves edges.
\end{proof}
\begin{example} \label{ex:Gamma-ncube-012}
    We fix $m = 2$ and depict the graphs $F(2, 0)$, $F(2, 1)$ and $F(2, 2)$.
    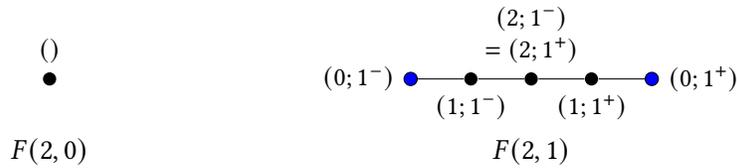
\begin{figure}[H]
        \centering
        \begin{tikzpicture}[scale=1.6]
            \node[vertex, label={$()$}] (F20) {};
            \path (F20) -- +(0, -0.6) node[font=\large] {$F(2, 0)$};
            \path (F20) -- +(3, 0) node[vertex, draw, fill=vertex2, label={west:$(0; 1^-)$}] (F21a) {};
            \path (F21a) -- +(0.5, 0) node[vertex, label={south:$(1; 1^-)$}] (F21b) {};
            \path (F21b) -- +(0.5, 0) node[vertex, label={[align=center]north:$(2; 1^-)$ \\ $= (2; 1^+)$}] (F21c) {};
            \path (F21c) -- +(0.5, 0) node[vertex, label={south:$(1; 1^+)$}] (F21d) {};
            \path (F21d) -- +(0.5, 0) node[vertex, draw, fill=vertex2, label={east:$(0; 1^+)$}] (F21e) {};
            \draw
                (F21a) -- (F21b)
                (F21b) -- (F21c) 
                (F21c) -- (F21d)
                (F21d) -- (F21e);
            \path (F21c) -- +(0, -0.6) node[font=\large] {$F(2, 1)$};
        \end{tikzpicture}
        \caption{The graphs $F(2, 0)$ and $F(2, 1)$.
        The blue vertices denote the boundary $\bd F(2, 1)$.} \label{fig:Gamma-ncube-01}
    \end{figure}
    \begin{figure}[H]
        \centering
        \begin{tikzpicture}
            % Top and right face, layer 0
            \node[vertex, draw, fill=vertex2] (TL0) {};
            \draw (TL0) -- +(1, 0) node[vertex, draw, fill=vertex2] (T0a) {};
            \draw (T0a) -- +(1, 0) node[vertex, draw, fill=vertex2] (T0b) {};
            \draw (T0b) -- +(1, 0) node[vertex, draw, fill=vertex2] (T0c) {};
            \draw (T0c) -- +(1, 0) node[vertex, draw, fill=vertex2, label={east:$(0; 1^+, 2^-)$}] (TR0) {};
            \draw (TR0) -- +(0, -1) node[vertex, draw, fill=vertex2, label={east:$\big( (0; 1^+), (1; 2^-) \big)$}] (R0a) {};
            \draw (R0a) -- +(0, -1) node[vertex, draw, fill=vertex2, label={east:$\big( (0; 1^+), (2; 2^+) \big)$}] (R0b) {};
            \draw (R0b) -- +(0, -1) node[vertex, draw, fill=vertex2 , label={east:$\big( (0; 1^-), (1;2^+) \big)$}] (R0c) {};
            % Left and bottom face, layer 0
            \draw (TL0) -- +(0, -1) node[vertex, draw, fill=vertex2] (L0a) {};
            \draw (L0a) -- +(0, -1) node[vertex, draw, fill=vertex2] (L0b) {};
            \draw (L0b) -- +(0, -1) node[vertex, draw, fill=vertex2] (L0c) {};
            \draw (L0c) -- +(0, -1) node[vertex, draw, fill=vertex2] (BL0) {};
            \draw (BL0) -- +(1, 0) node[vertex, draw, fill=vertex2] (B0a) {};
            \draw (B0a) -- +(1, 0) node[vertex, draw, fill=vertex2] (B0b) {};
            \draw (B0b) -- +(1, 0) node[vertex, draw, fill=vertex2] (B0c) {};
            \draw (B0c) -- +(1, 0) node[vertex, draw, fill=vertex2, label={east:$(0; 1^+, 2^+)$}] (BR0) {};
            % Layer 1, corners and centers
            \draw (TL0) -- +(0.6, -0.6) node[vertex] (TL1) {};
            \draw (TR0) -- +(-0.6, -0.6) node[vertex] (TR1) {};
            \draw (BL0) -- +(0.6, 0.6) node[vertex] (BL1) {};
            \draw (BR0) -- +(-0.6, 0.6) node[vertex] (BR1) {};
            \draw (T0b) -- +(0, -0.6) node[vertex] (T1b) {};
            \draw (L0b) -- +(0.6, 0) node[vertex] (L1b) {};
            \draw (R0b) -- +(-0.6, 0) node[vertex] (R1b) {};
            \draw (B0b) -- +(0, 0.6) node[vertex] (B1b) {};
            % Layer 1, midpoints
            \draw (TL1) -- node[vertex] (T1a) {} (T1b);
            \draw (T1b) -- node[vertex] (T1c) {} (TR1);
            \draw (TL1) -- node[vertex] (L1a) {} (L1b);
            \draw (L1b) -- node[vertex] (L1c) {} (BL1);
            \draw (TR1) -- node[vertex] (R1a) {} (R1b);
            \draw (R1b) -- node[vertex] (R1c) {} (BR1);
            \draw (BL1) -- node[vertex] (B1a) {} (B1b);
            \draw (B1b) -- node[vertex] (B1c) {} (BR1);
            % Center point
            \draw (L1b) -- node[vertex] (center) {} (R1b);

            % Edges
            \draw
                (BR0) -- (R0c) % Missing edge on the border
                (T0a) -- (T1a) % Missing edges between layers
                (T0c) -- (T1c)
                (L0a) -- (L1a)
                (L0c) -- (L1c)
                (R0a) -- (R1a)
                (R0c) -- (R1c)
                (B0a) -- (B1a)
                (B0c) -- (B1c)
                (TL1) -- (BR1) % Edges *through* center
                (T1b) -- (B1b)
                (TR1) -- (BL1)
                (T1a) -- (center)
                (T1c) -- (center)
                (L1a) -- (center)
                (L1c) -- (center)
                (R1a) -- (center)
                (R1c) -- (center)
                (B1a) -- (center)
                (B1c) -- (center);
            % Redraw center
            \draw (L1b) -- node[vertex, label={[fill=white, fill opacity=0.6, yshift=0.2ex]$(2; 1^+, 2^+)$}] {} (R1b);
            \draw (L1b) -- node[vertex, label={[yshift=0.2ex]$(2; 1^+, 2^+)$}] {} (R1b);

            \draw[dashed] (BL1) -- +(-1.8, -0.4) node[fill=white] {$(1; 1^-, 2^+)$};
            \draw[dashed] (B1a) -- +(0, -1.4) node[fill=white] {$\big( (1; 2^+), (1; 1^-) \big)$};
            \draw[dashed] (L1c) -- +(-2.1, 0.4) node[fill=white] {$\big( (1; 1^-), (1; 2^+) \big)$};
        \end{tikzpicture}
        \caption{The graph $F(2, 2)$, with a subset of vertices labelled by their reduced form. 
        The blue vertices denote the boundary $\bd F(2, 2)$.} \label{fig:Gamma-ncube-2}
    \end{figure}
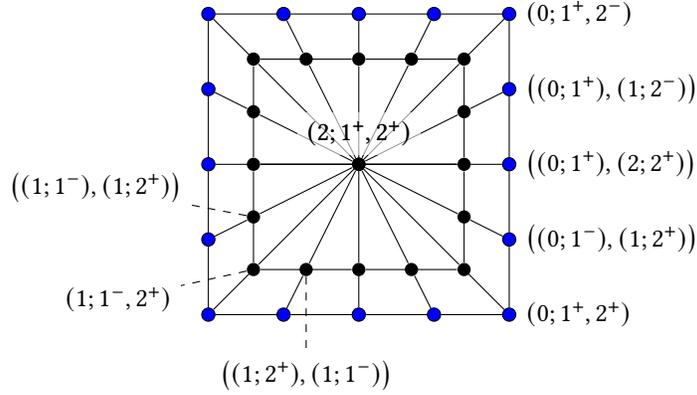
\end{example}
\begin{proposition}
    Let $m \geq 0$ be non-negative.
    The assignment $n \mapsto F(m, n)$ and $(i, \eps) \mapsto \face{}{i, \eps}$ defines a functor
    \[ F(m, \uvar) \from \boxcati \to \Graph. \]
\end{proposition}
\begin{proof}
    It remains only to verify the identity $\partial_{j', \varepsilon'} \partial_{j, \varepsilon} = \partial_{j, \varepsilon} \partial_{j'-1, \varepsilon'}$ for $j \leq j'$.
    % That is, for every $F$-sequence $(w; P_i^\pm)_{i=1}^k$, we have $(0; j^{\eps'}) \cdot (0; i^{\eps}) \cdot (w; P_i^\pm)_{i=1}^k = (0; i+1^{\pm}) \cdot (0; )$
    By the FR1-reduction rule, we have an equality
    $(0; j'^{\eps'}) \cdot (0; j^{\eps}) \cdot (w; P_i^\pm)_{i=1}^k = (0; j^\eps) \cdot (0; (j'-1)^{\eps'}) \cdot (w; P_i^\pm)_{i=1}^k$
    for every $F$-sequence $(w; P_i^\pm)_{i=1}^k$.
\end{proof}
\begin{theorem} \label{F-explicit-isomorphism}
    For $m ,n \geq 0$, there exists an isomorphism
    \[ \Phi^n \from \Gn[m]{\cube{n}} \cong F(m, n) \]
    which is natural in the variable $n \in \boxcati$.
\end{theorem}
Before proving \cref{F-explicit-isomorphism}, we prove the following lemmas.
\begin{lemma} \label{Fmn-pushout}
    For $m, n \geq 0$, there exists a graph map $\bd F(m, n) \gtimes I_m \to F(m, n)$ such that the diagram
    \[ \begin{tikzcd}
        {} & \bd F(m, n) \ar[r] \ar[d, "i_m"'] \ar[rd, phantom, "\potick" very near end] & I_0 \ar[d, "{(m; 1^+, \dots, n^+)}"] \\
        \bd F(m, n) \ar[r, "i_0"] \ar[rr, bend right=18, hook] & \bd F(m, n) \gtimes I_m \ar[r, "f"] & F(m, n)
    \end{tikzcd} \]
    is commutative and the right square is a pushout.
\end{lemma}
\begin{proof}
    Define the graph map $f \from \bd F(m, n) \gtimes I_m \to F(m, n)$ as follows:
    given an $F$-sequence $(w; P_i^\pm)_{i=1}^k$ and a vertex $t \in I_m$, set $f \big( (w; P_i^\pm)_{i=1}^k, t \big)$ to be the $F$-sequence obtained from the reduced form of $(w; P_i^\pm)_{i=1}^k$ by setting $w(1) = t$.
    By the uniqueness of reduced forms (\cref{unique-reduced-forms}), this map is well-defined.
    Observe that the value $w(1)$ is invariant under $F$-reduction.
    From this, we deduce that the map $f$ preserves edges, and moreover reflects edges when restricted to the subgraph $\bd F(m, n) \gtimes I_{m-1}$.
    In particular, $f$ is a graph map.

    The FR2-reduction rule implies that if $t = m$ then $f(\uvar, m)$ is constant at the vertex $(m; 1^+, \dots, n^+)$.
    If $t = 0$ then $f(\uvar, 0)$ coincides with the subgraph inclusion map $\bd F(m, n) \ito F(m, n)$.
    Thus, the above diagram commutes.
    To see that the right square is a pushout, we use the explicit description of pushouts combined with the fact that $f$ reflects edges on the subgraph $\bd F(m, n) \gtimes I_{m-1}$.
\end{proof}

Recall that, for $n \geq 0$, the \emph{latching category} $\bd(\boxcati \slice [1]^n)$ of $\cube{n} \in \cSeti$ is the comma category whose objects are non-identity injections $[1]^{p} \to [1]^n$ in $\boxcati$ and whose morphisms are injections $[1]^{p} \to [1]^{q}$ in $\boxcati$ making the induced triangle commute.
The \emph{latching object} of $\cube{n} \in \cSeti$ with respect to the functor $F(m, \uvar) \from \boxcati \to \Graph$ is the colimit
\[ L_{\cube{n}} F(m, \uvar) := \colim \left(  \bd(\boxcati \slice [1]^n) \xrightarrow{\Pi} \boxcati \xrightarrow{F(m, \uvar)} \Graph \right). \]
The inclusions $\face{}{i, \eps} \from F(m, n-1) \to F(m, n)$ induce a map $L_{\cube{n}}F(m, \uvar) \to F(m, n)$; one verifies that this map factors through the boundary, yielding a map
\[ L_{\cube{n}} F(m, \uvar) \to \bd F(m, n). \]
Explicitly, the vertices in the latching object $L_{\cube{n}} F(m, \uvar)$ are pairs $[f, (w; P_i^\pm)_{i=1}^k]$ where $f$ is a non-identity injection $[1]^p \to [1]^n$ and $(w; P_i^\pm)_{i=1}^k$ is an $F$-sequence in $F(m, p)$.
These pairs are subject to the identification $[f \circ \face{}{i,\eps}, (w; P_i^\pm)_{i=1}^k] = [f; (0; i^\eps) \cdot (w; P_i^\pm)_{i=1}^k]$ for any face map $\face{}{i,\eps} \in \boxcati$.
Edges are given by pairs of the form $[f, (w; P^\pm_i)_{i=1}^{k}] \sim [f, (w'; P'^\pm_{i})]_{i=1}^{k'}$ whose 2nd coordinates are connected by an edge in $F(m, p)$.

By the existence of normal forms for morphisms in $\boxcati$ (\cite[Lem.~4.1, Thm.~5.1]{grandis-mauri}), every injection $[1]^p \to [1]^n$ in $\boxcati$ can be written uniquely as a composite of face maps $\face{}{i_1, \eps_1} \dots \face{}{i_{n-p}, \eps_{n-p}}$ where $i_1 > \dots > i_{n-p}$.
Thus, the map $L_{\cube{n}} F(m, \uvar) \to \bd F(m, n)$ admits an explicit formula given by
\[ [\face{}{i_1, \eps_1} \dots \face{}{i_{n-p}, \eps_{n-p}}, (w; P_i^\pm)_{i=1}^k] \mapsto (0; i_1^{\eps_1}) \cdot \ldots \cdot (0; i_{n-p}^{\eps_{n-p}}) \cdot (w; P_i^\pm)_{i=1}^k. \]
\begin{lemma} \label{Fmn-latching-object-iso-boundary}
    The map $L_{\cube{n}} F(m, \uvar) \to \bd F(m, n)$ is an isomorphism.
\end{lemma}
\begin{proof}
    Before defining the inverse map $\bd F(m, n) \to L_{\cube{n}} F(m, \uvar)$, we make the following observation.
    A vertex $(w; P_i^\pm)_{i=1}^k \in \bd F(m, n)$ in reduced form satisfies $w(1) = 0$ by definition of the boundary.
    % and $w(2) \neq 0$, by definition of the reduced form.
    Enumerating the elements of the set complement $\{ 1, \dots, n \} - P_1$ yields a bijection 
    \[ \sigma \from \{ 1, \dots, p \} \xrightarrow{\cong} \big( \{ 1, \dots, n \} - P_1 \big) \]
    for some $p \geq 0$ (the case of $p=0$ occurs if and only if $P_1 = \{ 1, \dots, n \}$, in which case the domain and codomain of $\sigma$ are empty).
    Writing the reduced form of $(w; P_i^\pm)_{i=1}^k \in \bd F(m, n)$ as $(\varphi, k, P, w)$, we define an $F$-sequence $(\varphi', k-1, P', w')$ in $F(m, p)$ by
    \[ \begin{array}{c c c}
        \varphi'(j) := \varphi(\sigma(j)) &
        P'(j) := P(\sigma(j)) - 1 &
        w'(i) := w(i+1).
    \end{array} \]
    Going forward, we denote this $F$-sequence by $(w'; P'^\pm_i)_{i=1}^{k-1}$.
    We now sort the elements in $P_1$ as
    \[ P_1 = \{ x_{1} > \dots > x_{n - p} \}. \]
    and define $\eps_1, \dots, \eps_{n-p} \in \{0, 1\}$ by
    \[ \eps_{j} := \begin{cases}
        0 & \text{if } \varphi(x_j) = - \\
        1 & \text{if } \varphi(x_j) = +.
    \end{cases} \]
    We now define the inverse map $\bd F(m, n) \to L_{\cube{n}} F(m, \uvar)$ by sending an $F$-sequence $(w; P_i^\pm)_{i=1}^k$ to the pair given by $[\face{}{x_1, \eps_1} \dots \face{}{x_{n-p}, \eps_{n-p}}, (w'; P'^\pm_i)_{i=1}^{k-1}]$.
    The FR1-reduction rule can be used to deduce that this map preserves edges.
    % To see that this map preserves edges, we observe that if $(w_1; (P_1)_i^\pm)_{i=1}^{k_1}$ and $(w_2; (P_2)_i^\pm)_{i=1}^{k_2}$ are connected by an edge then there exist expanded forms of $(w_1; (P_1)_i^\pm)_{i=1}^{k_1}$ and $(w_2; (P_2)_i^\pm)_{i=1}^{k_2}$ such that 

    Composing the inverse map followed by the forward map recovers $\id[\bd F(m, n)]$ by construction.
    To show that the reverse composite equals $\id[L_{\cube{n}}F(m, \uvar)]$, we observe that if $[f, (w; P_i^\pm)_{i=1}^k]$ is a vertex in $L_{\cube{n}} F(m \uvar)$ then, by the identification $[f \circ \face{}{i,\eps}, (w; P_i^\pm)_{i=1}^k] = [f; (0; i^\eps) \cdot (w; P_i^\pm)_{i=1}^k]$, we may uniquely write $[f, (w; P_i^\pm)_{i=1}^k]$ as a pair $[\face{}{i_1, \eps_1} \dots \face{}{i_{n-p}, \eps_{n-p}}, (w'; P_i^\pm)_{i=1}^{k'}]$ so that $(w'; P_i^\pm)_{i=1}^{k'}$ is in reduced form and satisfies either $k' = 0$ or $w(1) \neq 0$.
    From this, the desired equality may be deduced.
\end{proof}

\begin{proof}[Proof of \cref{F-explicit-isomorphism}]
    We fix $m \geq 0$ and construct a family of isomorphisms
    \[ \Phi^n \from \Gn{\cube{n}} \to F(m, n) \]
    by induction on $n$.
    In the base case $n = 0$, the map $\Phi^0$ is an isomorphism between graphs with a single vertex.

    Fix $n$ and suppose we have constructed $\Phi^k$ for all $k \leq n$.
    This family of isomorphisms induces an isomorphism $\Gn{\bd \cube{n}} \to L_{\cube{n}} F(m, \uvar)$
    between latching objects.
    By \cref{Fmn-latching-object-iso-boundary}, we identify this as an isomorphism
    \[ \bd \Phi^{n+1} \from \Gn{\bd \cube{n}} \xrightarrow{\cong} \bd F(m, n). \]
    This isomorphism participates in a diagram
    \[ \begin{tikzcd}
        \Gn{\bd \cube{n}} \ar[rr] \ar[dd] \ar[rd, "\bd \Phi^{n+1}", "\cong"'] & {} & I_0 \ar[rd, equal] \ar[dd] & {} \\
        {} & \bd F(m, n) \ar[rr, crossing over] & {} & I_0 \ar[dd] \\
        \Gn{\bd \cube{n}} \gtimes I_m \ar[rr] \ar[rd, "\bd \Phi^{n+1} \gtimes I_m"', "\cong"] & {} & \Gn{\cube{n}} \ar[rd, dotted, "\cong"]  & {} \\
        {} & \bd F(m, n) \gtimes I_m \ar[from=uu, crossing over] \ar[rr] & {} & F(m, n)
    \end{tikzcd} \]
    The back face is a pushout by \cref{F-is-mapping-cone}.
    The front face is a pushout by \cref{Fmn-pushout}.
    Therefore, there is an induced isomorphism $\Gn{\cube{n}} \to F(m, n)$ as desired.
\end{proof}
\begin{corollary} \label{Gn-explicit-description}
    Let $X$ be a semicubical set.
    For $m \geq 0$, the graph $\Gn[m]{X}$ admits the following explicit description:
    \begin{enumerate}
        \item The vertices of $\Gn{X}$ are pairs $\big( u, (w; P_i^\pm)_{i=1}^k \big)$ where $u$ is a non-degenerate $n$-cube of $X$ and $(w; P_i^\pm)_{i=1}^k$ is an $F$-sequence in $F(m, n)$.
        These pairs are subject to the identification generated by $(u\face{}{i, \eps}, (w; P_i^\pm)_{i=1}^k) = (u, (0; i^\eps) \cdot (w; P_i^\pm)_{i=1}^k)$ for any face map $\face{}{i, \eps} \in \boxcati$.
        \item Edges are given by pairs of the form $\big( u, (w_0; (P_0)_i^\pm)_{i=1}^{k_0} \big) \sim \big( u, (w_1; (P_1)_i^\pm)_{i=1}^{k_1} \big)$ where $(w_0; (P_0)_i^\pm)_{i=1}^{k_0}$ and $(w_1; (P_1)_i^\pm)_{i=1}^{k_1}$ admit expanded forms $(w'_0; (P'_0)_i^\pm)_{i=1}^{n}$ and $(w'_1; (P'_1)_i^\pm)_{i=1}^{n}$ such that: 
        \begin{itemize}
            \item the sign functions of $(w'_0; (P'_0)_i^\pm)_{i=1}^{n}$ and $(w'_1; (P'_1)_i^\pm)_{i=1}^{n}$ are equal;
            \item the partition functions $P'_0 = P'_1$ are equal; and
            \item there exists $i \in \{ 1, \dots, n \}$ such that $|w'_0(i) - w'_1(i)| \leq 1$ and $w'_0(j) = w'_1(j)$ for all $j \neq i$ in $\{ 1, \dots, n \}$.
        \end{itemize}
    \end{enumerate}
\end{corollary}
\begin{proof}
    This is a standard argument using the formula $\colim\limits_{(\cube{n},\, u) \in \int\! X} \cube{n} \cong X$ and \cref{F-explicit-isomorphism}.
\end{proof}

 \section{Main theorem} \label{sec:main-thm}

At this point, we have established all the requisite notions and the goal of this section, as the name might suggest, is to prove the main theorem of the paper.
This is \cref{nerve-equiv-main-thm}, asserting that the graph nerve functor $\gnerve[\infty] \colon \Graph \to \cSet$ induces an equivalence of homotopy categories obtained by inverting $n$-equivalences.
In preparation for the proof, we need a couple of lemmas that we dispose of first.

Recall that the graph realization functors are related by natural surjections
\[ (\ell_!)_X, (r_!)_X \from \reali[m+1]{X} \to \reali[m]{X}. \]
Instantiating at $F_! X$ for some semicubical set $X$, we obtain natural transformations $\ell_!, r_! \from \Gn[m+1]{X} \to \Gn[m]{X}$.
\begin{proposition} \label{unit-ell-unit-ell-square}
    For any cubical set $X$ and $m \in \NN$, the squares
    \[ \begin{tikzcd}
        X \ar[r, "\eta^{m+1}_{X}"] \ar[d, "\eta^m_X"'] & \gnerve[m+1]{\reali[m+1]{X}} \ar[d, "{\gnerve[m+1]{(\ell_!)}}"] \\
        \gnerve[m]{\reali[m]{X}} \ar[r, "{\ell^*}"] & \gnerve[m+1]{\reali[m]{X}}
    \end{tikzcd} \qquad \begin{tikzcd}
        X \ar[r, "\eta^{m+1}_{X}"] \ar[d, "\eta^{m}_X"'] & \gnerve[m+1]{\reali[m+1]{X}} \ar[d, "{\gnerve[m+1]{(r_!)}}"] \\
        \gnerve[m]{\reali[m]{X}} \ar[r, "{r^*}"] & \gnerve[m+1]{\reali[m]{X}}
    \end{tikzcd} \]
    commute.
\end{proposition}
\begin{proof}
    Concerning the left square, given a $k$-cube $u \from \cube{k} \to X$, the top right composite sends $u$ to the morphism $\ell_! \circ \reali[m+1]{u}$, whereas the bottom left composite sends $u$ to the morphism $\reali[m]{u} \circ \ell^{\gtimes k}$.
    These two are equal by definition of $\ell_!$.
    The proof of the right square is analogous.
\end{proof}
\begin{corollary} \label{ell-r-n-equiv}
    Let $X$ be a semicubical set.
    For $n \geq 0$ and $m \geq n+2$, the maps
    \[ \ell^m_X, r^m_X \from \Gn[m+1]{X} \to \Gn[m]{X} \]
    are $n$-equivalences.
    Additionally, if $X$ is nonsingular then the result holds when $m = n+1$.
\end{corollary}
\begin{proof}
    For $\ell^m_X$ or $r^m_X$, we instantiate the corresponding square of \cref{unit-ell-unit-ell-square} to $F_! X$.
    The bottom morphism is a weak equivalence by \cref{nerve-main-thm}.
    The left and top maps are $n$-equivalences by \cref{Gn-weq}.
    Therefore, the right map is an $n$-equivalence as desired.
\end{proof}

The following is a generalization of \cref{Gn-bd-inclusion-hopo}.
\begin{lemma} \label{po-cubical-mono-is-hopo}
    Let $n \geq 0$ and $m \geq n+2$.
    Suppose $f \from X \ito Y$ is a monomorphism of semicubical sets and $g \from \Gn[m]{X} \to G$ is a graph map.
    Then, the pushout of $g$ along $\Gn[m]{f}$
    \[ \begin{tikzcd}
        \Gn[m]{X} \ar[r, "{\Gn[m]{f}}"] \ar[d, "g"'] \ar[rd, phantom, "\potick" very near end] & \Gn[m]{Y} \ar[d] \\
        G \ar[r] & \bullet
    \end{tikzcd} \]
    becomes a homotopy pushout of $n$-types after applying $\gnerve[m]$.
    % Additionally, if $g$ is a monomorphism then the result holds when $m = n+1$. -- this is false unless X -> Y is a "cubical complex extension"
\end{lemma}
\begin{proof}
    % We proceed using a skeletal induction argument similar to \cref{Gn-weq}, invoking \cref{Gn-bd-inclusion-hopo}.
    The monomorphism $f$ can be decomposed into a sequence of monomorphisms
    \[ X = (X \cup \Sk^{-1} Y) \ito (X \cup \Sk^0 Y) \ito (X \cup \Sk^1 Y) \ito \dots \ito (X \cup \Sk^k Y) \ito \dots \ito Y \]
    where each $X \cup \Sk^k Y$ is the pushout of $X$ and $\Sk^k Y$ along $\Sk^k X$, and $\colim\limits_{k \in \NN} \, (X \cup \Sk^k Y) = Y$.
    For $k \geq 0$, the square
    \[ \begin{tikzcd}
        \displaystyle \coprod\limits_{\substack{y \in (Y_k)_{\mathrm{nd}} \\ y \,\not\in\, \im f}} \bd \cube{k} \ar[r, "{[\bd y]}"] \ar[d, hook] \ar[rd, phantom, "\potick" very near end] & X \cup \Sk^{k-1} Y \ar[d, hook] \\
        \displaystyle \coprod\limits_{\substack{y \in (Y_k)_{\mathrm{nd}} \\ y \,\not\in\, \im f}} \cube{k} \ar[r, "{[y]}"] & X \cup \Sk^{k} Y
    \end{tikzcd} \]
    is a pushout.
    The functor $\Gn[m] \from \cSeti \to \Graph$ preserves both monomorphisms and colimits, hence we obtain a decomposition of $\Gn[m]{f}$:
    \[ \Gn[m] X \ito \dots \ito \Gn[m](X \cup \Sk^{k-1} Y) \ito \Gn[m](X \cup \Sk^k Y) \ito \dots \ito \Gn[m] Y  \]
    where each inclusion $\Gn[m](X \cup \Sk^{k-1} Y) \ito \Gn[m](X \cup \Sk^k Y)$ arises as a pushout
    \[ \begin{tikzcd}
        \displaystyle \coprod\limits_{\substack{y \in (Y_k)_{\mathrm{nd}} \\ y \,\not\in\, \im f}} \Gn[m] \bd \cube{k} \ar[r, "{[\bd y]}"] \ar[d, hook] \ar[rd, phantom, "\potick" very near end] & \Gn[m](X \cup \Sk^{k-1} Y) \ar[d, hook] \\
        \displaystyle\coprod\limits_{\substack{y \in (Y_k)_{\mathrm{nd}} \\ y \,\not\in\, \im f}} \Gn[m] \cube{k} \ar[r, "{[y]}"] & \Gn[m](X \cup \Sk^{k} Y)
    \end{tikzcd} \]
    and $\colim\limits_{k \in \NN} \Gn[m](X \cup \Sk^k Y) = \Gn[m] Y$.
    % By pushing out along the map $g \from \Gn[m]{X} \to G$, we factor 
    Via the above sequence of inclusions, we may factor the pushout of interest
    \[ \begin{tikzcd}
        \Gn[m]{X} \ar[r, "{\Gn[m]{f}}", hook] \ar[d, "g"'] \ar[rd, phantom, "\potick" very near end] & \Gn[m]{Y} \ar[d] \\
        G \ar[r] & P
    \end{tikzcd} \] 
    into a sequence of pushouts:
    \[ \begin{tikzcd}
        \Gn[m]{(X \cup \Sk^{-1} Y)} \ar[r, hook] \ar[d, "g"'] \ar[rd, phantom, "\potick" very near end] & \Gn[m]{(X \cup \Sk^{0} Y)} \ar[r, hook] \ar[d, dotted] \ar[rd, phantom, "\potick" very near end] & \Gn[m]{(X \cup \Sk^{1} Y)} \ar[r, hook] \ar[d, dotted] \ar[rd, phantom, "\potick" very near end] & \dots \ar[r, hook] \ar[d, dotted] \ar[rd, phantom, "\potick" very near end] & \Gn[m]{Y} \ar[d] \\
        G \ar[r, hook, dotted] & P_0 \ar[r, hook, dotted] & P_1 \ar[r, hook, dotted] & \dots \ar[r, hook, dotted] & P
    \end{tikzcd} \]
    After applying $\gnerve[m]$, the top and bottom sequences are both sequential colimit diagrams of monomorphisms (since $\gnerve[m]$ preserves both filtered colimits and monomorphisms), hence homotopy colimit diagrams in the Cisinski model structure for $n$-equivalences.
    Thus, it suffices to show that the pushout along each inclusion $\Gn[m](X \cup \Sk^{k-1} Y) \ito \Gn[m](X \cup \Sk^{k} Y)$ becomes a homotopy pushout of $n$-types after applying $\gnerve[m]$.

    By \cref{Gn-cube-n-is-cyl}, the map $\Gn[m]\bd \cube{k} \ito \Gn[m] \cube{k}$ is identified with the inclusion into the double mapping cylinder $\ell_0(\id, \bang) \from \Gn[m] \bd \cube{k} \ito \Cyl{m}(\id, \bang)$.
    Observe that, given any collection of maps $\{ f_i \from G_i \to H_i \}$ and $\{ g_i \from G_i \to K_i \}$, we have an isomorphism $\coprod_i \Cyl{m}(f_i, g_i) \cong \Cyl{m}(\coprod_i f_i, \coprod_i g)$ for which the cylinder inclusions fit into a commutative diagram
    \[ \begin{tikzcd}[sep = large]
        \coprod_i H_i \ar[r, "{\coprod \ell_0(f_i, g_i)}"] \ar[rd, "{\ell_0(\coprod f_i, \coprod g_i)}"'] & \coprod_i \Cyl{m}(f_i, g_i) \ar[d, "\cong" description] & \coprod_i K_i \ar[l, "{\coprod r_0(f_i, g_i)}"'] \ar[ld, "{r_0(\coprod f_i, \coprod g_i)}"] \\
        {} & \Cyl{m}(\coprod f_i, \coprod g_i) & {}
    \end{tikzcd} \]
    With this, we may rewrite the pushout defining $\Gn[m] (X \cup \Sk^n Y)$ up to isomorphism as
    \[ \begin{tikzcd}
        \displaystyle \coprod\limits_{\substack{y \in (Y_k)_{\mathrm{nd}} \\ y \,\not\in\, \im f}} \Gn[m] \bd \cube{k} \ar[r, "{[\bd y]}"] \ar[d, hook, "{\ell_0(\coprod \id, \coprod \bang)}"'] \ar[rd, phantom, "\potick" very near end] & \Gn[m](X \cup \Sk^{k-1} Y) \ar[d, hook, "{\ell_0([\bd y], \coprod \bang)}"] \\
        \Cyl{m}(\coprod \id, \coprod \bang) \ar[r, "{[y]}"] & \Cyl{m}([\bd y], \coprod \bang)
    \end{tikzcd} \]
    where the isomorphism $X \cup \Sk^n Y \cong \Cyl{m}([\bd y], \coprod \bang)$ follows from \cref{cyl-as-po}.
    We paste the pushout square defining $P_{k}$ to the right to obtain the diagram:
    \[ \begin{tikzcd}
        \displaystyle \coprod\limits_{\substack{y \in (Y_k)_{\mathrm{nd}} \\ y \,\not\in\, \im f}} \Gn[m] \bd \cube{k} \ar[r, "{[\bd y]}"] \ar[d, hook, "{\ell_0(\coprod \id, \coprod \bang)}"'] \ar[rd, phantom, "\potick" very near end] & \Gn[m](X \cup \Sk^{k-1} Y) \ar[d, hook, "{\ell_0([\bd y], \coprod \bang)}" description] \ar[r, dotted] \ar[rd, phantom, "\potick" very near end] & P_{k-1} \ar[d, hook, dotted] \\
        \Cyl{m}(\coprod \id, \coprod \bang) \ar[r, "{[y]}"] & \Cyl{m}([\bd y], \coprod \bang) \ar[r, dotted] & P_k
    \end{tikzcd} \]
    We invoke \cref{cyl-factorization-hopo} to the left and composite squares (with $g = \coprod \bang$, $p = n$ and $q = m-n$) to deduce that these are homotopy pushouts of $n$-types.
    Therefore, the right square is a homotopy pushout of $n$ types, as desired.
\end{proof}
\begin{corollary} \label{po-weq-along-cubical-mono-is-weq}
    Let $n \geq 0$ and $m \geq n+2$.
    Suppose $f \from X \ito Y$ is a monomorphism of cubical sets and $w \from \Gn[m]{X} \to G$ is an $n$-equivalence.
    Then, the pushout of $w$ along $\Gn[m]{f}$
    \[ \begin{tikzcd}
        \Gn[m]{X} \ar[r, "{\Gn[m]{f}}"] \ar[d, "\sim", "w"'] & \Gn[m]{Y} \ar[d, dotted, "\sim"] \\
        G \ar[r, dotted] & \bullet
    \end{tikzcd} \]
    is an $n$-equivalence.
    % Additionally, if $w$ is a monomorphism then the result holds when $m = n+1$. -- this is false unless f is a "cubical complex extension"
\end{corollary}
\begin{proof}
    Follows from \cref{po-cubical-mono-is-hopo} using left properness.
\end{proof}

\begin{theorem} \label{nerve-equiv-main-thm}
    The nerve functor induces an equivalence 
    \[ \Ho (\Graph) \xrightarrow[\simeq]{\Ho \gnerve[\infty]} \Ho(\cSet_n) \]
    between localizations at $n$-equivalences.
\end{theorem}
\begin{proof}
    We prove that $\gnerve[\infty] \from \Graph \to \cSet_n$ satisfies the right approximation properties, which suffices by \cref{AP-properties-suffices}.
    Property (AP1) follows from \cref{nerve-reflects-equivs}.

    To show property (AP2), fix a map $f \from X \to \gnerve[\infty] G$ where $X \in \cSet$ is $n$-fibrant in the transferred model structure for $n$-equivalences.
    We structure the proof into the following steps:
    \begin{enumerate}
        \item[] Step 1: Construct a semicubical set $X' \in \cSeti$ and an $n$-equivalence $w \from F_! X' \weto X$.
        \item[] Step 2: Construct a filtration $(X'_m \subseteq X')_{m = n+2}^\infty$ such that $X' = \colim_m X'_m$ and the map $fw \from F_! X' \to \gnerve[\infty] G$ restricts to a natural transformation between filtrations $(fw)_m \from F_! X'_m \to \gnerve[m] G$.
        \item[] Step 3: Construct a filtration of graphs $(H_m \subseteq H_\infty)_{m=n+2}^\infty$ such that each object in the filtration admits a graph map $H_m \to G$, and the cubical map $(fw)_m \from F_! X'_m \to \gnerve[m] G$ factors through the nerve of this graph map as
        \[ F_! X'_m \to \gnerve[m] H_m \to G \]
        % \item[] Step 3: Construct a filtration of graphs $(H_m \subseteq H_\infty)_{m=n+2}^\infty$.
        % This will induce 
        \item[] Step 4: The diagonal of the two-variable filtration 
        \[ \big( \! \gnerve[m_1](H_{m_2}) \subseteq \gnerve[\infty] H_\infty \big)_{m_1, m_2 = n+2}^\infty \] 
        of cubical sets will admit a natural transformation $\gnerve[m] H_m \to \gnerve[m] G$.
        We conclude by constructing a natural $n$-equivalence $F_! X'_m \weto \gnerve[m] H_m$ and showing that the square
        \[ \begin{tikzcd}
            X \ar[r, "f"] & \gnerve[\infty] G \\
            F_! X' \ar[u, "w", "\sim"'] \ar[r, "\sim"] & \gnerve[\infty] H_\infty \ar[u]
        \end{tikzcd} \]
        commutes.
    \end{enumerate}
    We now proceed with the proof.

    \textit{Step 1.} We invoke \cref{fondind-presheaf-cat-equiv,nonsingular-approximation-cospan} to construct a cospan of weak equivalences
    \[ X \weto \tilde{X} \weot i_! \tilde{X}' \]
    where $\tilde{X}'$ is a semicubical set and $X \weto \tilde{X}$ is an acyclic cofibration in the Grothendieck model structure.
    The functor $\Sk^{n+1} \from \cSet \to \cSet$ is a left Quillen functor from the Grothendieck model structure to the $n$-transferred model structure, thus yielding a cospan
    \[ \Sk^{n+1} X \weto \Sk^{n+1} \tilde{X} \weot \Sk^{n+1} i_! \tilde{X}' \]
    where all maps are $n$-equivalences and $\Sk^{n+1} X \to \Sk^{n+1} \tilde{X}$ is an acyclic cofibration in the $n$-transferred model structure.
    Since $X$ is $n$-fibrant, the triangle
    \[ \begin{tikzcd}
        \Sk^{n+1} X \ar[r, hook, "\sim"] \ar[d, "\sim"'] & X \\
        \Sk^{n+1} \tilde{X} \ar[ur, dotted, "{(\sim)}"']
    \end{tikzcd} \]
    admits a lift, which is moreover an $n$-equivalence by 2-out-of-3.

    We now invoke \cref{F-shriek-approx-cospan} to construct a cospan of weak equivalences
    \[ i_! \tilde{X}' \weto L \tilde{X}' \weot F_! \tilde{X}' \]
    where $i_! \tilde{X}' \weto L \tilde{X}'$ is an acyclic cofibration in the Grothendieck model structure.
    Again, applying $\Sk^{n+1}$ yields a cospan of $n$-equivalences
    \[ \Sk^{n+1} i_! \tilde{X}' \weto \Sk^{n+1} L \tilde{X}' \weot \Sk^{n+1} F_! \tilde{X}' \]
    where $\Sk^{n+1} i_! \tilde{X}' \weto \Sk^{n+1} L \tilde{X}$ is an acyclic cofibration in the $n$-transferred model structure.
    Again, $X$ is $n$-fibrant, hence the triangle
    \[ \begin{tikzcd}
        \Sk^{n+1} i_! \tilde{X}' \ar[r, "\sim"] \ar[d, "\sim"] & \Sk^{n+1} \tilde{X} \ar[r, dotted, "\sim"] & X \\
        \Sk^{n+1} L \tilde{X}' \ar[urr, dotted, "(\sim)"']
    \end{tikzcd} \]
    admits a lift, which is an $n$-equivalence by 2-out-of-3.
    % Observe that we have natural isomorphisms $\Sk^{n+1} i_! \cong i_! \Sk^{n+1}$ and $\Sk^{n+1} F_! \cong F_! \Sk^{n+1}$.
    Observe that we have a natural isomorphism $\Sk^{n+1} F_! \cong F_! \Sk^{n+1}$.
    We define $X'$ to be the semicubical set $\Sk^{n+1} \tilde{X}'$ and let $w \from F_! X' \weto X$ denote the composite $n$-equivalence
    \[ \begin{tikzcd}
        F_! X' \ar[r, "\cong"] & \Sk^{n+1} F_! \tilde{X}' \ar[r, "\sim"] & \Sk^{n+1} L \tilde{X}' \ar[r, dotted, "\sim"] & X \\
    \end{tikzcd}  \]
    % We introduce the notation $X' := \Sk^n \tilde{X}'$ and 
    % We use this lift to define an $n$-equivalence $w \from i_! (\Sk^n X') \to X$ as the composite
    % \[ i_! (\Sk^n X') \weto \Sk^n \tilde{X} \weto X. \]

    \textit{Step 2}. Pulling back along $fw$ forms a functor
    \[ (fw)^* \from \cSet \slice \gnerve[\infty] G \to \cSet \slice F_! X' \]
    which preserves colimits since $\cSet$ is a presheaf category.
    Recall the functor $F_! \from \cSeti \to \cSet$ admits a right adjoint, which we denote by $F^* \from \cSet \to \cSeti$.
    This adjunction induces an adjunction on slice categories
    \[ \begin{tikzcd}
        \cSeti \slice X' \ar[r, bend left, "F_!"{name=U}] & \ar[l, bend left, "\eta^* F^*"{name=B}] \cSet \slice F_! X' \ar[from=U, to=B, phantom, "\perp"]
    \end{tikzcd} \]
    where the right adjoint takes an object in the slice $Y \to F_! X'$, applies $F^*$ to obtain $F^* Y \to F^* F_! X'$, then pulls back along the unit map $\eta \from X' \to F^* F_! X'$ to obtain the object in the slice category $\cSeti \slice X'$ depicted on the right in:
    \[ \begin{tikzcd}
        \eta^* (F^* Y) \ar[r] \ar[d] \ar[rd, phantom, "\pbtick" very near start] & F^* Y \ar[d] \\
        X' \ar[r, "\eta"] & F^* F_! X'
    \end{tikzcd} \]
    This functor preserves colimits since applying $F^*$ preserves colimits (because $F^*$ admits a further right adjoint $F_*$) and pulling back along $\eta$ preserves colimits (again, because $\cSet$ is a presheaf category).
    With this, we have a composite adjunction
    \[ \begin{tikzcd}
        \cSeti \slice X' \ar[r, bend left, "F_!"{name=U}] & \ar[l, bend left, "\eta^* F^*"{name=B}] \cSet \slice F_! X' \ar[r, bend left, "{(fw)_!}"{name=UU}] & \ar[l, bend left, "{(fw)^*}"{name=BB}] \cSet \slice \gnerve[\infty] G \ar[from=U, to=B, phantom, "\perp"] \ar[from=UU, to=BB, phantom, "\perp"]
    \end{tikzcd} \]
    where the right adjoint preserves colimits.
    Thus, the sequential colimit
    \[ \big( \gnerve[n+2] G \ito \gnerve[n+3] G \ito \dots \ito \gnerve[m] G \ito \dots \big) \ito \gnerve[\infty] G \]
    in $\cSet \slice \gnerve[\infty] G$ is sent to a sequential colimit in $\cSeti \slice X'$, which we denote by
    \[ \big( X'_{n+2} \ito X'_{n+3} \ito \dots \ito X'_m \ito \dots \big) \ito X'. \]
    The counit of this composite adjunction instantiated at the object $\gnerve[m] G \in \cSet \slice \gnerve[\infty] G$ arises as the top morphism in the commutative square 
    \[ \begin{tikzcd}[column sep = 3.2em, row sep = 2.8em]
        F_! X'_m \ar[r, dotted, "{(fw)_m}"] \ar[d, hook] & \gnerve[m] G \ar[d, hook] \\
        F_! X' \ar[r, "fw"] & \gnerve[\infty] G
    \end{tikzcd} \]
    Denote this morphism by $(fw)_m$.
    As $m$ varies, we obtain a map between filtrations $(fw)_\bullet \from F_! X'_\bullet \to \gnerve[\bullet] G$.
    
    % where $i_m$ denotes the structure map $X'_{m} \ito X'$.

    \textit{Step 3}. For each $m$, we factor the map $(fw)_m \from F_! X'_m \to \gnerve[m] G$ using the universal property of the unit $\eta^m_{F_! X'_m} \from F_! X'_m \to \gnerve[m]\reali[m]{F_! X'_m}$ as
    \[ F_! X'_m \to \gnerve[m]\reali[m]{F_! X'_m} \xrightarrow{\gnerve[m] g_m} \gnerve[m] G \]
    for a unique graph map $g_m \from \reali[m]{F_! X'_m} \to G$.
    We rewrite $\reali[m]{F_! X'_m}$ as $\Gn[m]{X'_m}$ so that $g_m$ is a graph map $\Gn[m] X'_m \to G$.

    Let $i_m$ denote the sequence inclusion $X'_m \ito X'_{m+1}$.
    We claim that if $m$ is odd then the left square in:
    \[ \begin{tikzcd}[column sep = 3.5em, row sep = 3em]
        \Gn[m+1]{X'_m} \ar[r, "{\Gn[m+1]i_m}", hook] \ar[d, "\ell_!"'] & \Gn[m+1]{X'_{m+1}} \ar[d, "g_{m+1}"] \\
        \Gn[m]{X'_m} \ar[r, "g_m"] & G
    \end{tikzcd} \qquad \begin{tikzcd}[column sep = 3.5em, row sep = 3em]
        \Gn[m+1]{X'_m} \ar[r, "{\Gn[m+1]i_m}"] \ar[d, "r_!"'] & \Gn[m+1]{X'_{m+1}} \ar[d, "g_{m+1}"] \\
        \Gn[m]{X'_m} \ar[r, "g_m"] & G
    \end{tikzcd} \]
    commutes, and if $m$ is even then the right square commutes.
    We prove this claim for $m$ odd, as the case of $m$ even is analogous.
    By adjointness, it suffices to verify that the square commutes after applying $\gnerve[m+1]$ and precomposing with the unit $\eta^{m+1}_{F_! X'_m} \from F_! X'_m \to \gnerve[m+1] { \Gn[m+1] X'_m}$:
    \[ \begin{tikzcd}[column sep = 3.4em, row sep = 3em]
        F_! X'_m \ar[r, "\eta^{m+1}_{F_! X'_m}"] & \gnerve[m+1]\Gn[m+1]{X'_m} \ar[r, hook, "{\gnerve[m+1] \Gn[m+1] i_m}"] \ar[d, "{\gnerve[m+1] \ell_!}"'] &[+1.5em] \gnerve[m+1]\Gn[m+1]{X'_{m+1}} \ar[d, "{\gnerve[m+1] g_{m+1}}"] \\
        {} & \gnerve[m+1]\Gn[m]{X'_m} \ar[r, "{\gnerve[m+1] g_m}"] & \gnerve[m+1]G
    \end{tikzcd} \]
    We now compute:
    \[ \begin{array}{r@{\ }l l}
        \gnerve[m+1] g_{m+1} \circ \gnerve[m+1] \Gn[m+1] i_m \circ \eta^{m+1}_{F_! X'_m} &= \gnerve[m+1] g_{m+1} \circ \eta^{m+1}_{F_! X'_{m+1}} \circ F_! i_m & \text{by naturality of $\eta^{m+1}$} \\
        &= (fw)_{m+1} \circ F_! i_m & \text{by definition of $g_{m+1}$} \\
        &= \ell^* \circ (fw)_m & \text{by naturality of $(fw)_\bullet$} \\
        &= \ell^* \circ \gnerve[m] g_m \circ \eta^m_{F_! X'_m} & \text{by definition of $g_m$} \\
        &= \gnerve[m+1] g_m \circ \ell^* \circ \eta^m_{F_! X'_m} & \text{by naturality of $\ell^*$} \\
        &= \gnerve[m+1] g_m \circ \gnerve[m+1] \ell_! \circ \eta^{m+1}_{F_! X'_m} & \text{by \cref{unit-ell-unit-ell-square}},
    \end{array} \]
    hence the desired square commutes.
    
    With this, we have a cone diagram
    % \[ \begin{tikzcd}[sep = large]
    %     {} & {} & {} & {} & {} \\
    %     {} & {} & {} & \Gn[m] X'_m \ar[rddd, "g_{m}"] \ar[ur, phantom, "\iddots"] & {} \\
    %     {} & \Gn[n+4] X'_{n+3} \ar[r, "{\Gn[n+4] i_{n+3}}"] \ar[d, "{\ell_!}"', "{r_!}"] & \Gn[n+4] X'_{n+4} \ar[ur, phantom, "\iddots"] \ar[rrdd, "g_{n+4}"] & {} \\
    %     \Gn[n+3] X'_{n+2} \ar[r, "{\Gn[n+3] i_{n+2}}"] \ar[d, "{\ell_!}"', "{r_!}"] & \Gn[n+3] X'_{n+3} \ar[rrrd, "g_{n+3}"] & {} & {} \\
    %     \Gn[n+2] X'_{n+2} \ar[rrrr, "g_{n+2}"] & {} & {} & {} & G
    % \end{tikzcd} \]
    \[ \begin{tikzcd}[column sep = 0.2em, row sep = 2.4em]
        {} & \Gn[n+3] X'_{n+2} \ar[ld, "\ell_!"', "r_!"] \ar[rd, "{\Gn[n+3] i_{n+2}}" description] & {} & \Gn[n+4] X'_{n+3} \ar[ld, "\ell_!"', "r_!"] \ar[rd, "{\Gn[n+3] i_{n+2}}" description] & {} & {} \ar[ld] \ar[r, phantom, "\dots"] &[+1em] {} \ar[rd] & {} & {} \ar[ld] \ar[r, phantom, "\dots"] &[+1ex] {} \\
        \Gn[n+2] X'_{n+2} & {} & \Gn[n+3] X'_{n+3} & {} & \Gn[n+4] X'_{n+4} & {} & {} & \Gn[m] X'_{m} & {} \\[+1em]
        {} & {} & {} & G \ar[from=ulll, "g_{n+2}"', bend right=15] \ar[from=ul, "g_{n+3}" description, bend right=8] \ar[from=ur, "g_{n+4}" description] \ar[from=urrrr, "g_{m}"]
        % {} & {} & {} & {} & {} & {} & {} & G \ar[from=ulllllll, "g_{n+2}"', bend right=15] \ar[from=ulllll, "g_{n+3}", bend right=8] \ar[from=ulll, "g_{n+3}"] \ar[from=u, "g_{m}"]
    \end{tikzcd} \]
    which continues for all $m \geq n+2$, where the left-facing maps $\Gn[m+1] X'_m \to \Gn[m] X'_m$ are $\ell_!$ when $m$ is odd and $r_!$ when $m$ is even.
    
    Let $H_{m} \in \Graph$ denote the colimit of the truncated diagram
    \[ H_m := \colim \left( \begin{tikzcd}[column sep = 0.2em, row sep = 2.4em]
        {} & \Gn[n+3] X'_{n+2} \ar[ld, "\ell_!"', "r_!"] \ar[rd, "{\Gn[n+3] i_{n+2}}" description] & {} & \Gn[n+4] X'_{n+3} \ar[ld, "\ell_!"', "r_!"] \ar[rd, "{\Gn[n+3] i_{n+2}}" description] & {} & {} \ar[ld] \ar[r, phantom, "\dots"] &[+1em] {} \ar[rd] & {} \\
        \Gn[n+2] X'_{n+2} & {} & \Gn[n+3] X'_{n+3} & {} & \Gn[n+4] X'_{n+4} & {} & {} & \Gn[m] X'_{m}
    \end{tikzcd} \right) \]
    The inclusions of diagrams induce maps $j_m \from H_{m} \to H_{m+1}$ for all $m \geq n+2$, and the colimit of the sequence
    \[ H_\infty := \colim \big( H_{n+2} \xrightarrow{j_{n+2}} H_{n+3} \xrightarrow{j_{n+3}} H_{n+4} \to \dots \big) \]
    is the colimit of the infinite diagram
    \[ H_\infty = \colim \left(  \begin{tikzcd}[column sep = 0.2em, row sep = 2.4em]
        {} & \Gn[n+3] X'_{n+2} \ar[ld, "\ell_!"', "r_!"] \ar[rd, "{\Gn[n+3] i_{n+2}}" description] & {} & \Gn[n+4] X'_{n+3} \ar[ld, "\ell_!"', "r_!"] \ar[rd, "{\Gn[n+3] i_{n+2}}" description] & {} & {} \ar[ld] \ar[r, phantom, "\dots"] &[+1em] {} \ar[rd] & {} & {} \ar[ld] \ar[r, phantom, "\dots"] &[+1ex] {} \\
        \Gn[n+2] X'_{n+2} & {} & \Gn[n+3] X'_{n+3} & {} & \Gn[n+4] X'_{n+4} & {} & {} & \Gn[m] X'_{m} & {}
        % {} & {} & {} & G \ar[from=ulll, "g_{n+2}"', bend right=15] \ar[from=ul, "g_{n+3}" description, bend right=8] \ar[from=ur, "g_{n+4}" description] \ar[from=urrrr, "g_{m}"]
        % {} & {} & {} & {} & {} & {} & {} & G \ar[from=ulllllll, "g_{n+2}"', bend right=15] \ar[from=ulllll, "g_{n+3}", bend right=8] \ar[from=ulll, "g_{n+3}"] \ar[from=u, "g_{m}"]
    \end{tikzcd} \right) \]
    From this, it follows that the map $g_m \from \Gn[m] X'_m \to G$ factors as:
    \[ \Gn[m] X'_m \to H_m \to H_\infty \to G . \]

    For $m \geq n+2$, we show the colimit cone component $\lambda_m \from \Gn[m] X'_m \to H_m$ is an $n$-equivalence.
    In the base case, $\Gn[n+2] X'_{n+2} \to H_{n+2}$ is an isomorphism by definition.
    For the inductive step, observe that the colimit $H_{m+1}$ can be computed as a pushout:
    \[ \begin{tikzcd}[column sep = 3.8em]
        \Gn[m+1] X'_{m} \ar[d, "{\ell_!}", "{r_!}"'] \ar[r, "{\Gn[m+1] i_m}", hook] & \Gn[m+1] X'_{m+1} \ar[dd, "\lambda_{m+1}"] \\
        \Gn[m] X'_m \ar[d, "\lambda_m"'] \ar[rd, phantom, "\potick" very near end] & {} \\
        H_m \ar[r, hook, "j_{m}"] & H_{m+1} \tag{$\ast$}
    \end{tikzcd} \]
    % where the map $H_m \to H_{m+1}$ coincides with the diagram restriction map.
    By the inductive hypothesis, the bottom left map is an $n$-equivalence.
    By \cref{ell-r-n-equiv}, the top left map is an $n$-equivalence.
    The claim now follows from \cref{po-weq-along-cubical-mono-is-weq}.
    
    \textit{Step 4}. Applying the functors $\gnerve[n+2], \gnerve[n+3], \dots$ to the sequence
    \[ \big( H_{n+2} \ito H_{n+3} \ito H_{n+4} \ito \dots \ito H_{m} \ito \dots \big) \to H_\infty \to G \]
    yields a commutative grid:
    \[ \begin{tikzcd}[column sep = small, cramped]
        { \big( \gnerve[n+2] H_{n+2}} \ar[r, hook] \ar[d, "{\ell^*}"', "{r^*}", hook] & {\gnerve[n+2] H_{n+3}} \ar[r, hook] \ar[d, "{\ell^*}"', "{r^*}"] & {\gnerve[n+2] H_{n+4}} \ar[r, hook] \ar[d, "{\ell^*}"', "{r^*}"] & {\dots} \ar[r, hook] & {\gnerve[n+2] H_{m}} \ar[r, hook] \ar[d, "{\ell^*}"', "{r^*}", hook] & {\dots \big)} \ar[r, hook] & {\gnerve[n+2] H_\infty} \ar[r] \ar[d, "{\ell^*}"', "{r^*}"] & {\gnerve[n+2] G} \ar[d, "{\ell^*}"', "{r^*}"] \\
        { \big( \gnerve[n+3] H_{n+2}} \ar[r, hook] \ar[d, "{\ell^*}"', "{r^*}", hook] & {\gnerve[n+3] H_{n+3}} \ar[r, hook] \ar[d, "{\ell^*}"', "{r^*}", hook] & {\gnerve[n+3] H_{n+4}} \ar[r, hook] \ar[d, "{\ell^*}"', "{r^*}"] & {\dots} \ar[r, hook] & {\gnerve[n+3] H_{m}} \ar[r, hook] \ar[d, "{\ell^*}"', "{r^*}", hook] & {\dots \big)} \ar[r, hook] & {\gnerve[n+3] H_\infty} \ar[r] \ar[d, "{\ell^*}"', "{r^*}", hook] & {\gnerve[n+3] G} \ar[d, "{\ell^*}"', "{r^*}"] \\
        % { \big( \gnerve[n+4] H_{n+2}} \ar[r, hook] \ar[d, "{\ell^*}"', "{r^*}"] & {\gnerve[n+4] H_{n+3}} \ar[r, hook] \ar[d, "{\ell^*}"', "{r^*}"] & {\gnerve[n+4] H_{n+4}} \ar[r, hook] \ar[d, "{\ell^*}"', "{r^*}"] & {\dots} \ar[r, hook] & {\gnerve[n+4] H_{m}} \ar[r, hook] \ar[d, "{\ell^*}"', "{r^*}"] & {\dots \big)} \ar[r, hook] & {\gnerve[n+4] H_\infty} \ar[r] \ar[d, "{\ell^*}"', "{r^*}"] & {\gnerve[n+4] G} \ar[d, "{\ell^*}"', "{r^*}"] \\
        % { \big( \gnerve[n+4] H_{n+2}} \ar[r, hook] & {\gnerve[n+4] H_{n+3}} \ar[r, hook] & {\gnerve[n+4] H_{n+4}} \ar[r, hook] & {\dots} \ar[r, hook] & {\gnerve[n+4] H_{m}} \ar[r, hook] & {\dots \big)} \ar[r, hook] & {\gnerve[n+4] H_\infty} \ar[r] & {\gnerve[n+4] G} \\
        \vdots \ar[d, hook] & \vdots \ar[d, hook] & \vdots \ar[d, hook] & \ddots & \vdots \ar[d, hook] & \ddots & \vdots \ar[d, hook] & \vdots \ar[d, hook] \\
        { \big( \gnerve[\infty] H_{n+2}} \ar[r, hook] & {\gnerve[\infty] H_{n+3}} \ar[r, hook] & {\gnerve[\infty] H_{n+4}} \ar[r, hook] & {\dots} \ar[r, hook] & {\gnerve[\infty] H_{m}} \ar[r, hook] & {\dots \big)} \ar[r, hook] & {\gnerve[\infty] H_\infty} \ar[r] & {\gnerve[\infty] G}
    \end{tikzcd} \]
    where each vertical column is a sequential colimit.
    From this, we extract a map from the diagonal filtration to the rightmost filtration
    \[ \begin{tikzcd}
        \gnerve[n+2] H_{n+2} \ar[r] \ar[d, hook] & \gnerve[n+2] G \ar[d, hook] \\
        \gnerve[n+3] H_{n+3} \ar[r] \ar[d, hook] & \gnerve[n+3] G \ar[d, hook] \\
        \vdots \ar[d, hook] & \vdots \ar[d, hook] \\
        \gnerve[\infty] H_\infty \ar[r] & \gnerve[\infty] G
    \end{tikzcd} \]
    where, again, both vertical columns are sequential colimits.

    For $m \geq n+2$, the square
    \[ \begin{tikzcd}[column sep = 4em, row sep = 3em]
        F_! X'_{m} \ar[r, "{\eta^{m}_{F_! X'_m}}"] \ar[d, "F_! i_m"'] & \gnerve[m] \Gn[m]{X'_m} \ar[r, "{\gnerve[m] \lambda_m }"] & \gnerve[m] H_m \ar[d, "{\ell^* \circ \gnerve[m]j_m}"] \\
        F_! X'_{m+1} \ar[r, "{\eta^{m+1}_{F_! X'_{m+1}}}"] & \gnerve[m+1] \Gn[m+1]{X'_{m+1}} \ar[r, "{\gnerve[m+1] \lambda_{m+1}}"] & \gnerve[m+1] H_{m+1}
    \end{tikzcd} \]
    commutes by a straightforward calculation
    \[ \begin{array}{r@{\ }l l}
        \ell^* \circ \gnerve[m]j_m \circ \gnerve[m]\lambda_m \circ \eta^{m}_{F_! X'_m} &= \ell^* \circ \gnerve[m] (j_m \circ \lambda_m) \circ \eta^{m}_{F_! X'_m} \\
        &= \gnerve[m+1](j_m \circ \lambda_m) \circ \ell^* \circ \eta^{m}_{F_! X'_m} & \text{by naturality of $\ell^*$} \\
        &= \gnerve[m+1](j_m  \circ \lambda_m) \circ \gnerve[m+1] \ell_! \circ \eta^{m+1}_{F_! X'_m} & \text{by \cref{unit-ell-unit-ell-square}} \\
        &= \gnerve[m+1] (j_m \circ \lambda_m \circ \ell_!) \circ \eta^{m+1}_{F_! X'_m} \\
        &= \gnerve[m+1] (\lambda_{m+1} \circ \Gn[m+1] i_m) \circ \eta^{m+1}_{F_! X'_m} & \text{by commutativity of $(\ast)$} \\
        &= \gnerve[m+1] \lambda_{m+1} \circ \gnerve[m+1] \Gn[m+1] i_m \circ \eta^{m+1}_{F_! X'_m} \\
        &= \gnerve[m+1] \lambda_{m+1} \circ \eta^{m+1}_{F_! X'_{m+1}} \circ F_! i_m & \text{by naturality of $\eta^{m+1}$}.
    \end{array} \]
    With this, we have a map between filtrations $F_! X'_{\bullet} \to \gnerve[\bullet] H_\bullet$.
    Moreover, the components of the filtration
    \[ F_! X'_m \xrightarrow{\eta} \gnerve[m] \Gn[m] X'_m \xrightarrow{\gnerve[m] \lambda_m} \gnerve[m] H_m \]
    are $n$-equivalences: the left map $\eta$ is an $n$-equivalence by \cref{Gn-weq}, and we showed previously that the right map is an $n$-equivalence.
    Both filtrations consist of monomorphisms, hence they are homotopy colimits in the Cisinski model structure for $n$-types.
    From this, it follows that the induced map on colimits $F_! X' \to \gnerve[\infty] H_\infty$ is an $n$-equivalence.

    It remains only to show that the square
    \[ \begin{tikzcd}
        X \ar[r, "f"] & \gnerve[\infty] G \\
        F_! X' \ar[u, "w", "\sim"'] \ar[r, "\sim"] & \gnerve[\infty] H_\infty \ar[u]
    \end{tikzcd} \]
    commutes.
    This can be rephrased as commutativity of the triangle
    \[ \begin{tikzcd}
        F_! X' \ar[rr, "fw"] \ar[rd, "\sim"] & {} & \gnerve[\infty] G \\
        {} & \gnerve[\infty] H \ar[ur, "\sim"]
    \end{tikzcd} \]
    This triangle arises from a triangle of maps between sequential colimit diagrams whose components are given by:
    \[ \begin{tikzcd}
        F_! X'_m \ar[rr, "(fw)_m"] \ar[rd, "\sim"] & {} & \gnerve[m] G \\
        {} & \gnerve[m] H_m \ar[ur, "\sim"]
    \end{tikzcd} \]
    By construction, $(fw)_m$ factored as
    \[ \begin{tikzcd}
        F_! X'_m \ar[r, "{\eta_{F_! X'_m}}"] \ar[d, "\sim"'] & \gnerve[m] \Gn[m] X'_m \ar[d, "{\gnerve[m] g_m}"] \\
        \gnerve[m] H_m \ar[r, "\sim"] & \gnerve[m] G 
    \end{tikzcd} \]
    and, again, $g_m$ factored as
    \[ \begin{tikzcd}
        F_! X'_m \ar[r, "{\eta_{F_! X'_m}}"] \ar[d, "\sim"'] & \gnerve[m] \Gn[m] X'_m \ar[d, "{\gnerve[m] g_m}"] \ar[ld, "{\gnerve[m] \lambda_m}" description] \\
        \gnerve[m] H_m \ar[r, "\sim"] & \gnerve[m] G 
    \end{tikzcd} \]
    The top triangle now commutes by definition.
\end{proof}

 \section{Computing explicit discrete homotopy groups} \label{sec:applications}

In this final section of the paper, we derive some consequences of our earlier results as they pertain to specific applications of new discrete homotopy invariants.
One particularly active area of research is identifying discrete models of spheres, with several candidates proposed, for example, $\bd \gcube{5}{3}$.
Despite the seeming similarity to $S^2$, its discrete homotopy groups were only known up to degree $2$.
Contemporary research on the topic includes work by Babson who, during his talk at the American Institute for Mathematics in 2023, sketched a plausible way of constructing a non-trivial element in its third discrete homotopy group.

Our methods subsume previous approaches, allowing us to compute the entire group, rather than construct a non-trivial element therein, and additionally give a description of the fourth and fifth discrete homotopy groups.
For larger models of spheres, the situation is even better and, likewise, more can be said about tentative models of higher $S^n$'s.

In the second part of this section, we give an explicit prescription for turning a simplicial complex into a graph whose nerve is $n$-equivalent to the complex.
Given the generality in which the construction works (an arbitrary simplicial complex), we lament that it quickly balloons in size.
Indeed, to recover the $2$-type of $\mathbb{RP}^2$, we construct a graph with $1,801$ vertices and more than $2,700$ edges.

\subsection{Discrete models of spheres}

We begin by showing that the cylinder functor $\Cyl{m} \from \Fun(\Lambda, \Graph) \to \Graph$ is homotopical with respect to $n$-equivalences, provided the length $m$ is sufficiently large.
In particular, this shows that the suspension functor $\susp[m] \from \Graph \to \Graph$ preserves $n$-equivalences for large enough $m$.
\begin{theorem} \label{Cyl-preserves-n-equivs}
    Let $\Lambda$ denote the category with three objects $0, 1, 2$ and morphisms $0 \to 1$ and $0 \to 2$.
    For $m \geq 0$, the functor
    \[ \Cyl{m+2} \from \Fun(\Lambda, \Graph) \to \Graph \]
    sends objectwise $n$-equivalences between diagrams to $n$-equivalences for all $n \in \{ 0, \dots, m-2 \}$.
\end{theorem}
\begin{proof}
    Suppose
    \[ \begin{tikzcd}
        H \ar[d, "\sim"'] & G \ar[l, "f"'] \ar[r, "g"] \ar[d, "\sim"] & K \ar[d, "\sim"] \\
        H' & G' \ar[l, "{f'}"'] \ar[r, "{g'}"] & K'
    \end{tikzcd} \]
    is a morphism of diagrams which is an objectwise $n$-equivalence.
    From this, we obtain a commutative cube
    \[ \begin{tikzcd}
        \gnerve[1]{(G \gtimes I_m)} \ar[rr] \ar[dd] \ar[rd] & {} & \gnerve[1]{\Cyl{m+1}(f, \id[G])} \ar[dd] \ar[rd] & {} \\
        {} & \gnerve[1]{(G' \gtimes I_m)} \ar[rr, crossing over] & {} & \gnerve[1]{\Cyl{m+1}(f', \id[G'])} \ar[dd] \\
        \gnerve[1]{\Cyl{m+1}(\id[G], g)} \ar[rr] \ar[rd] & {} & \gnerve[1]{\Cyl{m+2}(f, g)} \ar[rd] & {} \\
        {} & \gnerve[1]{\Cyl{m+1}(\id[G'], g')} \ar[rr] \ar[from=uu, crossing over] & {} & \gnerve[1]{\Cyl{m+2}(f', g')}
    \end{tikzcd} \] 
    where the front and back squares are homotopy pushouts by \cref{double-mapping-cyl-inclusions-hopo}.
    To show that the diagonal map between the bottom-right objects is an $n$-equivalence, it suffices to show that the other three diagonal maps are $n$-equivalences.

    The three diagonal maps are obtained by applying $\gnerve[1]$ to the top maps in the squares:
    \[ \begin{tikzcd}
        G \gtimes I_m \ar[r] \ar[d, "\sim"'] & G' \gtimes I_m \ar[d, "\sim"] \\
        G \ar[r, "\sim"] & G'
    \end{tikzcd} \quad \begin{tikzcd}
        H \gtimes I_m \ar[r] \ar[d, "\sim"'] & H' \gtimes I_m \ar[d, "\sim"] \\
        H \ar[r, "\sim"] & H'
    \end{tikzcd} \quad \begin{tikzcd}
        K \gtimes I_m \ar[r] \ar[d, "\sim"'] & K' \gtimes I_m \ar[d, "\sim"] \\
        K \ar[r, "\sim"] & K'
    \end{tikzcd} \]
    The vertical maps are homotopy equivalences and the bottom maps are $n$-equivalences by assumption, hence the top maps are $n$-equivalences.
\end{proof}
\begin{corollary} \label{susp-preserves-n-equivs}
    For $m \geq 0$, the suspension functor $\susp[m+2] \from \Graph \to \Graph$ preserves $n$-equivalences for all $n \in \{ 0, \dots, m \}$. \qed
\end{corollary}

% \subsection{Discrete models of spheres}

We use \cref{cyl-cset-graph-equiv-mono,cyl-cset-graph-equiv-zigzag} to construct discrete models of the topological spheres up to $n$-equivalence.

Recall from \cref{ex:boundary-as-cyl} that the boundary of the $k$-cube $\bd \gcube{m}{k}$ of length $m$ arises as a double mapping cylinder.
The analogous statement is true in cubical sets, giving the following:
\begin{theorem} \label{sphere-model-bd-n-equiv}
    For $k, n \geq 0$ and $m \geq n$, the unit map
    \[ \eta_{\bd \cube{k}} \from \bd \cube{k} \to \gnerve[m]{\bd \gcube{m}{k}} \]
    is an $n$-equivalence.
\end{theorem}
\begin{proof}
    For $n = 0$, this is straightforward to verify.
    Thus, we may assume $n \geq 1$ (and hence $m \geq 1$).

    We now proceed by induction on $k$.
    The case $k = 0$ follows since both the domain and codomain are empty.
    For the inductive step, fix $k \geq 0$ and suppose $\eta_{\bd \cube{k}}$ is an $n$-equivalence.
    The cubical set $\bd \cube{k+1}$ is a double-mapping cylinder $\Cyl{}^{\boxcat}(i^k, i^k)$, where $i^k$ denotes the boundary inclusion $\bd \cube{k} \ito \cube{k}$.
    By \cref{cyl-cset-graph-equiv-mono}, it suffices to show the unit morphisms $\eta_{\bd \cube{k}}$ and $\eta_{\cube{k}}$ are $n$-equivalences.
    For $\eta_{\bd \cube{k}} \from \bd \cube{k} \to \gnerve[m]{\bd \gcube{m}{k}}$, we apply the inductive hypothesis.
    For $\eta_{\cube{k}} \from \cube{k} \to \gnerve[m]{\gcube{m}{k}}$, we observe that both the domain and codomain are contractible.
\end{proof}
In particular, setting $m = n$, the boundary $\bd \gcube{n}{k}$ of the $k$-cube of length $n$ is a model for the $(k-1)$-sphere up to $n$-equivalence.
Homotopy and homology groups up to dimension $n$ are invariant under $n$-equivalence (\cref{cubical-homology-inverts-equivs}), yielding the following computation:
\begin{corollary}
    For $m, k \geq 0$, we have isomorphisms
    \[ A_n(\bd \gcube{m}{k}) \cong \pi_n(S^{k-1}) \quad H_n(\bd \gcube{m}{k}) \cong H_n(S^{k-1}). \]
    for all $n \in \{ 0, \dots, m \}$. \qed
\end{corollary}

Another model for the $k$-sphere comes from the suspension functor $\susp[m] \from \Graph \to \Graph$.
We write $\susp[\boxcat] \from \cSet \to \cSet$ for the functor that sends a cubical set $X$ to the double-mapping cylinder $\Cyl{}^{\boxcat}(\bang, \bang)$ of the terminal map $\bang \from X \to \cube{0}$ with itself.
Additionally, let $\cube{1} / \bd \cube{1}$ denote the pushout:
\[ \begin{tikzcd}
    \bd \cube{1} \ar[r, hook] \ar[d] \ar[rd, phantom, "\potick" very near end] & \cube{1} \ar[d] \\
    \cube{0} \ar[r] & \cube{1} / \bd \cube{1}
\end{tikzcd} \]
Explicitly, $\cube{1} / \bd \cube{1}$ consists of a single non-degenerate 1-cube whose source and target have been identified.
The geometric realization of $\cube{1} / \bd \cube{1}$ is homeomorphic to the circle $S^1$.
\begin{theorem}
    For $k \geq 0$, $m \geq k$ and $L \geq 5$, there is a zig-zag of $n$-equivalences:
    \[ \susp[\boxcat]^k (\cube{1} / \bd \cube{1}) \simeq \gnerve[1]{\susp[m+2]^k(C_L)}. \]
\end{theorem}
\begin{proof}
    Let $C_L^{\boxcat} \in \cSet$ denote the coequalizer:
    \[ C_L^{\boxcat} := \colim \left( \begin{tikzcd}[column sep = 4em]
        \displaystyle \coprod\limits_{k \in \ZZ / L} \cube{0} \ar[r, yshift=1ex, "{[i_k \circ \face{}{1,1}]_k}"] \ar[r, yshift=-1ex, "{[i_{k + 1} \circ \face{}{1,0}]_k}"'] & \displaystyle \coprod\limits_{k \in \ZZ / L} \cube{1}
    \end{tikzcd} \right) \]
    That is, $C_L^{\boxcat}$ is the cubical set with $L$ many non-degenerate 1-cubes glued into a directed cycle. 
    The map $C_L^{\boxcat} \to \cube{1} / \bd \cube{1}$ which contracts all but one 1-cube to a point is a weak equivalence.
    Since the cubical suspension functor $\susp[\boxcat] \from \cSet \to \cSet$ preserves weak equivalences, we are reduced to constructing a zig-zag of $n$-equivalences:
    \[ \susp[\boxcat]^k (C_L^{\boxcat}) \simeq \gnerve[m]{\susp[m+2]^k(C_L)}. \]
    Similarly, the map between cycle graphs $C_{mL} \to C_L$ defined by the assignment $t \mapsto \lfloor t / m \rfloor$ is a weak equivalence; the domain and codomain have trivial discrete homotopy groups except in dimension 1 (see \cite[Thm.~1.1]{lutz}), and manual inspection shows this map induces an isomorphism on discrete fundamental groups.
    By \cref{susp-preserves-n-equivs}, we are reduced to constructing a zig-zag of $n$-equivalences:
    \[ \susp[\boxcat]^k (C_L^{\boxcat}) \simeq \gnerve[m]{\susp[m+2]^k(C_{mL})}. \]

    We now proceed by induction on $k$.
    For the base case $k = 0$, we observe that the $m$-realization of the cubical set $C_L^{\boxcat}$ is the cycle graph $C_{mL}$; that is, we have an isomorphism $\reali[m]{C_L^{\boxcat}} \cong C_{mL}$.
    With this, the unit of the realization-nerve adjunction instantiates to a map $C_L^{\boxcat} \to \gnerve[m]{C_{mL}}$.
    By \cref{nerve-main-thm}, both objects have trivial homotopy groups except in dimension $1$, hence one can directly show this map is a weak equivalence. 

    For the inductive step, fix $k \geq 0$ and suppose we have a zig-zag of $n$-equivalences $\susp[\boxcat]^k (C_L^{\boxcat}) \simeq \gnerve[m]{\susp[m+2]^k(C_{mL})}$.
    Observe that, for any cubical set $X \in \cSet$, the cubical suspension $\susp[\boxcat] X$ can be written as a pushout
    \[ \begin{tikzcd}
        X \gprod \bd \cube{1} \ar[r, hook] \ar[d, "\pi_{\bd \cube{1}}"'] \ar[rd, phantom, "\potick" very near end] & X \gprod \cube{1} \ar[d] \\
        \bd \cube{1} \ar[r] & \susp[\boxcat] X
    \end{tikzcd} \] 
    This is a pushout along a monomorphism, hence a homotopy pushout in the Cisinski model structure for $n$-types.
    From this fact, one can deduce that the cubical suspension functor preserves $n$-equivalences.
    Combining this fact with our inductive hypothesis, we have a zig-zag of $n$-equivalences:
    \[ \susp[\boxcat] \big( \susp[\boxcat]^k (C_L^{\boxcat}) \big) \simeq \susp[\boxcat] \big( \gnerve[m]{\susp[m+2]^k(C_{mL})} \big). \]
    By \cref{unit-graph-weq}, the unit of the realization-nerve adjunction applied to the cubical set $\gnerve[m]{\susp[m+2]^k(C_{mL})}$ is a weak equivalence.
    We instantiate \cref{cyl-cset-graph-equiv-zigzag} at $X = \gnerve[m]{\susp[m+2]^k(C_{mL})}$ and $Y = Z = \cube{0}$ to further extend the right-hand side by a zig-zag of $n$-equivalences:
    \[ \susp[\boxcat] \big( \susp[\boxcat]^k (C_L^{\boxcat}) \big) \simeq \susp[\boxcat] \big( \gnerve[m]{\susp[m+2]^k(C_{mL})} \big) \simeq \gnerve[m]\big( \susp[m+2] \reali[m]{ \gnerve[m]{\susp[m+2]^k(C_{mL})} } \big) \]
    The counit $\reali[m]{ \gnerve[m]{\susp[m+2]^k(C_{mL})} } \to {\susp[m+2]^k(C_{mL})}$ is a weak equivalence by \cref{counit-acyclic-fib}, so it is an $n$-equivalence.
    Since graph suspension preserves $n$-equivalences (\cref{susp-preserves-n-equivs}), we obtain the desired zig-zag:
    \[ \susp[\boxcat] \big( \susp[\boxcat]^k (C_L^{\boxcat}) \big) \simeq \susp[\boxcat] \big( \gnerve[m]{\susp[m+2]^k(C_{mL})} \big) \simeq \gnerve[m]\big( \susp[m+2] \reali[m]{ \gnerve[m]{\susp[m+2]^k(C_{mL})} } \big) \simeq \gnerve[m]\big( \susp[m+2] {\susp[m+2]^k(C_{mL})} \big). \qedhere \]
\end{proof}
\begin{corollary} \label{suspension-htpy-groups}
    For $k, m \geq 0$, and $L \geq 5$ we have isomorphisms
    \[ A_n \big( \susp[m+2]^k{(C_L)} \big) \cong \pi_n(S^{k+1}) \quad H_n \big( \susp[m+2]^k{(C_L)} \big) \cong H_n(S^{k+1}) \]
    for all $n \in \{ 0, \dots, m \}$. \qed
\end{corollary}
The above result allows us to identify a large number of homotopy groups of $\bd \gcube{m}{n}$ and $\susp[m]^k{C_L}$ that were previously unknown. 
We organize some of these in the following table:
\setlength{\tabcolsep}{4pt} % temporarily reduce spacing between columns for this larger table
\begin{table}[H]
    \centering
    \begin{tabular}{c c c c c c c c c c c c c c c c c}
        \toprule
        {} & \multicolumn{14}{c}{Boundaries of a graph hypercube} \\
        % \cmidrule(lr){2-15}
        {} & \multicolumn{5}{c}{Models of $S^2$} & \multicolumn{5}{c}{Models of $S^3$} & \multicolumn{4}{c}{Models of $S^4$} \\
        \cmidrule(lr){2-6} \cmidrule(lr){7-11} \cmidrule(lr){12-15}
        {} & $\bd \gcube{2}{3}$ & $\bd \gcube{3}{3}$ & $\bd \gcube{4}{3}$ & $\bd \gcube{5}{3}$ & $\bd \gcube{6}{3}$ & $\bd \gcube{2}{4}$ & $\bd \gcube{3}{4}$ & $\bd \gcube{4}{4}$ & $\bd \gcube{5}{4}$ & $\bd \gcube{6}{4}$ & $\bd \gcube{7}{5}$ & $\bd \gcube{8}{5}$ & $\bd \gcube{9}{5}$ & $\bd \gcube{10}{5}$ \\
        \midrule
        $A_1$ & 0 & 0 & 0 & 0 & 0 & 0 & 0 & 0 & 0 & 0 & 0 & 0 & 0 & 0 \\
        $A_2$ & $\ZZ$ & $\ZZ$ & $\ZZ$ & $\ZZ$ & $\ZZ$ & 0 & 0 & 0 & 0 & 0 & 0 & 0 & 0 & 0 \\
        $A_3$ & {} & $\ZZ$ & $\ZZ$ & $\ZZ$ & $\ZZ$ & {} & $\ZZ$ & $\ZZ$ & $\ZZ$ & $\ZZ$ & 0 & 0 & 0 & 0 \\
        $A_4$ & {} & {} & $\ZZ_{2}$ & $\ZZ_{2}$ & $\ZZ_{2}$ & {} & {} & $\ZZ_{2}$ & $\ZZ_{2}$ & $\ZZ_{2}$ & $\ZZ$ & $\ZZ$ & $\ZZ$ & $\ZZ$ \\
        $A_5$ & {} & {} & {} & $\ZZ_{2}$ & $\ZZ_{2}$ & {} & {} & {} & $\ZZ_{2}$ & $\ZZ_{2}$ & $\ZZ_{2}$ & $\ZZ_{2}$ & $\ZZ_{2}$ & $\ZZ_{2}$ \\
        $A_6$ & {} & {} & {} & {} & $\ZZ_{12}$ & {} & {} & {} & {} & $\ZZ_{12}$ & $\ZZ_{2}$ & $\ZZ_{2}$ & $\ZZ_{2}$ & $\ZZ_{2}$ \\
        $A_7$ & {} & {} & {} & {} & {} & {} & {} & {} & {} & {} & $\ZZ \times \ZZ_{12}$ & $\ZZ \times \ZZ_{12}$ & $\ZZ \times \ZZ_{12}$ & $\ZZ \times \ZZ_{12}$ \\
        $A_8$ & {} & {} & {} & {} & {} & {} & {} & {} & {} & {} & {} & $\ZZ_{2} \times \ZZ_{2}$ & $\ZZ_{2} \times \ZZ_{2}$ & $\ZZ_{2} \times \ZZ_{2}$ \\
        $A_9$ & {} & {} & {} & {} & {} & {} & {} & {} & {} & {} & {} & {} & $\ZZ_{2} \times \ZZ_{2}$ & $\ZZ_{2} \times \ZZ_{2}$ \\
        $A_{10}$ & {} & {} & {} & {} & {} & {} & {} & {} & {} & {} & {} & {} & {} & $\ZZ_{24} \times \ZZ_{3}$ \\
        \midrule
        {} & $\susp[4]$ & $\susp[5]$ & $\susp[6]$ & $\susp[7]$ & $\susp[8]$ & $\susp[4]^2$ & $\susp[5]^2$ & $\susp[6]^2$ & $\susp[7]^2$ & $\susp[8]^2$ & $\susp[9]^3$ & $\susp[10]^3$ & $\susp[11]^3$ & $\susp[12]^3$ \\
        \cmidrule(lr){2-6} \cmidrule(lr){7-11} \cmidrule(lr){12-15}
        {} & \multicolumn{5}{c}{Models of $S^2$} & \multicolumn{5}{c}{Models of $S^3$} & \multicolumn{4}{c}{Models of $S^4$} \\
        {} & \multicolumn{14}{c}{Suspensions of a cycle graph of length $\geq 5$} \\
        \bottomrule
    \end{tabular}
    \caption{Discrete homotopy groups of the various models of spheres.
    Blank entries indicate unknown/unproven values.}
\end{table}
\setlength{\tabcolsep}{6pt} % Restore column seperation to default value

\subsection{Constructing an $n$-equivalent graph from a simplicial complex}

We conclude by unwinding the inverse constructed in \cref{nerve-equiv-main-thm} into an explicit construction that takes a simplicial complex as input and outputs a graph whose discrete homotopy groups agree with those of the starting simplicial complex up to dimension $n$.

By a simplicial complex, we mean a pair $(X, \mathcal{A})$ where $X$ is a set and $\mathcal{A}$ is a collection of finite subsets of $X$ which is downward-closed.
For convenience, we often refer to the data of the pair $(X, \mathcal{A})$ as simply $X$.
We refer to a subset $A \in \mathcal{A}$ with $(n+1)$-elements as an \emph{$n$-simplex} of $X$.
We assume the set $X$ is equipped with a total ordering; this total ordering may be used to view the simplicial complex as a simplicial set by declaring the $n$-simplices to be the order-preserving maps from the linear order $[n]$ to $X$ whose image is in $\mathcal{A}$.

To $X$, we associate a semicubical set $CX$ as follows: an $n$-cube in $CX$ consists of a pair of non-empty simplices $\varnothing \neq B \subseteq A$ such that the set complement $A - B$ has $n$ elements.
To compute the faces of an $n$-cube $(A, B)$, we first use the ordering on $X$ to sort the vertices of $A - B$ as
\[ A - B = \{ x_1 < x_2 < \dots < x_{n} \}. \]
For $i \in \{ 1, \dots, n \}$ and $\eps = 0, 1$, we define the face $(A, B)\face{}{i,\eps}$ by
\[ (A, B)\face{}{i, \eps} := \begin{cases}
    \big( A - \{ x_i \}, B \big) & \text{if } \eps = 0 \\
    \big( A, B \sqcup \{ x_i \} \big) & \text{if } \eps = 1.
\end{cases} \]
In particular, a pair $(A', B')$ is a face of $(A, B)$ if $A' \subseteq A$ and $B' \supseteq B$.
A top-dimensional cube of $CX$ is a pair of the form $(A, \{ x \})$ where $A$ is a top-dimensional simplex of $X$.

When viewing $X$ as a simplicial set, the semicubical set $CX$ coincides with its cubification as defined in \cref{sec:fondind-cubification}.
In particular, \cref{C-recovers-Sd} implies the geometric realization of $CX$ is homeomorphic to (the subdivision of) the geometric realization of $X$.

Given $m \geq 0$, we may combine the explicit description of $CX$ with the explicit description guaranteed by \cref{Gn-explicit-description} to obtain the following description of $\Gn[m]{(CX)}$:
\begin{itemize}
    \item the vertices of $\Gn[m]{(CX)}$ are tuples $\big( A, B, (w; P_i^\pm)_{i=1}^k \big)$ where $B$ and $A$ are non-empty nested simplices of $X$, and $(w; P_i^\pm)_{i=1}^k$ is an $F$-sequence in $F(m, n)$, where $n$ is the cardinality of the set complement $A - B$.
    These pairs are subject to the identifications:
    \[ \big( A - \{ x_i \}, B, (w; P_i^\pm)_{i=1}^k \big) \sim \big(A, B, (0; i^-) \cdot (w; P_i^\pm)_{i=1}^k \big) \text{ and } \big( A, B \sqcup \{ x_i \}, (w; P_i^\pm)_{i=1}^k \big) \sim \big(A, B, (0; i^+) \cdot (w; P_i^\pm)_{i=1}^k \big) \]
    for any element $x_i$ in the set complement $A - B = \{ x_1 < \dots < x_n \}$.
    \item edges in $\Gn[m]{(CX)}$ are given by pairs of the form $\big( A, B, (w_0; (P_0)_i^\pm)_{i=1}^{k_0} \big) \sim \big( A, B, (w_1; (P_1)_i^\pm)_{i=1}^{k_1} \big)$ where $(w_0, (P_0)_i^\pm)_{i=1}^{k_0}$ and $(w_1; (P_1)_i^\pm)_{i=1}^{k_1}$ admit expanded forms $(w'_0; (P'_0)_i^\pm)_{i=1}^{n}$ and $(w'_1; (P'_1)_i^\pm)_{i=1}^{n}$ such that: 
    \begin{itemize}
        \item the sign functions of $(w'_0; (P'_0)_i^\pm)_{i=1}^{n}$ and $(w'_1; (P'_1)_i^\pm)_{i=1}^{n}$ are equal;
        \item the partition functions $P'_0 = P'_1$ are equal; and
        \item there exists $i \in \{ 1, \dots, n \}$ such that $|w'_0(i) - w'_1(i)| \leq 1$ and $w'_0(j) = w'_1(j)$ for all $j \neq i$ in $\{ 1, \dots, n \}$.
    \end{itemize}
\end{itemize}
Note that the semicubical set $CX$ is nonsingular; this is an instance of \cref{C-preserves-nonsingular}, since $X$ can be viewed as a nonsingular simplicial set.
As a result, given a fixed $n$, the graph $\Gn[n+1]{X}$ is a graph whose discrete homotopy and homology groups match those of $X$ up to dimension $\leq n$ by \cref{Gn-weq}.

% \begin{remark}
    Even for small simplicial complexes $X$ and small values of $m$, the graph $\Gn[m]{(CX)}$ can be quite large.
    As an example, consider the simplicial complex $X$ depicted on the left in:
    
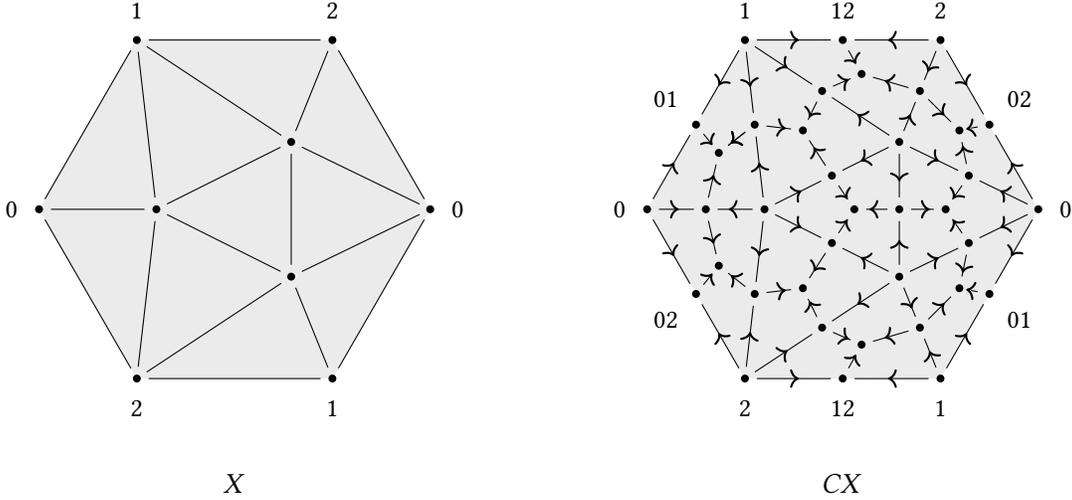
\[ \begin{tikzpicture}[
        scale=1.3, 
        decoration={markings,mark=at position 0.57 with {\arrow[thick]{>}}}
    ]
        % \coordinate (O) at (0, 0);
        \node[Zcube, label={west:$0$}] (a) at (180:2) {};
        \node[Zcube, label={south:$2$}] (b) at (-120:2) {};
        \node[Zcube, label={north:$1$}] (d) at (120:2) {};
        \node[Zcube, label={south:$1$}] (e) at (-60:2) {};
        \node[Zcube, label={north:$2$}] (h) at (60:2) {};
        \node[Zcube, label={east:$0$}] (i) at (0:2) {};
        \node[Zcube] (c) at (180:0.8) {};
        \node[Zcube] (f) at (-50:0.9) {};
        \node[Zcube] (g) at (50:0.9) {};

        \node[font=\large] at (-90:2.8) {$X$};
        \draw
            (a) -- (b) % horizontal edges
            (a) -- (c)
            (a) -- (d)
            (b) -- (e)
            (b) -- (f)
            (c) -- (f)
            (c) -- (g)
            (d) -- (g)
            (d) -- (h)
            (i) -- (h)
            (i) -- (g)
            (i) -- (f)
            (i) -- (e) 
            (b) -- (c) % vertical edges
            (c) -- (d)
            (e) -- (f)
            (f) -- (g)
            (g) -- (h);
        \begin{scope}[on background layer]
            \path[fill=grayfill] 
                (i.center) -- (h.center) -- (d.center) -- (a.center) -- (b.center) -- (e.center) -- cycle;
        \end{scope}
    \end{tikzpicture} \hspace{5em} \begin{tikzpicture}[
        scale=1.3
    ]
        % \coordinate (O) at (0, 0);
        \node[Zcube, label={west:$0$}] (a) at (180:2) {};
        \node[Zcube, label={south:$2$}] (b) at (-120:2) {};
        \node[Zcube, label={north:$1$}] (d) at (120:2) {};
        \node[Zcube, label={south:$1$}] (e) at (-60:2) {};
        \node[Zcube, label={north:$2$}] (h) at (60:2) {};
        \node[Zcube, label={east:$0$}] (i) at (0:2) {};
        \node[Zcube] (c) at (180:0.8) {};
        \node[Zcube] (f) at (-50:0.9) {};
        \node[Zcube] (g) at (50:0.9) {};
        % Barycenters
        \path
            (a) -- node[Zcube, label={south west:$02$}] (ab) {} (b) % horizontal edges
            (a) -- node[Zcube] (ac) {} (c)
            (a) -- node[Zcube, label={north west:$01$}] (ad) {} (d)
            (b) -- node[Zcube, label={south:$12$}] (be) {} (e)
            (b) -- node[Zcube] (bf) {} (f)
            (c) -- node[Zcube] (cf) {} (f)
            (c) -- node[Zcube] (cg) {} (g)
            (d) -- node[Zcube] (dg) {} (g)
            (d) -- node[Zcube, label={north:$12$}] (dh) {} (h)
            (i) -- node[Zcube, label={north east:$02$}] (ih) {} (h)
            (i) -- node[Zcube] (ig) {} (g)
            (i) -- node[Zcube] (if) {} (f)
            (i) -- node[Zcube, label={south east:$01$}] (ie) {} (e) 
            (b) -- node[Zcube] (bc) {} (c) % vertical edges
            (c) -- node[Zcube] (cd) {} (d)
            (e) -- node[Zcube] (ef) {} (f)
            (f) -- node[Zcube] (fg) {} (g)
            (g) -- node[Zcube] (gh) {} (h);
        \node[Zcube] (abc) at (barycentric cs:ab=1,ac=1,bc=1) {};
        \node[Zcube] (acd) at (barycentric cs:ac=1,ad=1,cd=1) {};
        \node[Zcube] (bfc) at (barycentric cs:b=1,f=1,c=1) {};
        \node[Zcube] (cfg) at (barycentric cs:c=1,f=1,g=1) {};
        \node[Zcube] (cgd) at (barycentric cs:c=1,g=1,d=1) {};
        \node[Zcube] (ebf) at (barycentric cs:be=1,bf=1,ef=1) {};
        \node[Zcube] (gdh) at (barycentric cs:dg=1,dh=1,gh=1) {};
        \node[Zcube] (ief) at (barycentric cs:ie=1,if=1,ef=1) {};
        \node[Zcube] (ifg) at (barycentric cs:if=1,ig=1,fg=1) {};
        \node[Zcube] (igh) at (barycentric cs:ig=1,ih=1,gh=1) {};
        % \node[Zcube, fill=red] (cfgD) at (barycentric cs:cf=1,cg=1,fg=1) {}; % TEST

        \begin{scope}[decoration={markings,mark=at position 0.57 with {\arrow[thick]{<}}}]
            % edges to barycenters 
            \draw[postaction=decorate] (abc) -- (bc);
            \draw[postaction=decorate] (abc) -- (ab);
            \draw[postaction=decorate] (abc) -- (ac);
            \draw[postaction=decorate] (acd) -- (cd);
            \draw[postaction=decorate] (acd) -- (ac);
            \draw[postaction=decorate] (acd) -- (ad);
            \draw[postaction=decorate] (acd) -- (cd);
            \draw[postaction=decorate] (bfc) -- (bf);
            \draw[postaction=decorate] (bfc) -- (bc);
            \draw[postaction=decorate] (bfc) -- (cf);
            \draw[postaction=decorate] (cfg) -- (cf);
            \draw[postaction=decorate] (cfg) -- (cg);
            \draw[postaction=decorate] (cfg) -- (fg);
            \draw[postaction=decorate] (cgd) -- (cg);
            \draw[postaction=decorate] (cgd) -- (cd);
            \draw[postaction=decorate] (cgd) -- (dg);
            \draw[postaction=decorate] (ebf) -- (be);
            \draw[postaction=decorate] (ebf) -- (bf);
            \draw[postaction=decorate] (ebf) -- (ef);
            \draw[postaction=decorate] (gdh) -- (dg);
            \draw[postaction=decorate] (gdh) -- (dh);
            \draw[postaction=decorate] (gdh) -- (gh);
            \draw[postaction=decorate] (ief) -- (ie);
            \draw[postaction=decorate] (ief) -- (if);
            \draw[postaction=decorate] (ief) -- (ef);
            \draw[postaction=decorate] (ifg) -- (if);
            \draw[postaction=decorate] (ifg) -- (ig);
            \draw[postaction=decorate] (ifg) -- (fg);
            \draw[postaction=decorate] (igh) -- (ig);
            \draw[postaction=decorate] (igh) -- (ih);
            \draw[postaction=decorate] (igh) -- (gh);
            % edges
            \draw[postaction=decorate] (ab) -- (b);
            \draw[postaction=decorate] (ac) -- (c);
            \draw[postaction=decorate] (ad) -- (d);
            \draw[postaction=decorate] (be) -- (e);
            \draw[postaction=decorate] (bf) -- (f);
            \draw[postaction=decorate] (cf) -- (f);
            \draw[postaction=decorate] (cg) -- (g);
            \draw[postaction=decorate] (dg) -- (g);
            \draw[postaction=decorate] (dh) -- (h);
            \draw[postaction=decorate] (ih) -- (h);
            \draw[postaction=decorate] (ig) -- (g);
            \draw[postaction=decorate] (if) -- (f);
            \draw[postaction=decorate] (ie) -- (e);
            \draw[postaction=decorate] (bc) -- (c);
            \draw[postaction=decorate] (cd) -- (d);
            \draw[postaction=decorate] (ef) -- (f);
            \draw[postaction=decorate] (fg) -- (g);
            \draw[postaction=decorate] (gh) -- (h);
        \end{scope}
        \begin{scope}[decoration={markings,mark=at position 0.57 with {\arrow[thick]{>}}}]
            % edges
            \draw[postaction=decorate] (a) -- (ab);
            \draw[postaction=decorate] (a) -- (ac); 
            \draw[postaction=decorate] (a) -- (ad); 
            \draw[postaction=decorate] (b) -- (be); 
            \draw[postaction=decorate] (b) -- (bf); 
            \draw[postaction=decorate] (c) -- (cf); 
            \draw[postaction=decorate] (c) -- (cg); 
            \draw[postaction=decorate] (d) -- (dg); 
            \draw[postaction=decorate] (d) -- (dh); 
            \draw[postaction=decorate] (i) -- (ih); 
            \draw[postaction=decorate] (i) -- (ig); 
            \draw[postaction=decorate] (i) -- (if); 
            \draw[postaction=decorate] (i) -- (ie); 
            \draw[postaction=decorate] (b) -- (bc); 
            \draw[postaction=decorate] (c) -- (cd); 
            \draw[postaction=decorate] (e) -- (ef); 
            \draw[postaction=decorate] (f) -- (fg); 
            \draw[postaction=decorate] (g) -- (gh); 
        \end{scope}

        % CX label
        \node[font=\large] at (-90:2.8) {$CX$};
        \begin{scope}[on background layer]
            \path[fill=grayfill] 
                (i.center) -- (h.center) -- (d.center) -- (a.center) -- (b.center) -- (e.center) -- cycle;
        \end{scope}
    \end{tikzpicture} \]
    whose geometric realization is homeomorphic to the real projective plane $\mathbb{RP}^2$.
    This simplicial complex has 6 vertices, 15 edges, and 10 triangles.
    The cubification of this simplicial complex $CX$ is depicted on the right (1-cubes with the same labelled source and target are identical).
    We count the number of vertices in $\Gn[3]{X}$:
    \begin{itemize}
        \item Each 0-simplex of $X$ contributes a 0-cube in $CX$.
        Each 0-cube in $CX$ contributes a vertex in $\Gn[3]{X}$.
        There are 6 vertices in $X$, which contribute a total of 6 vertices in $\Gn[3]{X}$.
        \item Each 1-simplex of $X$ contributes 1 additional 0-cube and 2 additional 1-cubes in $CX$.
        For $m = 3$, each 1-cube in $CX$ contributes 5 vertices in $\Gn[3]{X}$, one for each $F$-sequence in $F(3, 1)$ which does not lie on the boundary $\bd F(3, 1)$.
        There are 15 total 1-simplices in $X$, which contributes a total of $15 + (15*2*5) = 165$ additional vertices in $\Gn[3]{X}$.
        \item Each 2-simplex of $X$ contributes 1 additional 0-cube, 3 additional 1-cubes, and 3 additional 2-cubes in $CX$.
        For $m = 3$, each 2-cube in $CX$ contributes 49 vertices in $\Gn[3]{X}$, one for each $F$-sequence in $F(3, 2)$ which does not lie on the boundary $\bd F(3, 2)$.
        There are 10 total 1-simplices in $X$, which contributes a total of $10 + (10*3*5) + (10*3*49) = 1,630$ additional vertices in $\Gn[3]{X}$.
    \end{itemize}
    In total, the graph $\Gn[3]{CX}$ has $1,801$ vertices.
    Manual inspection shows that each vertex has degree at least 3, hence there are more than $2,700$ edges in $\Gn[3]{CX}$.

\appendix
\renewcommand{\thesection}{\Alph{section}}

\begin{appendices}
  \section{Nonsingular cubical sets} \label{sec:fondind-cubification}

In this appendix, we collect the necessary facts about Reedy theory and combinatorial models of spaces, including simplicial and cubical sets.
Our first main goal (\cref{fondind-presheaf-cat-equiv}) is to identify \emph{fondind} presheaves over a Reedy category $A$ with presheaves over its direct subcategory $A_+$.
We then show (\cref{nonsingular-approximation-cospan}) that every homotopy type can be modelled by a \emph{nonsingular} cubical set, meaning a cubical set with the property that every non-degenerate cube determines a monomorphism $\cube{n} \to X$.
This is done via the \emph{cubification} construction, which is a functor $C \from \sSet \to \cSet$ which recovers the standard subdivision operation after post-composing with the triangulation functor $T \from \cSet \to \sSet$.

The notion of fondind presheaves is due to Yuki Maehara.
We learned of many ideas presented in this appendix from Kensuke Arakawa and Mitsunobu Tsutaya, and suspect that many of them are well-known to experts.

\subsection{Reedy categories}

We begin by recalling the basic definitions of Reedy theory.
\begin{definition}
    A \emph{Reedy category} is a category $A$ with a function $\deg \from \ob A \to \NN \cup \{ 0 \}$ and two subcategories $A_-$ and $A_+$ of $A$ such that:
    \begin{enumerate}
        \item Every identity morphism is contained in both $A_-$ and $A_+$.
        % \item For $r, r' \in R$, if $\deg r = \deg r'$ then $r = r'$.
        \item If a non-identity map $a \to b$ is in $A_-$ then $\deg a > \deg b$; if a non-identity map $a \to b$ is in $A_+$ then $\deg a < \deg b$.
        \item For any morphism $\varphi \in A$, there are unique morphisms $\varphi_- \in A_-$ and $\varphi_+ \in A_+$ such that $\varphi = \varphi_+ \varphi_-$.
    \end{enumerate}
\end{definition}
% \begin{definition}
%     A \emph{Reedy category} is a tuple $(A, A_-, A_+, \deg)$ consisting of:
%     \begin{itemize}
%         \item a small category $A$;
%         \item two subcategories $A_-$ and $A_+$; and
%         \item a \emph{degree} function $\deg \from \ob A \to \NN \cup \{ 0 \}$.
%     \end{itemize}
%     satisfying the following:
%     \begin{enumerate}
%         \item $A_+$ and $A_-$ contain all identity morphisms;
%         \item if $a \to b$ is a morphism in $A_-$ then $\deg a > \deg b$; if $a \to b$ is a morphism in $A_+$ then $\deg a < \deg b$;
%         \item for every morphism $\varphi$ in $A$, there are unique $\varphi_- \in A_-$ and $\varphi_+ \in A_+$ such that $\varphi = \varphi_+ \varphi_-$.
%     \end{enumerate}
% \end{definition}
Note these conditions imply that $A$ has no non-identity isomorphisms.
\begin{example} \label{ex:delta-box-reedy}
    The simplex category $\Delta$ and several variants of the box category $\Box$ (e.g., with or without connections) are Reedy categories (cf.~\cite[Cor.~1.17]{doherty-kapulkin-lindsey-sattler}).
\end{example}
Borrowing language used for the simplex category, we refer to morphisms in $A_-$ as \emph{degeneracy maps} or \emph{degeneracies}; morphisms in $A_+$ are referred to as \emph{face maps} or \emph{faces}.

\begin{example} \label{ex:direct-inverse-reedy}
    A category $I$ is \emph{direct} (respectively, \emph{inverse}) if there exists a function $\deg \colon \ob I \to \nat$ such that, for every non-identity morphism $i \to j$ in $I$, we have $\deg i < \deg j$ (respectively, $\deg i > \deg j$).
    With these definitions, every direct or inverse category is a Reedy category.
    Moreover, for any Reedy category $R$, the subcategory $R_-$ is inverse and the subcategory $R_+$ is direct.
\end{example}

\begin{example} \label{ex:opposite-reedy}
    If $(A, A_-, A_+, \deg)$ is a Reedy category then the opposite category $(A^\op, (A_+)^\op, (A_-)^\op, \deg)$ is again a Reedy category.
\end{example}

% We use the theory of Reedy categories in two separate contexts.
% The first is \emph{covariant} in nature: if $A$ is a Reedy category and $M$ is a model category then the category of functors $A \to M$ carries the \emph{Reedy model structure}, which we will briefly review.
% The second, more substantial, usage of Reedy categories is \emph{contravariant} in nature: we consider the category of presheaves over a Reedy category, i.e.\ functors $A^\op \to \Set$.
% \subsection{The Reedy model structure on diagrams}

\subsection{Presheaves over Reedy categories}

Throughout, let $A$ be a Reedy category.
We write $\aSet$ for the category of presheaves on $A$, i.e.\ functors $A^\op \to \Set$.
We denote objects of $\aSet$ by $X, Y, \dots$.
Given an object $a \in A$, we write $X_a$ for the set $X(a)$.
Given a morphism $\varphi \from a \to b$ in $A$ and $x \in X_b$, we write $x\varphi$ for the application of the function $X(\varphi) \from X_b \to X_a$ to the element $x$. 
For $a \in A$, the associated representable presheaf is denoted $\hat{a}$.  

\begin{definition}
    Let $A$ be a Reedy category and $X \in \aSet$ be a presheaf over $A$.
    Given an element $x \in X_a$,
    \begin{enumerate}
        \item a \emph{face of $x$} is any element of the form $x \delta$ where $\delta$ is a face map.
        \item a \emph{degeneracy of $x$} is any element of the form $x \sigma$ where $\sigma$ is a degeneracy map.
        \item An element $x \in X_a$ is \emph{degenerate} if there exists a degeneracy map $\sigma \from a \to b$ and an element $y \in X_b$ such that $\sigma \neq \id[a]$ and $x = y\sigma$.
        We say $x$ is \emph{non-degenerate} if it is not degenerate.
        % \item We say $K$ is \emph{fondind} if 
    \end{enumerate}
\end{definition}
\begin{definition} \label{def:EZ-lemma-elegant}
    We say a presheaf $X \in \aSet$ \emph{satisfies the Eilenberg-Zilber lemma} if, for any element $x \in X_a$, there exists a unique degeneracy $\sigma \from a \to b$ and a unique non-degenerate $y \in X_b$ such that $x = y\sigma$.
    % \begin{enumerate}
    %     \item We say a presheaf $X \in \aSet$ satisfies the \emph{Eilenberg-Zilber lemma} if, for any element $x \in X_a$, there exists a unique degeneracy $\sigma \from a \to b$ and a unique non-degenerate $y \in X_b$ such that $x = y\sigma$.
    %     \item A Reedy category $A$ is \emph{elegant} if every presheaf over $A$ satisfies the Eilenberg-Zilber lemma.
    % \end{enumerate}
\end{definition}
\begin{example}
    In any Reedy category, every representable presheaf satisfies the Eilenberg-Zilber lemma.
    This follows by uniqueness of factorizations.
\end{example}
\begin{example}
    If $A$ is a direct category then every presheaf (trivially) satisfies the Eilenberg-Zilber lemma.
\end{example}
\begin{remark}
    In general, every element $x$ in a presheaf $X$ can be written as $x = y\sigma$ for some degeneracy $\sigma$ and non-degenerate $y$ (see the proof of \cref{i-shriek-N-epi-mono}).
    The content of \cref{def:EZ-lemma-elegant} is the \emph{uniqueness} of such a factorization, which fails in general.
    For a counterexample, consider the category $A$ given by two parallel arrows $1 \rightrightarrows 0$, both of which are in $A_-$.
    Then, the presheaf
    \[ X_0 = \{ a, b \} \rightrightarrows \{ \ast \} = X_1 \]
    does not satisfy the Eilenberg-Zilber lemma.
\end{remark}

\begin{definition} \label{def:fondind}
    A presheaf $X \in \aSet$ is \emph{fondind} if, for any non-degenerate element $x \in X_a$ and morphism $\varphi \from b \to a$ in $A_+$, the element $x\varphi \in X_b$ is again non-degenerate.
    % Let $X \in \aSet$ be a presheaf over $A$.
    % \begin{enumerate}
    %     \item A presheaf $X \in \aSet$ is \emph{nonsingular} if, for any non-degenerate element $x \in X_a$, the corresponding map $x \from \hat{a} \to K$ is a monomorphism.
    %     \item A presheaf $X \in \aSet$ is \emph{fondind} if, for any non-degenerate element $x \in X_a$ and morphism $\varphi \from b \to a$ in $A_+$, the element $x\varphi \in X_b$ is again non-degenerate.
    % \end{enumerate}
\end{definition}
The term \emph{fondind} is an abbreviation: it stands for \textbf{F}ace of \textbf{O}f \textbf{N}on-\textbf{D}egenerate \textbf{I}s \textbf{N}on-\textbf{D}egenerate.
% \begin{proposition}
%     Every nonsingular presheaf is fondind.
% \end{proposition}
% \begin{proof}
%     \fxnote{Need to know that maps in $A^+$ are mono, and a mono in $A_-$ must be identity.}
%     Suppose $X \in \aSet$ is nonsingular.
%     Let $x \in X_a$ be a non-degenerate element and $\varphi \from b \to a$ be a map in $A_+$.
%     We write $x\varphi = y\sigma$ for some degeneracy $\sigma \from b \to c$ and non-degenerate $y \in K_c$.
%     Since both morphisms $x \from \hat{a} \to K$ and $\hat{\varphi} \from \hat{b} \to \hat{a}$ are monomorphisms, the composite $x\hat{\varphi} = y \hat{\sigma}$ is a monomorphism, from which we deduce that $\hat{\sigma} \from b \to c$ is a monomorphism.
%     Therefore, $\sigma = \id[b]$, which proves that $x\varphi$ is non-degenerate.
% \end{proof}

The subcategory $A_+ \subseteq A$ is again a Reedy category (because it is a direct category).
We write $\aSeti$ for the category of presheaves over $A_+$.
% \begin{proposition}    
% \end{proposition}
The inclusion functor $i \from A_+ \ito A$ induces an adjoint triple:
\[ \begin{tikzcd}
    \aSeti \ar[r, bend left=50, "{i_!}"{name=T}] \ar[r, bend right=50, "{i_*}"'{name=B}] &[+1.8em] \aSet \ar[l, "{i^*}"{name=M, description}]
    \arrow[from=M, to=T, phantom, "\perp"]
    \arrow[from=B, to=M, phantom, "\perp"]
\end{tikzcd} \]
We use the Reedy structure on $A$ to give a description of the functor $i_!$.
\begin{proposition} \label{i-shriek-explicit-description}
    Given a presheaf $X \in \aSeti$, the presheaf $i_!(X)$ has the following explicit description: 
    \begin{itemize}
        \item at an object $a \in A$, the set $i_!(X)_a$ is given by
        \[ i_!(X)_a := \{ (\sigma, x) \mid \sigma \from a \to b \text{ in } A_- \text{ and } x \in X_b \}, \]
        where $\sigma \from a \to b$ is any degeneracy with domain $a$ and $x$ is an element of $X_b$.
        \item at a morphism $\varphi \from c \to a$ in $A$, the function $i_!(X)(\varphi)$ is given by the formula $(\sigma, x)\varphi := \big( (\sigma \varphi)_-, x (\sigma \varphi)_+ \big)$.
    \end{itemize}
\end{proposition}
\begin{proof}
    The colimit formula for the left Kan extension $i_!(X)$ is given by 
    \[ i_!(X)_a := \colim\limits_{\varphi \in (a \slice i)^\op} K_{\cod(\varphi)}, \]
    where $a \slice i$ denotes the comma category whose objects are arrows $a \to b$ in $A$ and whose morphisms are arrows $b \to c$ in $A_+$ that make the induced triangle commute.
    By the uniqueness of factorizations in a Reedy category, this comma category admits an initial subcategory which is the full subcategory spanned by objects $a \to b$ which are in $A_-$.
    Moreover, this full subcategory is a discrete category, again by uniqueness of factorizations.
    Thus, this colimit simplifies to a coproduct, whose explicit description recovers the theorem statement.
\end{proof}
We will show that the functor $i_!$ identifies $\aSeti$ with the subcategory of $\aSet$ spanned by fondind presheaves and maps that preserve non-degenerate elements.

We say a morphism $f \from K \to L$ in $\aSet$ is \emph{locally non-degenerate} if, for all objects $a \in A$, the function $f_a \from X_a \to L_a$ sends non-degenerate elements to non-degenerate elements.
Let $\aSet_{\mathrm{nd}} \subseteq \aSet$ denote the subcategory consisting of fondind presheaves and locally non-degenerate morphisms.
Define a functor $N \from \aSet_{\mathrm{nd}} \to \aSeti$ as follows:
\begin{itemize}
    \item for an object $X \in \aSet$, the presheaf $N(K) \from A_+^\op \to \Set$ is defined as follows: given an object $a \in A$, we define $N(K)_a$ to be the set $(X_a)_{\mathrm{nd}}$ of non-degenerate elements of $X_a$; a morphism $\varphi \from b \to a$ in $A_+$ is sent to the restricted function $\restr{K(\varphi)}{(X_a)_{\mathrm{nd}}}$. 
    This is well-defined because $K$ is fondind.
    \item for a morphism $f \from K \to L$, the natural transformation $N(f)$ is defined on an object $a \in A_+$ as the restricted function $\restr{f_a}{(X_a)_{\mathrm{nd}}}$.
    This is well-defined because $f$ is locally non-degenerate.
\end{itemize}

Let $X \in \aSeti$ be a presheaf over $A_+$.
From the explicit description of $i_!(X)$ given in \cref{i-shriek-explicit-description}, we see that an element $(\sigma, x) \in i_!(X)_a$ is non-degenerate if and only if $\sigma = \id[a]$.
This implies $i_!$ takes values in the subcategory $\aSet_{\mathrm{nd}}$.
Moreover, there is an evident natural isomorphism $X \xrightarrow{\cong} N(i_!(X))$ which sends an element $x \in X_a$ to $(\id[a], x) \in N(i_!(X))$. 

For a presheaf $X \in \aSet$ over $A$, there is a natural transformation $i_!(N(K)) \to K$ which sends an element $(\sigma, x) \in i_!(N(K))_a$ to the element $x\sigma \in X_a$.
\begin{proposition} \label{i-shriek-N-epi-mono}
    For a presheaf $X \in \aSet$, the natural map $i_!(N(K)) \to K$ is an epimorphism.
    Moreover, it is a monomorphism (hence an isomorphism) if and only if $K$ satisfies the Eilenberg-Zilber lemma.
\end{proposition}
\begin{proof}
    The second statement follows by definition.
    For the first statement, fix an element $x \in X_a$.
    Let $n \in \NN$ be minimal such that there exists $\sigma \from a \to b$ in $A_-$ and $y \in X_b$ satisfying $x = y\sigma$ and $\deg(b) = n$.
    By minimality of $n$, the element $y$ is non-degenerate, hence $(\sigma, y)$ is an element of $i_!(N(K))_a$ in the pre-image of $x$.
\end{proof}
\begin{theorem}
    The natural transformations $\id \Rightarrow N \circ i_!$ and $i_! \circ N \Rightarrow \id$ are the unit and counit of an adjunction
    \[ \begin{tikzcd}
        \aSeti \ar[r, bend left=20, "{i_!}"{name=T}] \ar[r, bend right=20, "{N}"'{name=B}, leftarrow] &[+1.8em] \aSet_{\mathrm{nd}}
        \arrow[from=B, to=T, phantom, "\perp"]
    \end{tikzcd} \]
    Moreover, the left adjoint $i_!$ is full and faithful, and its essential image consists of those presheaves $X \in \aSet$ which satisfy the Eilenberg-Zilber lemma.
\end{theorem}
\begin{proof}
    For the adjunction, we need only observe that the maps $K \to N(i_!(X))$ and $i_!(N(K)) \to K$ satisfy the triangle identities.
    Since the unit is an isomorphism, the functor $i_!$ is full and faithful.
    The characterization of the essential image follows from \cref{i-shriek-N-epi-mono}.
\end{proof}
\begin{corollary} \label{fondind-presheaf-cat-equiv}
    The functors $i_!$ and $N$ induce an equivalence of categories between $\aSeti$ and the subcategory of $\aSet$ whose objects are fondind presheaves satisfying the Eilenberg-Zilber lemma and whose morphisms are maps sending non-degenerate elements to non-degenerate elements. \qed
\end{corollary}
The notion of fondind presheaves is closely related to the notion of \emph{nonsingular} presheaves.
\begin{definition}
    A presheaf $X \in \aSet$ is \emph{nonsingular} if, for any non-degenerate element $x \in X_a$, the corresponding map $x \from \hat{a} \to X$ is a monomorphism.
 \end{definition}
 \begin{proposition}
     Suppose $A$ is a Reedy category in which:
     \begin{itemize}
         \item every morphism in $A_-$ is a split epimorphism, and
         \item every morphism in $A_+$ is a monomorphism.
     \end{itemize}
     Then, every nonsingular presheaf over $A$ is fondind.
 \end{proposition}
 \begin{proof}
     % \fxnote{Need to know that maps in $A^+$ are mono, and a mono in $A_-$ must be identity.}
     Let $X \in \aSet$ be a nonsingular presheaf.
     Fix a non-degenerate element $x \in X_a$ and a face map $\delta \from b \to a$.
     Suppose $x\delta = y\sigma$ for some degeneracy $\sigma \from b \to c$ and non-degenerate $y \in K_c$.
     Since both morphisms $x \from \hat{a} \to K$ and $\hat{\delta} \from \hat{b} \to \hat{a}$ are monomorphisms, the composite $x\hat{\delta} = y \hat{\sigma}$ is a monomorphism, from which we deduce that $\hat{\sigma} \from b \to c$ is a monomorphism.
     Since $\hat{\sigma}$ is also a split epimorphism, we must have $\sigma = \id[b]$. 
 \end{proof}

% \section{Nonsingular cubical sets}

\subsection{Triangulation and cubification}

We now turn our attention to cubical sets; our goal is to show that every cubical set is weakly equivalent to a nonsingular cubical set (\cref{nonsingular-approximation-cospan}).

Recall that cubical sets and simplicial sets are related via the \emph{triangulation} adjunction.
More specifically, we have a cocubical object $\boxcat \to \sSet$ in simplicial sets which sends the object $[1]^n$ to the simplicial $n$-cube $(\simp{1})^n$.
The triangulation functor $T \from \cSet \to \sSet$ is defined as the extension by colimits of this cocubical object.
\[ \begin{tikzcd}
    \boxcat \ar[r, "{[1]^n \mapsto (\simp{1})^n}"] \ar[d, "\yo"'] &[+2ex] \sSet \\
    \cSet \ar[ur, dotted, bend right, "T"']
\end{tikzcd} \]
This functor admits a right adjoint, denoted $U \from \sSet \to \cSet$, whose set of $n$-cubes is given by
\[ (UK)_n := \sSet \big( (\simp{1})^n, X \big). \]
The adjunction $T \adj U$ forms a Quillen equivalence between the Grothendieck model structure on cubical sets and the Kan--Quillen model structure on simplicial sets \cite[Prop.~8.4.30]{cisinski:presheaves}.

We now define a seperate functor $\sSet \to \cSet$, the \emph{cubification} functor.
Define a co-semisimplicial object $C \from \Delta \to \cSet$ which sends the $n$-simplex to the lower half $(n+1)$-cube
\[ C[n] := \bigcup\limits_{k = 1}^{n+1} \face{}{k, 1}, \]
i.e.\ the union of the $\eps = 1$ faces of the $(n+1)$-cube.

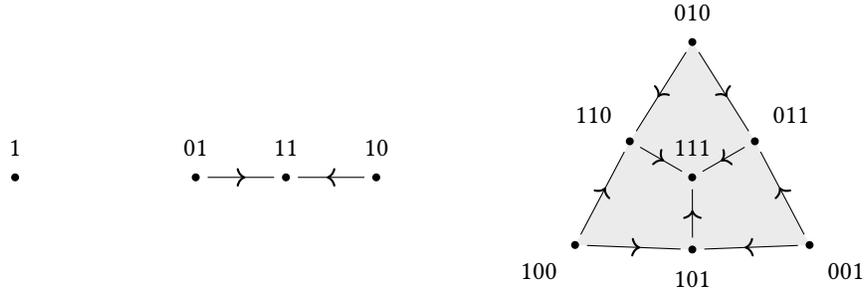
\begin{figure}[H]
    \centering
    \begin{tikzpicture}[scale=1.2, decoration={markings,mark=at position 0.57 with {\arrow[thick]{>}}}]
        \node[Zcube, label={$1$}] (A0) {};
        
        \path (A0) -- +(3, 0) node[Zcube, label={$11$}] (B11) {};
        % \node[vertex, label={$11$}] (B00) at (Bcenter) {};
        \path (B11) -- +(-1, 0) node[Zcube, label={$01$}] (B01) {};
        \path (B11) -- +(1, 0) node[Zcube, label={$10$}] (B10) {};
        \draw[postaction={decorate}] (B01) -- (B11);
        \draw[postaction={decorate}] (B10) -- (B11);

        \path (B11) -- +(4.5, 0) node[Zcube, label={$111$}] (C111) {};
        % \node[vertex, label={$111$}] (C111) at (Ccenter) {};
        \path (C111) -- +(90:1.5) node[Zcube, label={$010$}] (C010) {};
        \path (C111) -- +(210:1.5) node[Zcube, label={south west:$100$}] (C100) {};
        \path (C111) -- +(-30:1.5) node[Zcube, label={south east:$001$}] (C001) {};
        \path (C111) -- +(150:0.8) node[Zcube, label={north west:$110$}] (C110) {};
        \path (C111) -- +(-90:0.8) node[Zcube, label={south:$101$}] (C101) {};
        \path (C111) -- +(30:0.8) node[Zcube, label={north east:$011$}] (C011) {};
        \draw[postaction={decorate}] (C100) -- (C110);
        \draw[postaction={decorate}] (C100) -- (C101);
        \draw[postaction={decorate}] (C010) -- (C110);
        \draw[postaction={decorate}] (C010) -- (C011);
        \draw[postaction={decorate}] (C001) -- (C011);
        \draw[postaction={decorate}] (C001) -- (C101);
        \draw[postaction={decorate}] (C110) -- (C111);
        \draw[postaction={decorate}] (C101) -- (C111);
        \draw[postaction={decorate}] (C011) -- (C111);

        % Fill the squares
        \begin{scope}[on background layer]
            \path[fill=grayfill] (C100.center) -- (C110.center) -- (C111.center) -- (C101.center) -- cycle;
            \path[fill=grayfill] (C010.center) -- (C110.center) -- (C111.center) -- (C011.center) -- cycle;
            \path[fill=grayfill] (C001.center) -- (C101.center) -- (C111.center) -- (C011.center) -- cycle;
        \end{scope}
    \end{tikzpicture}
    \caption{Depiction of the cubical sets $C[0]$, $C[1]$, and $C[2]$.}
\end{figure}

To define the action of $C$ on morphisms of $\Delta$, let ${d_i}$ denote the top composite arrow in
\[ \begin{tikzcd}
    \{ 1, \dots, n \} \ar[d, "-1"', "\cong"] \ar[r, dotted, "{d_i}"] & \{ 1, \dots, n+1 \} \\
    {[n-1]} \ar[r, "\face{}{i}"] & {[n]} \ar[u, "+1"', "\cong"]
\end{tikzcd} \]
Explicitly, given $i \in \{ 0, \dots, n \}$, the set function $d_i$ ``skips'' the element $i+1$ in $\{ 1, \dots, n+1 \}$.

For $i \in \{ 0, \dots, n \}$, we define the simplicial face morphism $C\face{}{i} \from C[n-1] \to C[n]$ separately on each face of $C[n-1]$; given $k \in \{ 1, \dots, n \}$, the restriction $\restr{C\face{}{i}}{\face{}{k, 1}}$ is defined to be the cubical face inclusion $\face{}{{d}_i(k), 0} \from \cube{n} \to C[n+1]$.
For $i \in \{ 0, \dots, n \}$, the simplicial degeneracy morphism $C\degen{}{i} \from C[n+1] \to C[n]$ is the restriction of the cubical connection map
\[ \restr{\conn{}{i+1, 0}}{C[n+1]} \from C[n+1] \to C[n]. \]
\begin{proposition} \label{C-simp-ids}
    The morphisms $C\face{}{i}$ and $C\degen{}{i}$ satisfy the simplicial identities.
\end{proposition}
\begin{proof}
    Since the triangulation functor $T \from \cSet \to \sSet$ is faithful, it suffices to verify the simplicial identities after applying $T$.

    The triangulation $TC[n]$ of the lower half cube is the nerve of the partially ordered subset $P^n \subseteq [1]^{n+1}$ which omits the zero vertex $(0, \dots, 0)$; that is,
    \[  TC[n] = NP^n, \text{ where } P^n = [1]^{n+1} - \big\{ (0, \dots, 0) \}. \]
    The simplicial set $NP^n$ is isomorphic to the subdivided $n$-simplex $\sd [n]$.
    One verifies that $C\face{}{i}$ and $C\degen{}{i}$ are mapped to the usual face and degeneracy functions $\face{}{i} \from P^{n-1} \to P^n$ and $\degen{}{i} \from P^{n+1} \to P^n$, whose formulas are given by
    \[ \begin{array}{l l}
        % \face{}{i} \from P^{n-1} \to P^{n} & \degen{}{i} \from P^{n+1} \to P^{n} \\
        \face{}{i}(x_0, \dots, x_{n-1}) & \degen{}{i}(x_0, \dots, x_{n+1}) \\
        \ = (x_0, \dots, x_{i-1}, 0, x_{i}, \dots, x_{n-1}) & \ = (x_0, \dots, x_{i-1}, \max(x_i, x_{i+1}), x_{i+2}, \dots, x_{n+1}).
    \end{array} \]
    % These maps are in the image of the triangulation functor: they arise from the cubical maps $\face{}{i} \from C[n-1] \to C[n]$ and $\degen{}{i} \from C[n+1] \to C[n]$ given by
\end{proof}

We extend by colimits to define the \emph{cubification} functor $C \from \sSet \to \cSet$. 
\begin{proposition} \label{C-recovers-Sd}
    The triangle of functors
    \[ \begin{tikzcd}
        {} & \cSet \ar[rd, "T"] & {} \\
        \sSet \ar[ur, "C"] \ar[rr, "\Sd"] & {} & \sSet
    \end{tikzcd} \]
    commutes up to natural isomorphism.
\end{proposition}
\begin{proof}
    As a triangle of left adjoints, it suffices to construct the isomorphism on representables $\simp{n}$, where this is immediate (see the proof of \cref{C-simp-ids}).
\end{proof}
\begin{corollary} \label{C-Quillen-equiv}
    The functor $C \from \sSet \to \cSet$ is a left Quillen equivalence between the Kan-Quillen model structure on simplicial sets and the Grothendieck model structure on cubical sets.
\end{corollary}
\begin{proof}
    To see that $C$ is a left Quillen functor, observe that triangulation reflects both monomorphisms and weak equivalences.
    As a result, it suffices (by \cref{C-recovers-Sd}) to observe that $\Sd$ preserves cofibrations and acyclic cofibrations.

    To see that $C$ is an equivalence, use 2-out-of-3 on the triangle appearing in \cref{C-recovers-Sd}.
\end{proof}

Recall that, for any simplicial set $X$, there is a natural homotopy equivalence $\Sd X \to X$.
Rewriting this map as $TC(X) \to X$, we obtain the transposed natural transformation $C(X) \to U(X)$.
\begin{proposition} \label{C-U-weq}
    If $X \in \sSet$ is a Kan complex then the natural map $CX \to UX$ is a weak equivalence.
\end{proposition}
\begin{proof}
    Since $TC(X) \to X$ is a weak equivalence, this follows from \cite[Cor.~1.4.4.(b)]{hovey}.
\end{proof}

% Some example computations show that:
% \[ C(\bd \simp{n}) \cong C[n] - \{ (1, \dots, 1) \}

% \subsection{Nonsingular cubical sets}

% The cubification functor has the property that it sends nonsingular simplicial sets to nonsingular cubical sets.
% \begin{definition}
%     A cubical set $K$ (or simplicial set $K$) is \emph{nonsingular} if, for any non-degenerate cube $u$ (respectively, any non-degenerate simplex $u$), the corresponding map $\cube{n} \to K$ (respectively, $\simp{n} \to K$) is a monomorphism.
% \end{definition}
% \begin{proposition} \label{nonsingular-fondind}
%     If $K$ is a nonsingular cubical set (or simplicial set) then every face of a non-degenerate cube (respectively, non-degenerate simplex) is non-degenerate. 
% \end{proposition}
% \begin{proof}
%     If $u \from \cube{n} \to K$ is non-degenerate and $u\face{}{i, \eps}$ is a face of $K$ then $u\face{}{i, \eps} \from \cube{n-1} \to K$ is a monomorphism, from which it follows that $u$ must be non-degenerate.
% \end{proof}
\begin{proposition} \label{C-preserves-nonsingular}
    The cubification functor $C \from \sSet \to \cSet$ sends nonsingular simplicial sets to nonsingular cubical sets.
\end{proposition}
\begin{proof}
    Fix a simplicial set $K$.
    An $n$-cube $u \from \cube{n} \to CK$ is determined (not necessarily uniquely) by the data of a simplex $u_1 \from \simp{k} \to K$ of some dimension and an $n$-cube $u_2 \from \cube{n} \to C[k]$; from this data, the map $u$ is the composite $u = C(u_1) \circ u_2$.
    The $n$-cube $u$ is non-degenerate if and only if $u_1$ is a non-degenerate simplex and $u_2$ is a non-degenerate cube.
    In this case, the map $C(u_1)$ is a monomorphism because $C$ preserves monomorphisms and $K$ is nonsingular by assumption.
    The map $u_2$ is a monomorphism because the cubical set $C[k]$ is nonsingular (it is a subobject of the $(k+1)$-cube $\cube{k+1}$).
    % Since the restriction to the boundary is given by
    % \[ \bd \cube{n} \ito \cube{n} \xrightarrow{u_2} C[k] \xrightarrow{Cu_2} CK , \]
    % it suffices to show that if $u_1$ and $u_2$ are non-degenerate then $u_2$ and $C(u_2)$ are monomorphisms.
    % For $C(u_2)$, this follows because $C$ reflects monomorphisms and $K$ is assumed to be nonsingular.
\end{proof}

We can use the cubification functor to replace a cubical set (up to weak equivalence) by a nonsingular one.
% Recall that every simplicial Kan complex 
\begin{theorem} \label{nonsingular-approximation-cospan}
    There exist functors $\Phi, \overline{\Phi} \from \cSet \to \cSet$ together with natural weak equivalences
    \[ \id \weto \overline{\Phi} \weot \Phi \]
    such that, for any cubical set $X$,
    \begin{enumerate}
        \item $\Phi X$ is nonsingular;
        \item $\overline{\Phi} X$ is a Kan complex; and
        \item for any cubical set $X$, the map $X \to \overline{\Phi} X$ is an acyclic cofibration.
    \end{enumerate}
\end{theorem}
\begin{proof}
    Let $\Tfib \from \cSet \to \sSet$ denote the functor given by applying triangulation followed by some functorial fibrant replacement (e.g. the $\Ex^\infty$ functor).
    % We write $\eta$ for the derived unit $\id \to U \Tfib$ of the triangulation adjunction.
    Define $\overline{\Phi}$ to be $U \Tfib$.
    The derived unit $\id \to U \Tfib$ is both a monomorphism and a weak equivalence because $T \adj U$ is a Quillen equivalence.
    Additionally, the functor $U$ preserves fibrant objects, proving (2) and (3).

    Recall (e.g. from \cite[Variant 4.2.3.15]{lurie:htt}) that there exists a functor $\Psi \from \sSet \to \Pos$ from simplicial sets to posets and a natural weak equivalence $\varphi \from N \Psi \weto \id$.
    We define $\Phi$ to be the composite $C N \Psi \Tfib$, with the natural transformation $\Phi \to \overline{\Phi}$ being given by the composite natural transformation
    \[ C N \Psi \Tfib \xrightarrow{C \varphi_{\Tfib}} C \Tfib \to U \Tfib . \]
    The left map is a weak equivalence since $C$ preserves weak equivalences.
    The right map is a weak equivalence by \cref{C-U-weq}, since $\Tfib$ takes values in Kan complexes by construction.
\end{proof}
% Also recall from 
% \begin{corollary}
%     Every cubical set $X$ admits a cospan of weak equivalences
%     \[ X \weto X_1 \weot X_2 \]
%     where $X_2$ is a nonsingular cubical set.
%     Moreover, the assignments $X \mapsto X_1$ and $X \mapsto X_2$ are functorial in $X$, and each weak equivalence is natural in the variable $X$.
% \end{corollary}
% \begin{proof}
% \end{proof}
\end{appendices}

 \bibliographystyle{amsalphaurlmod}
 \bibliography{all-refs.bib}

 \newpage

\end{document}